\documentclass[12pt,letterpaper]{article}

\usepackage{graphicx}
\usepackage{amsmath}
\usepackage{amssymb}
\usepackage{amsfonts}
\usepackage{amsthm}
\usepackage{cancel}
\usepackage[all]{xy}
\usepackage[neveradjust]{paralist}
\usepackage{subcaption}
\usepackage{enumitem}
\usepackage{tikz}
\usepackage{mathtools}
\usepackage{blkarray}
\usepackage{tabularx,ragged2e,booktabs,caption}
\usepackage{color}
\usepackage{blkarray}
\usepackage{verbatim}
\usepackage{rotating}

\numberwithin{figure}{section}

\usetikzlibrary{shapes,snakes}

\newcommand{\Z}{\mathbb{Z}}
\newcommand{\Q}{\mathfrak{Q}}
\newcommand{\R}{\mathbb{R}}
\newcommand{\C}{\mathbb{C}}

\renewcommand{\P}{\mathfrak{P}}

\newcommand{\Rcal}{\mathfrak{R}}

\newcommand{\A}{{\mathcal A}}
\newcommand{\K}{{\mathbb K}}

\renewcommand{\a}{\mathbf{a}}

\newcommand{\bx}{\mathbf{x}}

\newcommand{\bw}{{\mathbf w}}

\newcommand{\bb}{\mathbf{b}}
\newcommand{\bv}{\mathbf{v}}
\newcommand{\bc}{\mathbf{c}}
\newcommand{\ur}{\underline{r}}
\newcommand{\ux}{\underline{x}}
\newcommand{\uy}{\underline{y}}
\newcommand{\uz}{\underline{z}}

\newcommand{\Acal}{\mathcal{A}}

\newcommand{\Ecal}{{\mathcal E}}
\newcommand{\Fcal}{{\mathcal F}}

\newcommand{\G}{{\mathcal G}}

\newcommand{\B}{{\mathcal B}}
\newcommand{\by}{{\mathbf y}}
\newcommand{\bu}{{\mathbf u}}
\newcommand{\be}{{\mathbf e}}
\newcommand{\bz}{{\mathbf z}}
\newcommand{\Quad}{\mathrm{Quad}}
\newcommand{\Lift}{\mathrm{Lift}}
\newcommand{\Root}{\mathrm{Root}}
\newcommand{\Swap}{\mathrm{Swap}}

\newcommand{\zeven}{\Z_2^{n,\mathrm{even}}}
\newcommand{\Int}{\mathrm{Int}}

\DeclareSymbolFont{bbold}{U}{bbold}{m}{n}
\DeclareSymbolFontAlphabet{\mathbbold}{bbold}
\newcommand{\bbone}{\mathbbold{1}}

\renewcommand{\bar}{\overline}

\newtheorem{thm}{Theorem}
\newtheorem*{thm*}{Theorem}

\newtheorem{lemma}[thm]{Lemma}
\newtheorem{prop}[thm]{Proposition}
\newtheorem{cor}[thm]{Corollary}
\theoremstyle{definition}
\newtheorem{definition}[thm]{Definition}
\newtheorem{example}[thm]{Example}
\newtheorem{remark}[thm]{Remark}
\numberwithin{thm}{section}

\title{Toric Geometry of the Cavender-Farris-Neyman Model with a Molecular Clock}
\author{Jane Ivy Coons and Seth Sullivant}
\date\today

\begin{document}
\maketitle

\begin{abstract}
We give a combinatorial description of the toric ideal of invariants of the 
Cavender-Farris-Neyman model with a molecular clock (CFN-MC) on a rooted binary phylogenetic 
tree and prove results about the polytope associated to this toric ideal.  
Key results about the polyhedral structure include that 
the number of vertices of this polytope is a Fibonacci number, 
the facets of the polytope can be described using the combinatorial 
``cluster" structure  of the underlying rooted tree, and the volume is
equal to an Euler zig-zag number. The toric ideal of invariants of the CFN-MC
model has a quadratic Gr\"obner basis with squarefree initial terms. Finally, 
we show that the Ehrhart polynomial of these polytopes, and therefore the Hilbert series of the ideals, 
depends only on the number of leaves of the underlying binary tree, and not on the topology of the tree itself.
These results are analogous to classic results for the Cavender-Farris-Neyman model
without a molecular clock.  However, new techniques are required because the molecular
clock assumption destroys the toric fiber product structure that governs group-based models 
without the molecular clock.
\end{abstract}

\section{Introduction}
The field of phylogenetics is concerned with reconstructing evolutionary histories
of different species or other taxonomic units (\emph{taxa} for short), such as genes or bacterial strains
(see \cite{felsenstein2004,semple2003} for general background on mathematical phylogenetics). 
In phylogenetics, we use trees to 
model these evolutionary histories. The leaves of these trees represent the 
extant taxa of interest, while the internal nodes represent their extinct common ancestors. 
Branching within the tree represents speciation events, wherein two 
species diverged from a single common ancestor. One may use combinatorial 
trees to depict only the evolutionary relationships between the organisms, 
or include branch lengths to represent time or amount of genetic mutation.

An organism's DNA is made up of chemical compounds called 
\textit{nucleotides}, or \textit{bases}. There are four different 
types of bases: adenine, thymine, guanine and cytosine, abbreviated A,T,G and C. 
These bases are split into two different types based upon their chemical structure. 
Adenine and guanine are  \textit{purines}, and thymine and cytosine are  \textit{pyrimidines}.
Evolution occurs via a series of \emph{substitutions}, or \emph{mutations}, 
within the DNA of an organism, wherein one nucleotide gets 
swapped out for another. Phylogenetic models often assume that base 
substitutions occur as a continuous-time Markov process along the edges of a tree, 
where the edge lengths are the time parameters in the Markov process 
\cite{jukes1969,kimura1980}. The entries of the transition matrices 
in this Markov process are the probabilities of observing a substitution 
from one base to another at a site in the genome at the end of the given time interval.

The two possible states of the Cavender-Farris-Neyman, or CFN, model
are purine and pyrimidine.  This is based on the observed fact that
within group substitutions (purine-purine or pyrimidine-pyrimidine) are much more
common, so a two state model like the CFN model only focuses on cross group
substitutions (which are known as \emph{transversions}). 
In the CFN model we further assume that the rate of 
substitution from purine to pyrimidine is equal to the rate of substitution from 
pyrimidine to purine. The algebraic and combinatorial structure of the CFN-model 
has been studied by a number of authors and key results include
descriptions of the generating set of the vanishing ideal \cite{sturmfels2005},
Gr\"obner bases \cite{sturmfels2005}, polyhedral geometry \cite{buczynska2007},
t Hilbert series \cite{buczynska2007}, and connections to the Hilbert scheme and
toric degenerations \cite{sturmfels2008}.

In this paper, we study the CFN model with an added \emph{molecular clock} condition, 
or the CFN-MC model. The molecular clock condition adds the requirement 
that the time elapsed from the root of the tree to any leaf of the tree is the same for all leaves.
For a fixed tree, the set of all probability distributions in the CFN-MC model along that tree is a
semialgebraic set in the probability simplex $\Delta_{2^n-1}$, 
where $n$ is the number of taxa. 
We describe the homogeneous vanishing
ideal of the Zariski closure of the CFN-MC model in projective space after an appropriate change of coordinates.
To obtain the CFN-MC model itself from the projective variety defined by this ideal,
one must intersect this variety with the simplex
in the positive orthant on which the sum of the coordinates is one,
as well as with other semialgebraic constraints that are not addressed in this paper.

The CFN-MC model belongs to a special class of phylogenetic
models called \textit{group-based} models. This means that under a linear change of coordinates
called the discrete Fourier transform, we can view the Zariski closure 
of the CFN-MC model as a toric variety \cite{evans1993,hendy1993}. 
This allows us to study it from the point of view of polyhedral geometry.
Our main results are summarized by the following:

\begin{thm*}
Let $T$ be a binary rooted tree with $n$ leaves. Let $I_T$ be the toric vanishing ideal of the CFN-MC model
on the tree $T$ as defined in Definition \ref{def:CFNMCideal},
and let $R_T$ be the associated polytope.
\begin{enumerate}
\item  The toric vanishing ideal $I_T$ has a generating set of quadratic
binomials.  These binomials form a Gr\"obner basis that has squarefree
initial terms.  (Theorem \ref{thm:generatorsmain})
\item  The polytope $R_T$ has $F_n$ vertices where $F_n$ denotes the 
 $n$th  Fibonacci number.  (Proposition \ref{thm:fibonacci})
 \item The polytope $R_T$ has an explicit facet description
 that can be understood in terms of the combinatorial structure of the tree (Corollary \ref{cor:finalfacets} )
\item  The normalized volume of $R_T$ is $E_{n-1}$, where $E_{n-1}$ is the Euler
zig-zag number. (Theorem \ref{thm:volumemain})
\item  The Hilbert series of $I_T$ only depends on $n$, not the specific  tree
$T$.  (Theorem \ref{thm:hilbertseries})  
\end{enumerate}
\end{thm*}

These results are analogues of the major results in \cite{buczynska2007, sturmfels2005,  sturmfels2008}
for the CFN model, but the molecular clock assumption presents new challenges.

The outline of the paper is as follows.
In Section 2, we give preliminary definitions regarding phylogenetic trees.
In Section 3, we describe the CFN-MC model in detail.
We describe how the discrete Fourier transform allows us to view the CFN-MC model as a toric variety. 
In Section 4, we study the combinatorics of the polytope associated to the 
toric ideal of phylogenetic invariants of the CFN-MC model. 
In particular, we give vertex and facet descriptions for this polytope, 
and show that the number of vertices of the CFN-MC polytope is equal to a Fibonacci number.
In Section 5, we study the generators of the 
toric ideal of phylogenetic invariants of the CFN-MC model and 
give a quadratic Gr\"obner basis for this ideal. To accomplish this, 
we make use of the theory of toric fiber products \cite{sullivant2007}, as well as new tools to
handle the case of cluster trees where the toric fiber product structure is not present.
In Section 6, we prove our results on the Hilbert series of the CFN-MC ideal
which implies the results on the volume of the polytope $R_T$.


\section{Preliminaries on Trees}

This section provides a brief background on combinatorial
trees and metric trees.  A more detailed description can be found
in \cite{felsenstein2004,semple2003}.

\begin{definition}
A \textit{tree} is a connected graph with no cycles. 
A \textit{leaf} of the tree $T$ is a node of $T$ of degree 1. 
A tree is \textit{rooted} if it has a distinguished node of degree 2, 
called the \textit{root}. A \textit{rooted binary tree} is a rooted tree in 
which all non-leaf, non-root nodes have degree 3. 
An internal node of $T$ is a \emph{cherry node} if it is adjacent to two leaves. 
\end{definition}

\begin{example}
Consider the rooted binary tree in Figure \ref{fig:treeex}. It is rooted with root $a$. The leaves of this tree are $f,g,h,i$ and $j$. This tree is binary, since the three nodes $b,c$ and $e$ that are not the root or the leaves have degree three. The nodes $c$ and $e$ are both cherry nodes.

\begin{figure}
\begin{subfigure}{.45\textwidth}
\begin{center}
\begin{tikzpicture}[scale=.7]
\draw(0,0)--(3.5,3.5)--(7,0);
\draw(1,1)--(2,0);
\draw(2.5,2.5)--(5,0);
\draw(4,1)--(3,0);
\draw[fill] (1,1) circle [radius=.05];
\node[left] at (1,1) {$c$};
\draw[fill] (2.5,2.5) circle [radius=.05];
\node[left] at (2.5,2.5) {$b$};
\draw[fill] (3.5,3.5) circle [radius=.05];
\node[left] at (3.5,3.5) {$a$};
\draw[fill] (4,1) circle [radius=.05];
\node[left] at (4,1) {$e$};
\draw[fill] (0,0) circle [radius = .05];
\node[below] at (0,0) {$f$};
\draw[fill] (2,0) circle [radius = .05];
\node[below] at (2,0) {$g$};
\draw[fill] (3,0) circle [radius = .05];
\node[below] at (3,0) {$h$};
\draw[fill] (5,0) circle [radius = .05];
\node[below] at (5,0) {$i$};
\draw[fill] (7,0) circle [radius = .05];
\node[below] at (7,0) {$j$};
\end{tikzpicture}
\end{center}
\caption{An example of a rooted binary tree}
\label{fig:treeex}
\end{subfigure} \hfill \begin{subfigure}{.45\textwidth}
\begin{center}
\begin{tikzpicture}[scale=.7]
\node[left] at (.5,.5) {2};
\node[right] at (1.5,.5) {2};
\node[left] at (1.75,1.75) {2};
\node[right] at (3.25,1.75){3};
\node[left] at (3.5,.5){1};
\node[right] at (4.5,.5){1};
\node[left] at (3,3){1};
\node[right] at (5.25,1.75){5};
\draw(0,0)--(3.5,3.5)--(7,0);
\draw(1,1)--(2,0);
\draw(2.5,2.5)--(5,0);
\draw(4,1)--(3,0);
\draw[fill] (1,1) circle [radius=.05];
\draw[fill] (2.5,2.5) circle [radius=.05];
\draw[fill] (3.5,3.5) circle [radius=.05];
\draw[fill] (4,1) circle [radius=.05];
\draw[fill] (0,0) circle [radius = .05];
\draw[fill] (2,0) circle [radius = .05];
\draw[fill] (3,0) circle [radius = .05];
\draw[fill] (5,0) circle [radius = .05];
\draw[fill] (7,0) circle [radius = .05];
\end{tikzpicture}
\end{center}
\caption{Branch lengths that make this tree equidistant.}
\label{fig:equitreeex}
\end{subfigure}
\caption{}
\end{figure}
\end{example}

Typically, we orient trees with the root at the top 
of the page and the leaves toward the bottom. This allows us to 
think of the tree as being directed, so that time starts at the 
root and progresses in the direction of leaves. Labeling the leaves 
of the tree with the taxa $\{1, \dots, n\}$ gives a proposed 
evolutionary history of these taxa. For any tree $T$, there 
exists a unique path between any two nodes in the tree. This 
allows us to say that if $a$ and $b$ are nodes of the rooted tree $T$, 
then $a$ is an \textit{ancestor} of $b$ and $b$ is a 
\textit{descendant} of $a$ if $a$ lies along the path from the root 
of $T$ to $b$. Furthermore, a rooted binary tree on $n$ leaves has 
$n-1$ internal nodes and $2n-2$ edges. Proofs of these facts can 
be found in Chapter 2.1 of \cite{west2001}.  

Trees may also come equipped with \emph{branch lengths}, 
which can represent time, amount of substitution, etc.
The branch lengths are assignments of positive real numbers to
each edge in the tree.  The assignment of branch lengths to the edges of
a tree induces a metric on the nodes in the tree, where the
distance between a pair of nodes is the sum of the branch lengths
on the unique path in the tree connecting those nodes.
Not every metric on a finite set arises in this way; for example, 
 the resulting \emph{tree metrics}
must satisfy the four-point condition \cite{buneman1974}.  In this paper we are interested in 
the following restricted class of tree metrics.

\begin{definition}
An \textit{equidistant} tree  $T$ is a rooted  tree with 
positive branch lengths 
such that the distance between the root and any leaf is the same.
\end{definition}

The tree pictured in Figure \ref{fig:equitreeex} is an example of an equidistant tree.
In the phylogenetic modeling literature, this is
known as imposing the \textit{molecular clock} condition on the model.
In other contexts, an equidistant tree metric is also known as an \emph{ultrametric}.


\section{The CFN-MC Model}\label{sec:CFNMC}

In this section, we  review the discrete Fourier transform,
as well as known results concerning the toric structure of the CFN model
without the molecular clock \cite{sturmfels2005}.
Note that the CFN model is also referred to as
the binary Jukes-Cantor model and the binary symmetric model throughout the literature.
We use these results to provide a combinatorial description of the
toric ideal of phylogenetic invariants of the CFN model \emph{with} the molecular clock,
which is the main object of study for the present paper.

The CFN model describes substitutions 
at a single site in the gene sequences of the taxa in question. 
It is a two-state model, where the states are purine (adenine and guanine) 
and pyrimidine (thymine and cytosine). We denote purines 
with  $U$ and pyrimidines with  $Y$. The CFN model assumes a 
continuous-time Markov process along a fixed rooted binary tree 
with positive branch lengths. The rate matrix for the Markov process in the CFN model is 
\[
Q = \begin{blockarray}{ccc}
U & Y & \\
\begin{block}{[cc]c}
-\alpha & \alpha & U \\
\alpha & -\alpha & Y\\
\end{block}
\end{blockarray},
\]
for some parameter $\alpha > 0$ that describes the rate of change of 
purines to pyrimidines or vice versa. Note that in the CFN model, 
we assume that the rate of substitution from purine to pyrimidine 
is equal to the rate of substitution from pyrimidine to purine. 
We also assume that the distribution of states at the root of the tree,
or \emph{root distribution}, is uniform.

Let $t_e > 0$ be the branch length of an edge $e$ in the rooted
binary tree $T$. The transition matrix $M^e$ associated 
to the edge $e$ is the matrix exponential,
\begin{align*}
M^e & = \exp(Qt_e) \\
&= \begin{bmatrix}
(1 + e^{-2\alpha t_e})/2 & (1 - e^{-2\alpha t_e})/2 \\
(1 - e^{-2\alpha t_e})/2 & (1 + e^{-2\alpha t_e})/2
\end{bmatrix}.
\end{align*}

Denote by $a(e)$ and $d(e)$ the two nodes adjacent to $e$ 
so that $d(e)$ is a descendant of $a(e)$. Then the $(i,j)$th entry 
of $M^e$, $M^e(i,j)$, is the probability that $d(e)$ 
has state $j$ given that $a(e)$ has state $i$ for all $i,j \in \{U,Y\}$.

Let $T$ be a rooted binary tree with edge set $E$ and nodes 
$1, \dots, 2n-1$. For the following section, we label these nodes so that $1, \dots, n$ are leaf labels. 
We can identify the set of states $\{U,Y\}$ with elements of the two element group $\Z_2$. 
(Note that it does not matter which identification is chosen; either $U = 0, Y= 1$, or $Y = 0, U = 1$
produce the same results.)
Let $\bu \in \Z_2^{2n-1}$ be a labeling of all of the nodes 
of $T$ by states in the state space, and let $u_i$ denote the 
$i$th coordinate of $\bu$, which is the labeling of node $i$. 
Then the probability of observing the set of states $\bu$ is
\begin{equation}\label{eqn:probability1}
\frac{1}{2} \prod_{e \in E} M^e (u_{a(e)}, u_{d(e)}).
\end{equation}
We note that the factor of $\frac{1}{2}$ appears in this formula
because the distribution of states at the root of the tree is uniform.

\begin{example}
For the tree in Figure \ref{fig:probabilityexample1}, the probability of 
observing the states $(0,1,1,0,1)$ (left to right and bottom to top) is
\[
\frac{1}{2} M^{e_1}(1,0) M^{e_2}(0,0) M^{e_3}(0,1) M^{e_4}(1,1).
\]

\begin{figure}

\begin{center}
\begin{tikzpicture}
\draw(0,0)--(2,2)--(4,0);
\draw(1,1)--(2,0);
\draw[fill] (0,0) circle [radius=.05];
\draw[fill] (1,1) circle [radius=.05];
\draw[fill] (2,2) circle [radius=.05];
\draw[fill] (2,0) circle [radius=.05];
\draw[fill] (4,0) circle [radius=.05];
\node[below] at (0,0) {0};
\node[below] at (2,0) {1};
\node[below] at (4,0) {1};
\node[left] at (1,1) {0};
\node[above] at (2,2) {1};
\node[left] at (1.5,1.5) {$e_1$};
\node[left] at (.5,.5) {$e_2$};
\node[right] at (1.5,.5) {$e_3$};
\node[right] at (3,1) {$e_4$};
\end{tikzpicture}
\end{center}
\caption{A labeling of all nodes of tree $T$ with elements of $\Z_2$.}
\label{fig:probabilityexample1}
\end{figure}
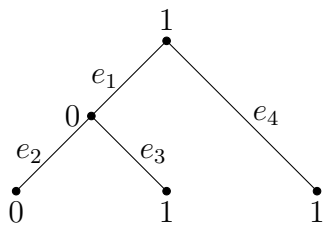
\end{example}

We have described the CFN model thus far with all variables observed.
However, in typical phylogenetic analysis we do not have access to the
DNA of the unknown ancestral species, and hence we need to consider a
hidden variable model where all internal nodes correspond to hidden
states.
In this case, to determine the probability of observing a certain set 
of states at the leaves, we sum the probabilities given by Equation (\ref{eqn:probability1}) 
over all possible labelings of the internal nodes of the tree.  
Let $\bv \in \Z_2^n$ be a labeling of the leaves of $T$. 
Since the CFN-MC model assumes a uniform distribution of states at the root, 
the probability of observing the set of states $\bv$ at the leaves is
\begin{equation}\label{eqn:prob}
p(v_1, \dots, v_n) = \frac{1}{2} \sum_{(v_{n+1}, \dots, v_{2n-1}) \in \Z_2^{n-1}}  
\prod_{e \in E} M^e (v_{a(e)}, v_{d(e)}).
\end{equation}
Note that $d(e)$ might be a leaf, in which case $v_{d(e)} = v_i$ for the appropriate value of $i$.

\begin{example}
Consider the tree from Figure \ref{fig:probabilityexample1}. 
We use Equation (\ref{eqn:prob}) to compute the probability 
$p(0,1,1)$ of observing states $(0,1,1)$ at the leaves of $T$. 
Summing over all possible labelings of the internal nodes of $T$ yields
\begin{align*}
p(0,1,1)  = & \tfrac{1}{2} \left( M^{e_1}(0,0) M^{e_2}(0,0) M^{e_3}(0,1) M^{e_4}(0,1)  \right.\\
& + M^{e_1}(1,0) M^{e_2}(0,0) M^{e_3}(0,1) M^{e_4}(1,1) \\
& + M^{e_1}(0,1) M^{e_2}(1,0) M^{e_3}(1,1) M^{e_4}(0,1) \\
&  \left. + M^{e_1}(1,1) M^{e_2}(1,0) M^{e_3}(1,1) M^{e_4}(1,1) \right).
\end{align*}

\end{example}

In the following discussion, we perform a linear change 
of coordinates on the probability coordinates 
and introduce new free parameters in terms of the entries of the
transition matrices.  This allows us 
to realize this parametrization as a monomial map. In order to accomplish this, 
we first provide some background concerning group-based models 
and the discrete Fourier transform. We always assume that $G$ is a finite abelian group. 

\begin{definition}
Let $M^{e} = \exp(Qt_e)$ be a transition matrix arising from a 
continuous-time Markov process along a tree. Let $G$ be a finite abelian 
group under addition with order equal to the number of states of the model, and identify 
the set of states with elements of $G$. The model is \textit{group-based} 
with respect to $G$ if for each transition matrix $M^{e}$ arising from the model, 
there exists a function $f^e: G \rightarrow \R$ such that 
$M^{e}(g,h) = f^e(g-h)$ for all $g,h \in G$.
\end{definition}

In particular, note that the CFN-MC model is group-based 
with respect to $\Z_2$ with function $f^e: \Z_2 \rightarrow \R$ defined by 
\[
f^e(0) = (1+\exp(-2\alpha t_e))/2 \quad \mbox{ and } \quad f^e(1) = (1-\exp(-2\alpha t_e))/2.
\]

\begin{definition}
The \textit{dual group} $\hat{G} = \text{Hom}(G, \C^{\times})$ 
of a group $G$ is the group of all homomorphisms $\chi: G \rightarrow \C^{\times}$, 
where $\C^{\times}$ denotes the group of non-zero complex numbers under multiplication. 
Elements of the dual group are called \textit{characters}. 
Let $\bbone$ denote the constant character that maps all elements of $G$ to 1.
\end{definition}

Throughout this section, we will make use of the following classical theorems. Proofs of these can be found in \cite{simon1996}.

\begin{prop}
Let $G$ be a finite abelian group. Its dual group $\hat{G}$ is isomorphic to $G$. 
Furthermore, for two finite abelian groups $G_1$ and $G_2$, 
$\widehat{G_1 \times G_2} \cong \hat{G_1} \times \hat{G_2}$ via 
$\chi((g_1, g_2)) = \chi_1(g_1) \chi_2 (g_2)$ for $g_1 \in G_1$ and $g_2 \in G_2$ and
some $\chi_1 \in \hat{G_1}$ and $\chi_2 \in \hat{G_2}$.
\end{prop}

\begin{definition}
Let $f: G \rightarrow \C$ be a function. 
The \textit{discrete Fourier transform} of $f$ is the function
\[
\hat{f}:  \hat{G} \rightarrow \C,  \quad \quad \chi  \mapsto  \sum_{g \in G} \chi(g) f(g).
\]
\end{definition}

The discrete Fourier transform is the linear change of coordinates 
that allows us to view Equation (\ref{eqn:prob}) as a monomial parametrization. 
We can write the Fourier transform of $p$ over $\Z_2^n$ in equation (\ref{eqn:prob}) as
\[
\hat{p}(\chi_1, \dots, \chi_n) = \sum_{(g_1, \dots, g_n) \in \Z_2^n} 
p(g_1, \dots, g_n) \prod_{i=1}^n \chi_i (g_i). \\
\]
Let $\hat{\Z}_2 = \{ \bbone, \phi \}$, where $\phi$ denotes the only nontrivial 
homomorphism from $\Z_2$ to $\C$. Then $\Z_2$ and $\hat{\Z}_2$ are isomorphic via
the map that identifies $0$ to $\bbone$ and $1$ with $\phi$.
Using this fact, 
we can write $\hat{p}$ as a function of $n$ elements of $\Z_2$ as
\[
\hat{p}(i_1, \dots, i_n) = \sum_{(j_1, \dots, j_n) \in \Z_2^n} (-1)^{i_1j_1 + \dots + i_n j_n} p(j_1, \dots, j_n)
\]
for all $(i_1, \dots, i_n) \in \Z_2^n$. 

The following theorem, 
independently discovered by Evans and Speed in \cite{evans1993} and 
Hendy and Penny in \cite{hendy1993}, describes the monomial parametrization 
obtained from the discrete Fourier transform.  A detailed account can
also be found in Chapter 15 of \cite{sullivant2018}.

\begin{thm}
Let $p(g_1, \dots, g_n)$ be the polynomial describing 
the probability of observing states $(g_1, \dots, g_n) \in G^n$ at the 
leaves of phylogenetic tree $T$ under a group-based model. 
Denote by $\pi$ the distribution of states at the root of the tree. 
Let $f^e: G \rightarrow \R$ denote the function associated to edge $e$
by the definition of a group based model.  
Then the Fourier transform of $p$ is
\[
\hat{p}(\chi_1, \dots, \chi_n) = \hat{\pi}\Big(\prod_{i=1}^n \chi_i\Big) \prod_{e \in E(T)} \hat{f^e} \Big( \prod_{l \in \lambda(e)} \chi_l\Big),
\]
where $\lambda(e)$ is the set of all leaves that are descended from edge $e$.
\end{thm}

 Note that since $\pi$ is the uniform distribution,
\[
\hat{\pi}(\chi) = \frac{1}{2} (\chi(0) + \chi(1))  = 
\begin{cases}
1, & \text{ if } \chi = \bbone, \\
0, & \text{ if } \chi = \phi.
\end{cases}
\]
Interpreting this in the context of the Fourier transform of $p$ and using the isomorphism of $\Z_2$ and $\hat{\Z}_2$  gives that
\begin{equation} \label{eqn:transformafterroot}
\hat{p}(g_1, \dots, g_n) = \begin{cases}
\prod_{e \in E(T)} \hat{f^e} \big( \sum_{l \in \lambda{e}} g_l\big), 
& \text{ if } \sum_{i=1}^n g_i = 0  \\
0, & \text{ if } \sum_{i=1}^n g_i = 1.
\end{cases}
\end{equation}
Consider the Fourier transform of each $f^e(g)$. In the case of the CFN-MC model, we can think 
of the discrete Fourier transform as a simultaneous diagonalization of the transition matrices via 
a $2 \times 2$ Hadamard matrix. Indeed, letting $H = \begin{bmatrix}
1 & 1 \\
1 & -1
\end{bmatrix}$ gives that
\[
H^{-1} M^e H = \begin{bmatrix}
a_0^e & 0 \\
0 & a_1^e 
\end{bmatrix},
\]
where $a_1^e = \exp(-2\alpha t_e)$ and $a_0^e = 1$. 
(Although $a_0^e = 1$ for all $e$, it is useful to think of $a^e_0$ 
as a free parameter for the following discussion.) Note that these values 
are exactly those obtained by performing the discrete Fourier transform on $f^e$:
\begin{align*}
\hat{f^e}(\bbone) & = f^e(0) \bbone(0) + f^e(1) \bbone(1)   \\
& = \frac{1 + \exp(-2\alpha t_e)}{2} + \frac{1 - \exp(-2 \alpha t_e)}{2} \\
&= 1 \\
\hat{f^e}(\phi) &= f^e(0) \phi(0) + f^e(1) \phi(1) & \\
& =  \frac{1 + \exp(-2\alpha t_e)}{2} - \frac{1 - \exp(-2 \alpha t_e)}{2}\\
 &= \exp(-2 \alpha t_e).
\end{align*}

Using these new parameters, the isomorphism of $\Z_2$ and $\hat{\Z}_2$ 
and equation (\ref{eqn:transformafterroot}), we can see that
\begin{equation}\label{eqn:fourierparam1}
\hat{p}(g_1, \dots, g_n) = \begin{cases}
\prod_{e \in E(T)} a_{i(e)}^e, & \text{ if } \sum_{i=1}^n g_i = 0 \\
0, & \text{ if } \sum_{i=1}^n g_i = 1 ,
\end{cases}
\end{equation}
where $i(e)$ denotes the sum in $\Z_2$ of the group 
elements at all leaves descended from $e$. Note that this is, 
in fact, a monomial parametrization, as desired.

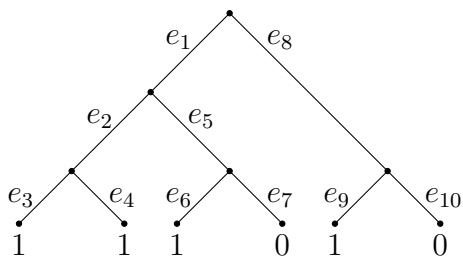
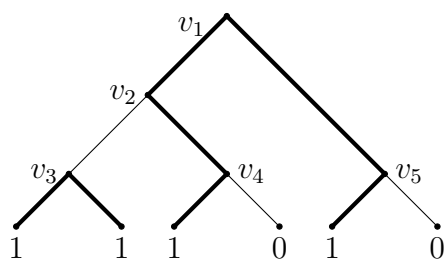
\begin{figure}
\begin{center}
\begin{subfigure}{.45\textwidth}
\begin{tikzpicture}[scale=.7]
\draw(0,0)--(4,4)--(8,0);
\draw(1,1)--(2,0);
\draw(2.5,2.5)--(5,0);
\draw(4,1)--(3,0);
\draw(7,1)--(6,0);
\draw[fill] (0,0) circle [radius=.05];
\node[below] at (0,0) {1};
\draw[fill] (2,0) circle [radius=.05];
\node[below] at (2,0) {1};
\draw[fill] (3,0) circle [radius=.05];
\node[below] at (3,0) {1};
\draw[fill] (5,0) circle [radius=.05];
\node[below] at (5,0) {0};
\draw[fill] (6,0) circle [radius=.05];
\node[below] at (6,0) {1};
\draw[fill] (8,0) circle [radius=.05];
\node[below] at (8,0) {0};
\draw[fill] (1,1) circle [radius=.05];
\draw[fill] (4,1) circle [radius=.05];
\draw[fill] (2.5,2.5) circle [radius=.05];
\draw[fill] (4,4) circle [radius=.05];
\draw[fill] (7,1) circle [radius=.05];
\node[left] at (.5,.5) {$e_3$};
\node[left] at (2,2) {$e_2$};
\node[left] at (3.5,3.5) {$e_1$};
\node[right] at (1.5,.5) {$e_4$};
\node[right] at (3,2) {$e_5$};
\node[left] at (3.5,.5) {$e_6$};
\node[right] at (4.5,.5) {$e_7$};
\node[right] at (4.5,3.5) {$e_8$};
\node[left] at (6.5,.5) {$e_9$};
\node[right] at (7.5,.5) {$e_{10}$};
\end{tikzpicture}
\caption{A leaf labeling of $T$ with elements of $\Z_2$ that sum to $0$}
\label{fig:paramexample}
\end{subfigure}\hfill\begin{subfigure}{.45\textwidth}
\begin{tikzpicture}[scale=.7]
\draw(0,0)--(4,4)--(8,0);
\draw(1,1)--(2,0);
\draw(2.5,2.5)--(5,0);
\draw(4,1)--(3,0);
\draw(7,1)--(6,0);
\draw[fill] (0,0) circle [radius=.05];
\node[below] at (0,0) {1};
\draw[fill] (2,0) circle [radius=.05];
\node[below] at (2,0) {1};
\draw[fill] (3,0) circle [radius=.05];
\node[below] at (3,0) {1};
\draw[fill] (5,0) circle [radius=.05];
\node[below] at (5,0) {0};
\draw[fill] (6,0) circle [radius=.05];
\node[below] at (6,0) {1};
\draw[fill] (8,0) circle [radius=.05];
\node[below] at (8,0) {0};
\draw[fill] (1,1) circle [radius=.05];
\node[left] at (1,1) {$v_3$};
\draw[fill] (4,1) circle [radius=.05];
\node[right] at (4,1) {$v_4$};
\draw[fill] (2.5,2.5) circle [radius=.05];
\node[left] at (2.5,2.5) {$v_2$};
\draw[fill] (4,4) circle [radius=.05];
\node[left] at (3.8,3.8) {$v_1$};
\draw[fill] (7,1) circle [radius=.05];
\node[right] at (7,1) {$v_5$};
\draw[ultra thick] (0,0) -- (1,1) -- (2,0);
\draw[ultra thick] (3,0) -- (4,1)-- (2.5,2.5) -- (4,4)--(7,1) -- (6,0);
\end{tikzpicture}
\caption{The path system associated to this labeling}
\label{fig:disjointpathsexample}
\end{subfigure}
\end{center}
\caption{The tree $T$ referenced in Example \ref{ex:monomialparam1}}
\end{figure}

\begin{example}\label{ex:monomialparam1}
Consider the tree $T$ in Figure \ref{fig:paramexample}. We compute
 $\hat{p}(1,1,1,0,1,0)$. Using Equation (\ref{eqn:fourierparam1}), we see that
\[
\hat{p}(g_1, \dots, g_n) = a_1^{e_1} a_0^{e_2} a_1^{e_3} a_1^{e_4} a_1^{e_5} a_1^{e_6} a_0^{e_7} a_1^{e_8} a_1^{e_9} a_0^{e_{10}}.
\]
\end{example}

Let $e_1, e_2, e_3$ be edges of rooted binary tree $T$ that are adjacent 
to a single node $v$, where $v = d(e_1)$ and 
$v = a(e_2) = a(e_3)$. Then $i(e_1) = i(e_2) + i(e_3)$, so 
$i(e_1) + i(e_2) + i(e_3) = 0$. In particular, this means that at any 
internal node, the edges adjacent to that node have an even number of 1's. 
Since the labels at the leaves of the tree sum to 0,
 if $e_1, e_2$ are the edges adjacent to the root, 
then $i(e_1) + i(e_2) = 0$. Therefore, to each labeling of the leaves of $T$ with elements of $\Z_2$ 
that sum to $0$  we may associate a set of disjoint paths, or path systems,  between leaves of $T$. 
 Furthermore, given a set of disjoint paths between leaves of $T$,
we obtain a labeling of the leaves that sums to 0 by assigning a
1 to each leaf included in one of the paths and a 0 elsewhere.
So labelings of the leaves of $T$ that sum to 0 and sets of disjoint paths between leaves of $T$
are in bijection with one another.

\begin{definition}
Let $\zeven$ denote the set of all labelings of the 
leaves of $T$ with elements of $\Z_2$ that sum to $0 $.
The \emph{path system} associated to a 
labeling $(i_1, \dots, i_n) \in \zeven$ is the unique 
set of paths in $T$ that connect the leaves of $T$ 
that are labeled with $1$ and do not use any of the same edges. 
We often denote a path system by $\P$.
\end{definition}

In this context, the edges for which $a_1^e$ appears in the parametrization (\ref{eqn:fourierparam1}) 
of $\hat{p}(i_1, \dots, i_n)$ for $(i_1, \dots, i_n) \in \zeven$ are exactly those that 
appear in the path system associated to $(i_1, \dots, i_n)$.

\begin{example}
The path system associated to the labeling 
$(1,1,1,0,1,0)$ from Example \ref{ex:monomialparam1} 
is pictured in Figure \ref{fig:disjointpathsexample}. Notice that the bold 
edges $e$ in $T$ are exactly those for which $a_1^e$ 
appears in the parametrization of $\hat{p}(1,1,1,0,1,0)$.
\end{example}

We now restrict the parametrization of the CFN model to trees that satisfy 
the molecular clock condition, which restricts us to a lower-dimensional 
subspace of the parameter space and provides a new combinatorial 
way of interpreting the Fourier coordinates. The molecular 
clock condition imposes that if $e_1, \dots, e_s$ and $f_1, \dots, f_r$ 
are two paths from an internal node $v$ to leaves descended from $v$, then 
$t_{e_1} + \dots + t_{e_s} = t_{f_1} + \dots + t_{f_s}$. 
On the level of transition matrices, this means that
\begin{align*}
M^{e_1} \dots M^{e_s} & = \exp(Qt_{e_1}) \dots \exp(Qt_{e_s}) \\
&= \exp(Q(t_{e_1} + \dots + t_{e_s})) \\
&= \exp(Q(t_{f_1} + \dots + t_{f_r}))\\
&=M^{f_1} \dots M^{f_r}.
\end{align*}

Since we can apply the Fourier transform to diagonalize the resulting matrices
on both sides of this equation, we see that the products of the new parameters
satisfy the identities:
\[
a_0^{e_1} \dots a_0^{e_s}  = a_0^{f_1} \dots a_0^{f_r} \quad \text{ and } \quad
a_1^{e_1} \dots a_1^{e_s}  = a_1^{f_1} \dots a_1^{f_r}.
\]

In particular, this means that we may define new parameters $a_0^v$ and $a_1^v$ for each internal node $v$ by
\[
a_i^v = a_i^{e_1} \dots a_i^{e_s}
\]
for $i = 0,1$ where $e_1, \dots, e_s$ is a path from $v$ to any 
leaf descended from $v$. Note that if $v$ is a leaf, then $a_i^v = 1$.
So we may exclude it from the parametrization,
and restrict to a parametrization by $a_0^v$ and $a_1^v$
where $v$ is an internal node.
Furthermore, note that for any edge $e_1$ and $i=0,1$,
we have the relations
\begin{align}\label{eqn:nodeparamrep}
a_i^{a(e_1)} \big(a_i^{d(e_1)})^{-1} & = a_i^{e_1} \dots a_i^{e_s} (a_i^{e_s})^{-1} \dots (a_i^{e_2})^{-1} \nonumber \\
&= a_i^{e_1},
\end{align}
where $e_1, \dots, e_s$ is a path from the ancestral node $a(e_1)$ to a leaf descended from $d(e_1)$.

Let $(i_1, \dots, i_n) \in \zeven$. Let $\P$ be the path system
 associated to $i_1, \dots, i_n$. Denote by $\text{Int}(T)$ 
the set of all internal nodes of tree $T$.

\begin{definition}
We say that $v$ is the \emph{top-most node} of a path in $\P$ if both 
of the edges descended from $v$ are in a path in $\P$. In other words, $v$ 
is the node of the path that includes it that is closest to the root. 
The \textit{top-set} Top$(i_1, \dots, i_n)$ is the set of all top-most nodes 
of paths in $\P$. The \textit{top-vector} is the vector in $\R^{\text{Int}(T)}$ 
with $v$ component equal to $1$ if $v \in \text{Top}(i_1, \dots, i_n)$ and $0$ otherwise. 
We denote the top-vector of a particular 
path system $\P$ by $[\P]$ or $\bx^{\P}$, depending upon the context.
\end{definition}

\begin{example}
Consider the path system associated to labeling $(1,1,1,0,1,0) \in \zeven$ 
pictured in Figure \ref{fig:disjointpathsexample}. The top-set of this path system
 is $\text{Top}(1,1,1,0,1,0) = \{ v_1, v_3 \}$. 
The top-vector is $(1,0,1,0,0) \in \R^{\text{Int}(T)}$.
\end{example}

For each internal node $v$, we define 
two new parameters $b_0^v$ and $b_1^v$ by
\begin{align}\label{eqn:bparams}
b_0^v &=\begin{cases} (a_0^v)^2 & \text{ if $v$ is the root, and} \\
a_0^v & \text{ otherwise}, \end{cases} \\ \nonumber
b_1^v &= \begin{cases} (a_1^v)^2 & \text{ if $v$ is the root, and} \\
(a_1^v)^2(a_0^v)^{-1} & \text{ otherwise}.
\end{cases}
\end{align}
We now rewrite the parametrization of
$\hat{p}(i_1,\dots,i_n)$ from Equation (\ref{eqn:fourierparam1})
  in terms of these new parameters
using Equation (\ref{eqn:nodeparamrep}).
Let $\P$ be the path system associated to $(i_1,\dots,i_n)$.
The choice of $b_0^v$ or $b_1^v$ in the new parametrization
will depend upon the position of $v$ in $\P$.

Let $v$ be an internal node of $T$ that is not the root.
Let $e_1, e_2, e_3$ be the edges adjacent to $v$
so that $v = d(e_1)$ and $v = a(e_2) = a(e_3)$.
By Equation (\ref{eqn:nodeparamrep}), $a_i^{e_1}$, $a_i^{e_2}$ and $a_i^{e_3}$
are the only parameters in which some $a_0^v$ or $a_1^v$ appear.

If $v$ is not in any path in $\P$, then $\P$ does not
use edges $e_1, e_2$ or $e_3$. So the factors of $\hat{p}(i_1,\dots,i_n)$
in Equation (\ref{eqn:fourierparam1}) associated to these edges are
\begin{align*}
a_0^{e_1} a_0^{e_2} a_0^{e_3} &= a_0^{a(e_1)} (a_0^v)^{-1} \cdot a_0^v (a_0^{d(e_2)})^{-1} \cdot a_0^v (a_0^{d(e_3)})^{-1} \\
&= a_0^{a(e_1)}(a_0^{d(e_2)})^{-1}(a_0^{d(e_3)})^{-1} \cdot (a_0^v)^2 (a_0^v)^{-1} \\
&= a_0^{a(e_1)}(a_0^{d(e_2)})^{-1}(a_0^{d(e_3)})^{-1} \cdot a_0^v, \\
\end{align*}
by Equation (\ref{eqn:nodeparamrep}).
So if $v$ is not in any path in $\P$,
$a_0^v$ is the only factor of $\hat{p}(i_1,\dots,i_n)$ involving $v$.

Similarly, consider the case where $v$ is in a path in $\P$, 
but it is not the top-most node of a path in $\P$.
Then this path includes $e_1$, and without loss of generality,
we may assume it includes $e_2$ and not $e_3$.
Then the factors of $\hat{p}(i_1,\dots,i_n)$ involving $e_1, e_2$ and $e_3$ are
\begin{align*}
a_1^{e_1} a_1^{e_2} a_0^{e_3} &= a_1^{a(e_1)} (a_1^v)^{-1} \cdot a_1^v (a_1^{d(e_2)})^{-1} \cdot a_0^v (a_0^{d(e_3)})^{-1} \\
&= a_1^{a(e_1)}(a_1^{d(e_2)})^{-1}(a_0^{d(e_3)})^{-1} \cdot (a_0^v) (a_1^v) (a_1^v)^{-1} \\
&= a_0^{a(e_1)}(a_0^{d(e_2)})^{-1}(a_0^{d(e_3)})^{-1} \cdot a_0^v, \\
\end{align*}
by Equation (\ref{eqn:nodeparamrep}).
So in this case, we have again that
$a_0^v$ is the only factor of $\hat{p}(i_1,\dots,i_n)$ involving $v$.

Finally, consider the case where $v$ is the top-most node of a path in $\P$.
Then this path includes $e_2$ and $e_3$ but not $e_1$.
Then the factors of $\hat{p}(i_1,\dots,i_n)$ involving $e_1, e_2$ and $e_3$ are
\begin{align*}
a_0^{e_1} a_1^{e_2} a_1^{e_3} &= a_0^{a(e_1)} (a_0^v)^{-1} \cdot a_1^v (a_1^{d(e_2)})^{-1} \cdot a_1^v (a_1^{d(e_3)})^{-1} \\
&= a_0^{a(e_1)}(a_1^{d(e_2)})^{-1}(a_1^{d(e_3)})^{-1} \cdot (a_0^v)^{-1} (a_1^v)^2 
\end{align*}
by Equation (\ref{eqn:nodeparamrep}). So in this case,
$(a_0^v)^{-1} (a_1^v)^2$ is exactly the factor of $\hat{p}(i_1,\dots,i_n)$ involving $v$.

An analogous argument shows that when $v$ is the root,
the factor of $\hat{p}(i_1,\dots,i_n)$ involving $v$ is $(a_0^v)^2$ 
when $v$ is not the top-most node of a path in $\P$,
and $(a_1^v)^2$ when $v$ is the top-most node of a path in $\P$.
Note that it can never be the case that the root is in a path and is not the top-most node.

So the new parameters $b_0^v$ and $b_i^v$ for each internal node $v$ of $T$
allow us to rewrite the parametrization in Equation (\ref{eqn:fourierparam1}) as
\begin{equation}\label{eqn:fourierparam2}
\hat{p}(i_1, \dots, i_n) = \prod_{v \in \text{Top}(i_1, \dots, i_n)} b_1^v \times \prod_{v \in \bar{\text{Top}(i_1, \dots, i_n)}} b_0^v,
\end{equation}
where $ \bar{\text{Top}(i_1, \dots, i_n)}$ denotes the compliment of
$\text{Top}(i_1,\dots,i_n)$ in the set of all internal nodes of $T$.

%

\begin{example}\label{ex:monomialparam2}
Consider the parametrization of $\hat{p}(1,1,1,0,1,0)$ given in 
Example \ref{ex:monomialparam1} for the tree $T$ pictured in \ref{fig:paramexample}. 
First, we verify the identity in Equation (\ref{eqn:nodeparamrep}) for 
$a_1^{e_1}$. We have that $a(e_1) = v_1$ and $d(e_1) = v_2$. We have defined 
$a_1^{v_1} = a_1^{e_1} a_1^{e_2} a_1^{e_3}$ and $a_1^{v_2} = a_1^{e_2} a_1^{e_3}$. Therefore
\begin{align*}
a_1^{v_1} (a_1^{v_2})^{-1} & = a_1^{e_1} a_1^{e_2} a_1^{e_3} (a_1^{e_3})^{-1} (a_1^{e_2})^{-1} \\
&= a_1^{e_1}.
\end{align*} 
Note that while the choices of paths from $v_1$ and $v_2$ to leaves descended 
from them was not unique, the molecular clock condition implies that the 
above holds for any such choice of paths.

Substituting the identities in Equation (\ref{eqn:nodeparamrep}) into $\hat{p}(1,1,1,0,1,0)$, 
and applying the fact that if $l$ is a leaf of $T$ then $a_i^l = 1$ for $i=0,1$ yields
\begin{align*}
\hat{p}(1,1,1,0,1,0) &= a_1^{v_1} (a_1^{v_2})^{-1} a_0^{v_2} (a_0^{v_3})^{-1} a_1^{v_3} a_1^{v_3}  a_1^{v_2} (a_1^{v_4})^{-1} a_1^{v_4} a_0^{v_4}  a_1^{v_1} (a_1^{v_5})^{-1} a_1^{v_5} a_0^{v_5} \\
&= (a_1^{v_1})^2 a_0^{v_2} (a_1^{v_3})^2 (a_0^{v_3})^{-1} a_0^{v_4} a_0^{v_5}.
\end{align*}

Substituting the new parameters, $b_0^v$ and $b_1^v$ as defined in Equation (\ref{eqn:bparams}) yields
\[
\hat{p}(1,1,1,0,1,0) = b_1^{v_1} b_0^{v_2} b_1^{v_3} b_0^{v_4} b_0^{v_5},
\]
as needed.
\end{example}

Note that
two labelings of the leaves with group elements $(i_1, \dots, i_n)$ and $(j_1, \dots, j_n)$
 have the same top-sets if and only if $\hat{p}(i_1, \dots, i_n) = \hat{p}(j_1, \dots, j_n)$. 
 Therefore, Equation (\ref{eqn:fourierparam2})
 allows us to define new coordinates that are indexed by valid top-sets of 
 path systems in $T$. These coordinates are in the polynomial ring 
\[
\K[\ur] := \K [ r_{k_1, \dots, k_{n-1}} : (k_1, \dots, k_{n-1}) = 
[\P] \text{ for some path system } \P]
\]
where $(i_1, \dots, i_n)$ ranges over all elements of $\zeven$. 
By applying this change of coordinates, we effectively quotient by the 
linear relations among the $\hat{p}$ coordinates that arise from the 
fact that their parametrizations in terms of the $b_i^v$ parameters are equal.
This restricts our attention to equivalences classes of labelings in $\zeven$ with the same top-sets.

\begin{definition}\label{def:CFNMCideal}
Label the internal nodes of $T$ with $v_1, \dots, v_{n-1}$. 
The \textit{CFN-MC ideal} $I_T$ is the kernel of the map
\begin{align*}
\K[\ur] & \longrightarrow \K[b_i^v \mid i = 0,1, v \in \text{Int}(T)] \\
r_{k_1, \dots, k_{n-1}} &\longmapsto \prod_{i=1}^{n-1} b_{k_{i}}^{v_i},
\end{align*}
where $(k_1, \dots, k_{n-1})$ ranges over all indicator vectors 
corresponding to top-sets of path systems in $T$.
\end{definition}

Note that the polynomials in the ideal $I_T$ evaluate to zero for every choice of parameters
in the CFN-MC model for the tree $T$.  In particular, these polynomials are \emph{phylogenetic
invariants} of the CFN-MC model.
Another important observation is that $I_T$ is the kernel of a monomial map. 
This implies that $I_T$ is a \textit{toric ideal} and can be analyzed from
a combinatorial perspective.

\begin{definition}
A \emph{toric ideal} is the kernel of a monomial map. 
Equivalently, it is a prime ideal that is generated by binomials.
\end{definition} 

To every monomial map $\K[x_1, \dots, x_m] \rightarrow \K[y_1, \dots, y_d]$, we can associate a $d \times m$ integer matrix.
The entry in the $(i,j)$th position of this matrix is the exponent of $y_j$ in the image of $x_i$ under this map.

Background on toric ideals can be found beginning in Chapter 4 of \cite{sturmfels1996}. 
Some applications of toric ideals to phylogenetics are detailed in \cite{sturmfels2005}.

An equivalent way to define the CFN-MC ideal is as the kernel of the map
\begin{align}\label{eqn:finalparametrization}
\K[\ur] & \longrightarrow \K[t_0, \dots, t_{n-1}] \nonumber \\
r_{k_1, \dots, k_{n-1}} &\longmapsto t_0 \prod_{k_i = 1} t_i,
\end{align}
where $t_0$ is a homogenizing indeterminate.  
Note that these indeterminates $t_i$ are not related to the branch lengths in $T$.
From this perspective, we define
the matrix $A_T$ associated to this monomial map to be the matrix whose 
columns are the indicator vectors of top-sets of path systems
 in $T$ with an added homogenizing row of ones. 
The convex hull of these indicator vectors gives a polytope in 
$\R^{n-1}$ that encodes important information about the ideal $I_T$ \cite[Chapter~4]{sturmfels1996}.
Our goal in this paper is to study the ideals $I_T$ for binary trees and
the corresponding polytopes $R_T$ (to be defined in detail in Section \ref{sec:polytope}).

\begin{example}\label{ex:parametrization}
Let $T$ be the tree pictured in Figure \ref{fig:treeex}. The CFN-MC ideal $I_T$ is in the polynomial ring
$\K[\ur] = \K[r_{0000},r_{1000},r_{0100},r_{0010},r_{0001},r_{1010},r_{1001},r_{0011}]$
where each subscript is the indicator vector of a top-set of a path system in $T$
indexed alphabetically by the internal nodes of $T$.
Therefore, the parametrization in Equation (\ref{eqn:finalparametrization}) is given by

\begin{center}
\begin{minipage}{.2\textwidth}
\begin{align*}
r_{0000} &\mapsto t_0 \\
r_{1000} &\mapsto t_0t_a 
\end{align*}
\end{minipage}\begin{minipage}{.2\textwidth}
\begin{align*}
r_{0100} &\mapsto t_0 t_b \\
r_{0010}&\mapsto t_0 t_c
\end{align*}
\end{minipage}\begin{minipage}{.2\textwidth}
\begin{align*}
r_{0001} &\mapsto t_0 t_d \\
r_{1010} &\mapsto t_0 t_a t_c
\end{align*}
\end{minipage}\begin{minipage}{.2\textwidth}
\begin{align*}
r_{1001} &\mapsto t_0 t_a t_d\\
r_{0011} &\mapsto t_0 t_c t_d.
\end{align*}
\end{minipage}
\end{center}

The matrix $A_T$ associated to this monomial map is obtained by taking its columns
to be all of the subscripts of an indeterminate in $\K[\ur]$ and adding a
homogenizing row of ones. In this case, this matrix is
\[
A_T = \begin{bmatrix}
1 & 1 & 1 & 1 & 1 & 1 & 1 & 1\\
0 & 1 & 0 & 0 & 0 & 1 & 1 & 0\\
0 & 0 & 1 & 0 & 0 & 0 & 0 & 0\\
0 & 0 & 0 & 1 & 0 & 1 & 0 & 1\\
0 & 0 & 0 & 0 & 1 & 0 & 1 & 1
\end{bmatrix}.
\]

We can write $I_T$ implicitly from its parametrization using standard elimination 
techniques \cite[Algorithm~4.5]{sturmfels1996}. This ideal is generated by the binomials
\begin{eqnarray*}
r_{0000}r_{0011} - r_{0010}r_{0001}   & \quad  & r_{0000}r_{1010} - r_{1000}r_{0010}\\
r_{1000}r_{0011} - r_{1010}r_{0001} &   &  r_{0000}r_{1001} -r_{1000}r_{1010}\\
r_{1000}r_{0011} - r_{0010}r_{1001} &  &  r_{0010}r_{1001} -r_{1010}r_{0001}.
\end{eqnarray*}

In fact, these are exactly the binomials described in the proof of Proposition \ref{prop:liftable}; 
in this setting, the first column are the elements of the 	``Lift" set and the second column are the 
elements of the ``Swap" set. 
\end{example}

We conclude this section by remarking
 that the combinatorial interpretation for the parametrization of the Fourier coordinates
described in this section relies upon having a model with only two states.
One starting point for future work towards applying the molecular clock condition
to models with three or more states may be to give a comparable combinatorial description
of the non-zero Fourier coordinates in these models. 


\section{The CFN-MC Polytope} \label{sec:polytope}

In this section we give a description of the combinatorial
structure of the polytope associated to the CFN-MC model.
In particular, we show that the number of vertices of the
CFN-MC polytope is a Fibonacci number and we give a
complete facet description of the polytope.  One interesting
feature of these polytopes is that while the facet structure
varies widely depending on the structure of the tree (e.g.~some
trees with $n$ leaves have exponentially many facets, while others
only have linearly many facets), the number of vertices is
fixed.  Similarly, we will see in Section \ref{sec:ehrhart} that
the volume also does not depend on the number of leaves.

Let $T$ be a rooted binary tree on $n$ leaves. For any path system 
$\P$ in $T$, let $\bx^{\P} \in \R^{n-1}$ have $i$th component $x^{\P}_i = 1$ 
if $i$ is the highest internal node in some path in $\P$ and $x^{\P}_i = 0$ otherwise. 
Hence $\bx^{\P}$ is the top-vector of $\P$ as discussed in the
previous section.

\begin{definition}
Let $T$ be a rooted binary tree on $n$ leaves.
The \emph{CFN-MC polytope} $R_T$ is 
the convex hull of all $\bx^{\P}$ for $\P$ a path system in $T$. 
\end{definition}

\begin{example}
For the tree in Figure \ref{fig:treeex}, the polytope $R_T$ is the
 convex hull of the column vectors of the matrix $A_T$ in Example \ref{ex:parametrization}.
 \end{example}
 
 We note that the convex hull of the column vectors of $A_T$ is
 actually a subset of the hyperplane $\{ \bx \in \R^n \mid x_0 = 1\}$. To obtain $R_T$, we
 identify this hyperplane with $\R^{n-1}$ by deleting the first coordinate. We write
 $\text{conv}(A_T)$ to mean the convex hull of the column vectors of $A_T$ after we have
 deleted the first coordinate.

Recall that the \emph{Fibonacci numbers} are defined by the recurrence $F_n = F_{n-1} + F_{n-2}$ subject to initial conditions $F_0 = F_1 = 1$.

\begin{prop}\label{thm:fibonacci}
Let $T$ be a rooted binary tree on $n \geq 2$ leaves. The number of vertices of $R_T$ is $F_n$, 
the $n$-th Fibonacci number.
\end{prop}

\begin{proof}
We proceed by induction on $n$. For the base cases, we note that
if $T$ is the 2-leaf tree, then $R_T = \text{conv}\begin{bmatrix} 0 & 1 \end{bmatrix}$, and
if $T$ is the 3-leaf tree, then $R_T = \text{conv}\begin{bmatrix} 0 & 1 & 0 \\ 0 & 0 & 1 \end{bmatrix}$.

Let $T$ be an $n$-leaf tree with $n \geq 4$. Let $l_1$ and $l_2$ be 
leaves of $T$ that are adjacent to the same internal node $a$, so that $a$ is a cherry node. 
Leaves $l_1$ and $l_2$ exist because every rooted binary tree with $n \geq 2$ leaves has a cherry.

Let $T'$ be the tree obtained from $T$ by deleting leaves $l_1$ and $l_2$
and their adjacent edges so that $a$ becomes a leaf. 
If $\P$ is a path system
 in $T$ with $a \notin \text{Top}(\P)$, then we can realize 
the top-vector of $\P$ without the $a$-coordinate as the top-vector of a 
path system in $T'$. Furthermore, any path system in 
$T'$ can be extended to a path system in $T$ without $a$ in its top-set. 
So the number of vertices of $R_T$ with $a$-coordinate equal to $0$ 
is the number of vertices of $R_{T'}$, which is $F_{n-1}$ by induction.

Let $a'$ be the direct ancestor of $a$ in $T$. 
Let $T''$ be the tree obtained from $T$ by deleting $l_1, l_2$ and $a$, and all edges adjacent to $a$, and merging the two remaining edges incident to $a'$ so that $a'$ is no longer a node. 
In the case where $a'$ is the root of $T$,
we simply delete it and its other incident edge to form $T''$.
If $\P$ is a path system  
in $T$ with $a \in \text{Top}(\P)$, then note that $a' \notin \text{Top}(\P)$. 
Furthermore, edge $aa'$ is not an edge in any path in $\P$. Therefore, we can realize 
the top-vector of $\P$ without the $a$- and $a'$-coordinates as the top-vector 
of a path system in $T''$. Furthermore, any path system in $T''$ 
can be extended to a path system in $T$ with $a$ in its top-set. 
So the number of vertices of $R_T$ with $a$-coordinate equal to $1$ is the 
number of vertices of $R_{T''}$, which is $F_{n-2}$ by induction.

Therefore, the total number of vertices of $R_T$ is $F_{n-2} + F_{n-1} = F_n$, as needed.
\end{proof}

In order to give a facet description of the CFN-MC polytope of a tree, 
we define several intermediary polytopes between the CFN polytope and 
the CFN-MC polytope, along with linear maps between them. 
The CFN polytope is the analogue of the CFN-MC polytope for the CFN model;
 it is obtained by taking the convex hull of the indicator vectors of path systems $\P$
in $T$ indexed by the \emph{edges} in the path system.
We will
trace the known description of the facets of the CFN polytope through 
these linear maps via Fourier-Motzkin elimination to arrive at the facet 
description of the CFN-MC polytope.  See \cite[Chapter 1]{ziegler1995} for
background on Fourier-Motzkin elimination.

Let $T $ be a rooted binary tree with $n$ leaves, oriented with 
the root as the highest node and the leaves as the lowest nodes. 
Then $T$ has $n-1$ internal nodes.  The internal nodes will now
be labeled  by $1, \ldots , n-1$. 

\begin{remark}
For the remainder of the paper, we have changed the convention of labeling the 
nodes so that $1, \ldots , n-1$ label the internal nodes of the
trees, whereas in Section \ref{sec:CFNMC} we used $1, \ldots , n$ to denote the leaves.
This is because of the importance that the internal nodes now play in the
combinatorics of the CFN-MC model, whereas in Section \ref{sec:CFNMC} the leaves were
the main objects of interest in analyzing and simplifying the parametrization.
\end{remark}

Let $v$ be a non-root node  in $T$. Denote by $e(v)$ the unique edge that 
has $v = d(e(v))$.  
Introduce a poset $\Int(T)$
whose elements are the internal nodes of $T$ and with relations $v \leq w$ 
if $v$ is a descendant of $w$.
The Haase diagram of $\Int(T)$ is the tree $T$
with leaf and edge incident to a leaf removed.
Recall that an order ideal of $\Int(T)$ is a subset of $\Int(T)$
that is downwards closed.
Let $I$ be an order ideal of $\Int(T)$ with $s$ elements.   Then the number of 
edges not below an element of $I$ is $2(n-s-1)$, since $T$ 
has $2n-2$ edges and each node in $I$ has exactly two edges directly beneath it. 

\begin{definition}
Let $T$ be a tree, $\Int(T)$ the associated poset, and $I$ an order ideal of $\Int(T)$.
Denote by $T-I$ the tree obtained by removing all nodes and edges 
descended from any node in $I$. The \textit{edge set } of $T - I$, denoted
$\Ecal ( T - I)$, is the set of all edges in $T$ that are not descended 
from an element of $I$. Notice $T-I$ includes all maximal nodes of $I$, 
and all edges that join a node in $I$ with an internal node outside of $I$.
\end{definition}

\begin{definition}
Let $\P$ be a path systems between leaves of $T$. 
Let $[\bx, \by]^{\P}_I$ be the point in $\R^I \oplus \R^{ \Ecal(T - I)}$ defined by
\[
x_i = \begin{cases}
1 & \text{ if } i \text{ is the top-most node in some path in } \P \\
0 & \text{ otherwise, }
\end{cases}
\]
for all $i \in I$ and
\[
y_j = \begin{cases}
1 & \text{ if } e(j) \text{ is an edge in some path in } \P \\
0 & \text{ otherwise, } 
\end{cases}
\]
for all $e(j) \in \Ecal(T - I)$. The polytope $R_T(I)$ is the convex hull of 
all $[ \bx, \by ]^{\P}_I$ for all path systems 
$\P$ in $T$. If $I$ is the set of all internal nodes 
of $T$, then this polytope is exactly $R_T$, and $[ \bx, \by ]^{\P}_I = \bx^{\P} = [\P]$.
\end{definition}

\begin{example} \label{ex:4leaf} Let $T$ be the 4-leaf tree pictured in Figure \ref{fig:4leafexample}.
Let the distinguished order ideal in the set of internal nodes of $T$ be $I = \{ 3 \}$. Then $R_T(I)$ is the convex hull of the following 6 vertices with coordinates corresponding to the labeled edge or node.

\[
\begin{blockarray}{ccccccc}
\begin{block}{c(cccccc)}
e(2) & 0 & 1 & 1 & 0 & 0 & 0 \\
e(3) & 0 & 1 & 1 & 0 & 0 & 0 \\
e(4) & 0 & 1 & 0 & 1 & 0 & 1 \\
e(5) & 0 & 0 & 1 & 1 & 0 & 1 \\
3 & 0 & 0 & 0 & 0 & 1 & 1 \\
\end{block}
\end{blockarray}
\]

The dotted lines in Figure \ref{fig:equivalentpaths} shows the paths through $T$ that realize the vertex $\begin{bmatrix} 1 & 1 & 1 & 0 & 0 \end{bmatrix}^{\text{T}}$ in $R_T(I)$.

\begin{figure}
\begin{subfigure}{.3\textwidth}
\begin{tikzpicture}[scale=.87]
\draw(0,0)--(2.5,2.5)--(5,0);
\draw(1,1)--(2,0);
\draw(4,1)--(3,0);
\draw[fill](0,0) circle [radius=.05];
\draw[fill](2,0) circle [radius=.05];
\draw[fill](3,0) circle [radius=.05];
\draw[fill](5,0) circle [radius=.05];
\draw[fill](2.5,2.5) circle [radius=.05];
\node[above] at (2.5,2.5) {1};
\draw[fill](1,1) circle [radius=.05];
\node[left] at (1,1) {2};
\draw[fill](4,1) circle [radius=.05];
\node[right] at (4,1) {3};
\draw[fill](4,1) circle [radius=.05];
\node[below] at (0,0) {4};
\node[below] at (2,0) {5};
\node[below] at (3,0) {6};
\node[below] at (5,0) {7};
\node[left] at (2.25,2.25) {$e(2)$};
\node[right] at (2.75,2.25) {$e(3)$};
\node[left] at (.5,.5) {$e(4)$};
\node[right] at (1.5,.5) {$e(5)$};
\node[left] at (3.5,.5) {$e(6)$};
\node[right] at (4.5,.5) {$e(7)$};
\end{tikzpicture}
\caption{Leaf and edge labels in the tree $T$}\label{fig:4leaf}
\end{subfigure} \hfill \begin{subfigure}{.6\textwidth}
\begin{center}
\begin{tikzpicture}[scale=.87]
\draw[dashed] (0,0)--(2.5,2.5)--(5,0);
\draw(1,1)--(2,0);
\draw(4,1)--(3,0);
\draw[fill](0,0) circle [radius=.05];
\draw[fill](2,0) circle [radius=.05];
\draw[fill](3,0) circle [radius=.05];
\draw[fill](5,0) circle [radius=.05];
\draw[fill](2.5,2.5) circle [radius=.05];
\draw[fill](1,1) circle [radius=.05];
\draw[fill](4,1) circle [radius=.05];
\draw[fill](4,1) circle [radius=.05];
\node[above] at (2.5,2.5){\quad};
\node[below] at (0,0) {\quad};
\end{tikzpicture}  \begin{tikzpicture}[scale=.87]
\draw[dashed] (0,0)--(2.5,2.5)--(4,1);
\draw[dashed](4,1)--(3,0);
\draw(4,1)--(5,0);
\draw(1,1)--(2,0);
\draw[fill](0,0) circle [radius=.05];
\draw[fill](2,0) circle [radius=.05];
\node[below] at (2,0) {};
\draw[fill](3,0) circle [radius=.05];
\draw[fill](5,0) circle [radius=.05];
\draw[fill](2.5,2.5) circle [radius=.05];
\draw[fill](1,1) circle [radius=.05];
\draw[fill](4,1) circle [radius=.05];
\draw[fill](4,1) circle [radius=.05];
\node[below] at (0,0) { };
\end{tikzpicture}
\end{center}
\caption{Two paths in $T$ that correspond to the same vertex of $R_T(I)$.}\label{fig:equivalentpaths}
\end{subfigure}
\caption{The four-leaf tree in Example \ref{ex:4leaf}}\label{fig:4leafexample}
\end{figure}
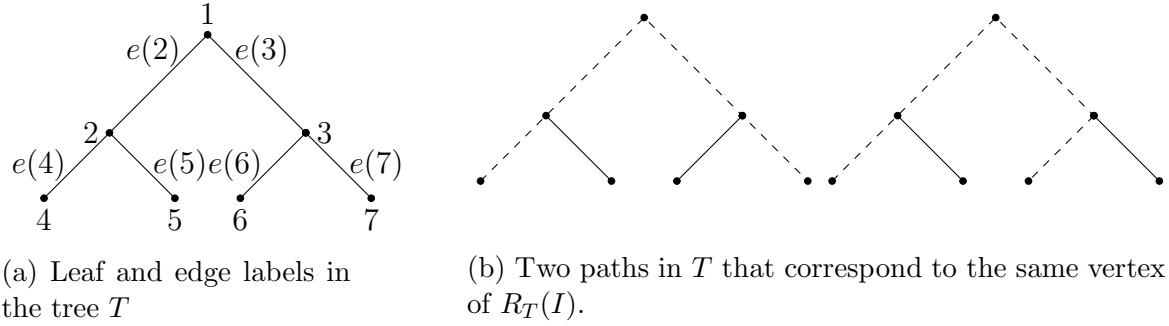

\end{example}

For any order ideal $I$ with maximal node $r$, let $a$ and $b$ be the 
direct
descendants of $r$. This scenario is pictured in Figure \ref{fig:localfoldingpicture}. Then we can define a linear map
\[
\phi_{I, r} : \R^{ ( I - \{r \} ) }\oplus \R^{  \Ecal (T - (I - \{ r\} ) ) } 
\rightarrow \R^{I} \oplus \R^{ \Ecal(T - I)},
\]
sending $\phi_{I,r} \big( (\bx ' , \by ') \big) = (\bx, \by)$ where
\[
\begin{cases}
x_i = x_i' & \text{ if } i \in I - \{r\} \\
y_j = y_j' & \text{ if } e(j) \in \Ecal (T - I) \\
x_r = \frac{-y_r ' + y_a' + y_b'}{2}.

\end{cases}
\]
Note that if $r$ is the root, then $y_r '$ is undefined. So we interpret the formula for $x_r$
as if $y_r' = 0$, and we have
$x_r = \frac{  y_a ' + y_b '}{ 2}$. But in this case, $R_{I - \{ r \} }$ lies in 
the hyperplane defined by $y_a ' = y_b '$. So $x_r = y_a' = y_b'$.
Here, $\bx '$ has elements indexed by nodes 
in $I - \{ r \}$ and $\by '$ has elements indexed by nodes in 
$\Ecal (T - (I - \{ r \} ) )$. Then $\bx$ has elements 
indexed by nodes in $I$ and $\by$ has elements index by nodes in $\Ecal(T - I)$. 

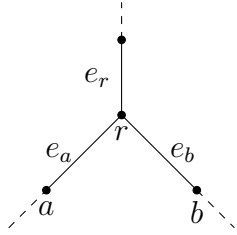
\begin{figure}[h]
\centering
\begin{tikzpicture}
\draw(1,1) -- (2,2);
\draw (2,2) -- (3,1);
\draw (2,2) -- (2,3);
\draw [dashed] (.5, .5) -- (1,1);
\draw[dashed] (3,1) -- (3.5, .5);
\draw[dashed] (2,3) -- (2, 3.5);
\draw[fill] (1,1) circle [radius = .05];
\node[below] at (1,1) {$a$};
\draw[fill] (2,2) circle [radius = .05];
\node[below] at (2,2) {$r$};
\draw[fill] (3,1) circle [radius = .05];
\node[below] at (3,1) {$b$};
\draw[fill] (2,3) circle [radius = .05];
\node[left] at (1.5, 1.5) {$e_a$};
\node[right] at (2.5, 1.5){$e_b$};
\node[left] at (2, 2.5) {$e_r$};
\end{tikzpicture}
\caption{The edges and nodes surrounding node $r$.} \label{fig:localfoldingpicture}
\end{figure}

\begin{prop}
The function $\phi_{I,r}$ maps  $R_T(I - \{ r \})$ onto $R_T(I)$.
\end{prop}

\begin{proof}
We show that for all path systems  $\P$ in $T$, 
the image of $[ \bx ', \by ']^{\P}_{I - \{ r \}}$ under $\phi_{I,r}$ 
is $[\bx, \by]^{\P}_I$. If $r$ 
is a node in a path in $\P$, then the path includes exactly two edges 
about $r$. So we have the following cases.

Case 1: Suppose that $y_a' = y_b' = y_r' = 0$. Then $e(a), e(b)$ and $e(r)$ 
are not edges in any path in $\P$, so $r$ cannot be the highest node in any 
path in $\P$. After applying $\phi_{I,r}$, we have $x_r = 0$, $x_i = x_i'$ 
for all $i \in I - \{r \}$, and $y_j = y_j'$ for all $e(j) \in \Ecal(T - I)$. 
So, the image of $[ \bx ', \by ']^{\P}_{I - \{ r \}}$ under $\phi_{I,r}$ 
is $[\bx, \by]^{\P}_I$ in this case.

Case 2: Suppose that $y_a' = y_r' =1$ and $y_b' = 0$. In this case, 
the path in $\P$ containing $r$ passes through $r$ along $e(a)$ and 
then upwards out of $r$ along $e(r)$. So, $r$ is not the highest node in this path. 
Since all paths in $\P$ are disjoint, $r$ is not the highest 
node in \textit{any} path in $\P$. Applying $\phi_{I,r}$ gives
\[
x_r = \frac{y_a' + y_b' - y_r'}{2} = 0,
\]
as needed. So, the image of $[ \bx ', \by ']^{\P}_{I - \{ r \}}$ 
under $\phi_{I,r}$ is $[\bx, \by]^{\P}_I$ in this case. The case 
where $y_b' = y_r ' = 1$ and $y_a' = 0$ is analogous.

Case 3: Suppose that $y_a'  = y_b' = 1$ and $y_r' = 0$. In this case, 
the path in $\P$ containing $r$ comes up to $r$ along $e(a)$ and then 
back downwards along $e(b)$. So, $r$ is the highest node in this path. 
Applying $\phi_{I,r}$ gives
\[
x_r = \frac{y_a' + y_b' - y_r'}{2} = 1,
\]
as needed. So, the image of $[ \bx ', \by ']^{\P}_{I - \{ r \}}$ 
under $\phi_{I,r}$ is $[\bx, \by]^{\P}_I$ in this case.

So, every vertex of $R_T(I - \{r \})$ maps to a vertex 
of $R_T(I)$ under $\phi_{I,r}$. Furthermore, every vertex $[\bx, \by]^{\P}_I$ of $R_T(I)$ 
 is the image of $[\bx ', \by']^{\P}_{I - \{ r \}}$. 
The result holds by linearity of the map $\phi_{I,r}$.
\end{proof}

\begin{definition}
Let $I$ be an order ideal in the poset consisting of all internal 
nodes of $T$.  A node $v \in I$ is called a \textit{cluster node} if 
$v$ is connected by edges to three other internal nodes.  A connected set 
of cluster nodes of $T$ is called a \textit{cluster}. 
Given a cluster $C \subseteq I$, 
$N_I(C)$  denotes the \textit{neighbor set} of $C$, which is the set of 
all internal nodes of $T$ that lie in $I - C$ and are 
adjacent to some node in $C$. When $I$ is the set of all 
internal nodes of $T$, we denote the neighbor set by $N(C)$. 
Denote by $m(C)$ the maximal node of $C$.
\end{definition}

Note that the maximal node of a cluster always exists since the cluster is a connected
subset of the rooted tree $T$.

\begin{example}\label{ex:cluster}
Consider the tree $T$ in Figure \ref{fig:cluster}. Then the set of nodes marked 
with triangles, $\{b,c\}$ forms a cluster since $b$ and $c$ are both cluster nodes and are adjacent. 
The neighbor set of this cluster, $N(\{b,c\})$, is the set of nodes marked with squares, $\{ a, d, e, f\}$.
The maximal element is $m( \{b,c\}) = b$.
\end{example}

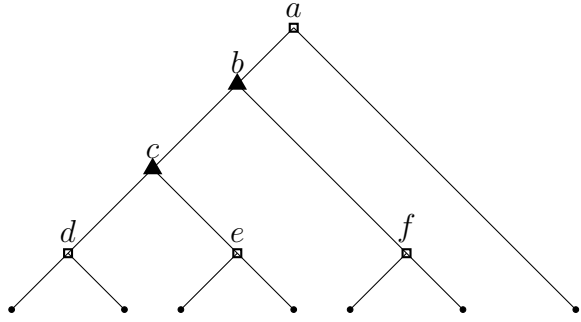
\begin{figure}
\begin{center}
\begin{tikzpicture}[scale=.75]
\draw (0,0) -- (5,5);
\draw(1,1) -- (2,0);
\draw(2.5,2.5) -- (5,0);
\draw (4,1) -- (3,0);
\draw (4,4) -- (8,0);
\draw(7,1) -- (6,0);
\draw(5,5) -- (10,0);
\draw[fill] (0,0) circle [radius = .05];
\draw[fill] (2,0) circle [radius = .05];
\draw[fill] (3,0) circle [radius = .05];
\draw[fill] (5,0) circle [radius = .05];
\draw[fill] (6,0) circle [radius = .05];
\draw[fill] (8,0) circle [radius = .05];
\draw[fill] (10,0) circle [radius = .05];
\draw[thick] ([xshift=-2pt,yshift=-2pt]1,1) rectangle ++(4pt,4pt);
\node [above] at (1,1) {$d$};
\draw [thick] ([xshift=-2pt,yshift=-2pt]4,1) rectangle ++(4pt,4pt);
\node[above] at (4,1) {$e$};
\draw[thick] ([xshift=-2pt,yshift=-2pt]7,1) rectangle ++(4pt,4pt);
\node [above] at (7,1) {$f$};
\draw[thick] ([xshift=-2pt,yshift=-2pt]5,5) rectangle ++(4pt,4pt);
\node[above] at (5,5) {$a$};
\node[fill=black,regular polygon, regular polygon sides=3,inner sep=1.5pt] at (2.5,2.5) {};
\node[above] at (2.5,2.5) {$c$};
\node[fill=black,regular polygon, regular polygon sides=3,inner sep=1.5pt] at (4,4) {};
\node [above] at (4,4) {$b$};
\end{tikzpicture}
\end{center}
\caption{An example of a cluster ${b,c}$, which are marked with triangles, and the elements of their neighbor set, which are marked with squares.} \label{fig:cluster}
\end{figure}

The main result of this section is Corollary \ref{cor:finalfacets}, which
gives a list of the facet defining inequalities of the polytopes $R_T$.
This result is obtained by proving the following more general results for the
polytopes $R_T(I)$.  This facet description depends on the underlying 
structure of the clusters in $T$.

\begin{thm} \label{thm:facets}
The polytope $R_T(I)$ is the solution to the following set of constraints:
\begin{itemize}[label=\raisebox{0.25ex}{\tiny$\bullet$}]
\item $y_s = y_t$, where edges $e(s)$ and $e(t)$ are joined to the root.
 \item $-y_i \leq 0$, $i$ maximal in $I$
\item $y_i - y_j - y_k \leq 0$, where $e(i), e(j), e(k)$ are three distinct edges that meet at a single node not in $I$,
\item $y_i + y_j + y_k \leq 2$, where $e(i), e(j), e(k)$ are three distinct edges that meet at a single node not in $I$,
\item $-x_i \leq 0$, for all $i \in I$,
\item $x_i + x_j \leq 1$ for all $i,j \in I$ with $i$ and $j$ adjacent
\item $x_i + y_i \leq 1$ for $i$ maximal in $I$
\item $2 \sum_{i \in C} x_i + \sum_{j \in N_I(C)} x_j + y_{m(C)} \leq \vert C \vert + 1 $ for all clusters $C \subset I$. 
\end{itemize}
\end{thm}

Note that if $m(C)$ is not a maximal node of $I$, then there is no coordinate $y_{m(C)}$.  
In this case, the final cluster inequality in Theorem \ref{thm:facets} reduces to
\[
2 \sum_{i \in C} x_i + \sum_{j \in N_I(C)} x_j  \leq \vert C \vert + 1.
\]
Note that we have chosen to write all of our inequalities with all
indeterminates on the left side and using all $\leq$ inequalities, as this will
facilitate our proof of Theorem \ref{thm:facets}.

\begin{example} \label{ex:facetexample}
Consider the tree $T$ in Figure \ref{fig:facetexample}. Note that the only cluster in $T$ is $\{ c \}$. Let $I \subset \Int(T)$ be the order ideal $\{b,c,d,e\}$. Then $R_T(I)$ lies in the hyperplane $y_b = y_f$ and has facets:

\[
\begin{array}{rrr}
y_f - y_g - y_h \leq 0, & \quad \quad &  x_b + x_c \leq 1,\\
-y_f + y_g - y_h \leq 0, &  &  x_c + x_d \leq 1, \\
-y_f - y_g + y_h \leq 0,&   & x_c + x_e \leq 1,\\
y_f + y_g + y_h \leq 2, &  & x_b + y_b \leq 1, \\
 x_b + 2x_c + x_d + x_e \leq 2 & &
\end{array}
\]
and $-x_i \leq 0$ for all $i \in I$.
\end{example}

\begin{figure}
\begin{center}
\begin{tikzpicture}[scale=.7]
\draw (0,0)--(5,5)--(10,0);
\draw (1,1)--(2,0);
\draw (2.5,2.5)--(5,0);
\draw(4,1)--(3,0);
\draw(3.5,3.5)--(7,0);
\draw(9,1)--(8,0);
\draw[fill] (0,0) circle [radius=.05];
\draw[fill] (1,1) circle [radius=.05];
\draw[fill] (2,0) circle [radius=.05];
\draw[fill] (2.5,2.5) circle [radius=.05];
\draw[fill] (4,1) circle [radius=.05];
\draw[fill] (3,0) circle [radius=.05];
\draw[fill] (5,0) circle [radius=.05];
\draw[fill] (5,5) circle [radius=.05];
\draw[fill] (3.5,3.5) circle [radius=.05];
\draw[fill] (7,0) circle [radius=.05];
\draw[fill] (9,1) circle [radius=.05];
\draw[fill] (8,0) circle [radius=.05];
\draw[fill] (10,0) circle [radius=.05];
\node[left] at (5,5) {$a$};
\node[left] at (3.5,3.5){$b$};
\node[left] at (2.5,2.5){$c$};
\node[left] at (1,1) {$d$};
\node[right] at(4,1) {$e$};
\node[right] at (9,1){$f$};
\node[left] at (8,0) {$g$};
\node[right] at (10,0){$h$};
\end{tikzpicture}
\end{center}
\caption{The tree $T$ in Example \ref{ex:facetexample}} \label{fig:facetexample}
\end{figure}
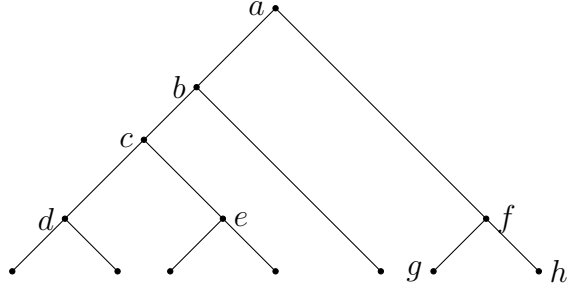

\begin{proof}[Proof of Theorem \ref{thm:facets}]
We proceed by induction on the size of the order ideal $I$.

When $\vert I \vert = 0$, $R_T({\emptyset})$ is the polytope associated to the CFN model,
 as described in \cite{sturmfels2005}. 
It follows from the results in \cite{ buczynska2007, sturmfels2005} 
that $R_T(\emptyset)$ has facets defined by $y_i - y_j - y_k \leq 0$ and 
$y_i + y_j + y_k \leq 2$ for all distinct $i,j,k$ such that 
$e(i), e(j),$ and $e(k)$ that meet at the same internal node.

Let $\vert I \vert \geq 1$ and let $1 , \dots , r$ be the maximal nodes of 
$I$. Suppose that $R_T({I - \{ r \}})$ has its facets defined by the proposed inequalities. About node $r$, we have the edges and  nodes depicted in Figure \ref{fig:localfoldingpicture}. Note that it is possible that $a$, $b$ or both are leaves. In the case that $a$ is a leaf, the inequalities below in which $x_a'$ is a term would not exist, and similarly for $b$ and $x'_b$.

We use Fourier-Motzkin elimination along with the linear map $\phi_{I,r}$ 
to show that the facets of $R_T(I)$ are defined by a subset of the proposed inequalities.

Recall that primed coordinates such as $y_a'$
indicate coordinates of $R_T(I - \{r\})$. 
In order to project $R_T({I - \{ r \} })$ onto $R_T(I)$, we 
``contract" onto $r$ by replacing $y_b'$ with $2x_r + y_r' - y_a'$, since under $\phi_{I,r}$,
\[
x_r = \frac{-y_r' + y_a' + y_b'}{2}.
\]
Then we use Fourier-Motzkin elimination to project out $y_a'$.

By the inductive hypothesis, the following are the inequalities in 
$R_T({I - \{ r \}})$ that involve $y_a'$ or $y_b'$. Note that these are 
the only types of inequalities that we need to consider, since any inequalities 
not involving $y_a'$ or $y_b'$ remain unchanged by Fourier-Motzkin elimination.

\begin{minipage}{.4\textwidth}
\begin{align*}
- y_a' &\leq 0, \\
- y_b' &\leq 0,\\
x_a ' + y_a' &\leq 1,\\
x_b' + y_b' &\leq 1,\\
2 \sum_{i \in C} x_i' + \sum_{j \in N_{I - \{ r \} }(C)} x_{j}' +  y_a' &\leq \vert C \vert +1,
\end{align*}
\end{minipage}\hfill\begin{minipage}{.4\textwidth}
\begin{align*}
y_a' - y_b' - y_r' & \leq 0,\\
-y_a' + y_b' - y_r' &\leq 0,\\
-y_a' - y_b' + y_r' &\leq 0,\\
y_a' + y_b' + y_r'& \leq 2,\\
2 \sum_{i \in D} x_i' + \sum_{j \in N_{I - \{ r \} }(D)} x_{j}' +  y_b' &\leq \vert D \vert +1,
\end{align*}
\end{minipage}
where $C$ ranges over all clusters that contain $a$ and are contained in the subtree beneath $a$, and $D$ ranges overall clusters that contain $b$ and are contained in the subtree beneath $b$. The same is true of $C$ and $D$ throughout the following discussion. Note that if $a$ (resp. $b$) is not a cluster node, then no such $C$ (resp. $D$) exists.

Applying $\phi_{I,r}$ yields the following inequalities, labeled by 
whether the coefficient of $y_a$ is positive or negative in order to facilitate Fourier-Motzkin elimination.

\begin{align}
-x_r & \leq 0 \tag{0}\\
 x_r + y_r & \leq 1 \tag{00} \\
y_a - y_r -2 x_r & \leq  0 \tag{$1_+$}\\
y_a + x_a &\leq  1 \tag{$2_+$}\\
y_a - y_r - x_r &\leq 0 \tag{$3_+$}\\
y_a + 2 \sum_{i \in C} x_i + \sum_{j \in N_{I - \{ r \} }(C)} x_{j} &\leq \vert C \vert +1 \tag{$4_+$}\\
-y_a & \leq 0 \tag{$1_-$}\\
-y_a + y_r + x_b + 2 x_r &\leq 1 \tag{$2_-$}\\
-y_a + x_r &\leq 0 \tag{$3_-$}\\
-y_a + y_r + 2 x_r + 2\sum_{i \in D} x_i + \sum_{j \in N_{I - \{ r \} }(D)} x_j & \leq \vert D \vert + 1  \tag{$4_-$}
\end{align}

If, without loss of generality, $a$ is an internal node and $b$ is a leaf, then inequalities $2_-$ and $4_-$ do not exist. If both $a$ and $b$ are leaves, then inequalities $2_+$, $4_+$, $2_-$ and $4_-$ do not exist.

We perform Fourier-Motzkin elimination to obtain the following
 $17$ types of inequalities, labeled by which of the above 
 inequalities where combined to obtain them. 
 The inequalities from $R_T({I -  \{ r \} })$ that did not contain $y_a'$ or $y_b'$
 also remain facet-defining inequliaties for $R_T(I)$.

\begin{align}
-x_r & \leq 0 \tag{0}\\
x_r + y_4 & \leq 1 \tag{00}\\
-2x_r - y_r & \leq 0 \tag{$1_+ 1_-$} \label{eqn8} \\
x_b & \leq 1 \tag{$1_+ 2_-$} \label{eqn3} \\
-x_r - y_r & \leq 0 \tag{$1_+ 3_-$} \label{eqn4} \\
2 \sum_{i \in D} x_i+ \sum_{j \in N_{I - \{ r \} }(D)} x_{j}&\leq \vert D \vert +1 \tag{$1_+ 4_-$} \label{eqn15} \\
x_a & \leq 1 \tag{$2_+ 1_-$} \label{eqn2} \\
2x_r + x_a + x_b + y_r & \leq 2 \tag{$2_+ 2_-$} \label{eqn9} \\
x_r + x_a & \leq 1 \tag{$2_+ 3_-$} \label{eqn7} \\
2\sum_{i \in D \cup \{ r \} } x_i + y_r + \sum_{j \in N_I(D \cup \{ r \})} x_j & \leq \vert D \vert +2 \tag{$2_+ 4_-$} \label{eqn14} \\
-x_r - y_r & \leq 0 \tag{$3_+ 1_-$} \label{eqn5} \\
x_r + x_b & \leq 1 \tag{$3_+ 2_-$} \label{eqn6} \\
-y_r & \leq 0 \tag{$3_+ 3_-$} \label{eqn1} \\
2\sum_{i \in D} x_i + \sum_{j \in N_I(D)} x_j  & \leq \vert D \vert +1 \tag{$3_+ 4_-$}\label{eqn13} \\
2 \sum_{i \in C} x_i + \sum_{j \in N_{I - \{ r \} }(C)} x_{j} &\leq \vert C \vert +1 \tag{$4_+ 1_-$} \label{eqn10} \\
2\sum_{i \in C \cup \{ r \}} x_i + y_r + \sum_{j \in N_I(C \cup \{ r \})} x_j & \leq \vert C \vert +2 \tag{$4_+ 2_-$} \label{eqn11} \\
2 \sum_{i \in C} x_i + x_r \sum_{j \in N_I(C)} x_{j} & \leq \vert C \vert +1 \tag{$4_+ 3_-$} \label{eqn12} \\
2x_r + 2 \sum_{i \in C} x_i + 2\sum_{j \in D} x_j + y_r + \sum_{k \in N_{I - \{ r \} }(C)} x_k + \sum_{l \in N_{I - \{ r \} }(D)} x_l & \leq \vert C \vert + \vert D \vert + 2 \tag{$4_+ 4_-$} \label{eqn16}
\end{align}

The inequalities encompassed by \ref{eqn11} (resp. \ref{eqn14}) give the 
proposed inequalities for all clusters of size greater than or equal to 2 
that contain $r$ and for which all other nodes are contained in the 
$a$-subtree (resp. $b$-subtree). The inequalities given by \ref{eqn12} (resp. \ref{eqn15}) 
are all of the proposed inequalities for clusters containing $a$ (resp. $b$) 
and not $r$. Inequality \ref{eqn9} gives the inequality for the cluster 
$\{ r \}$. Finally, the inequalities given by \ref{eqn16} encompass all 
clusters with the highest node $r$ that contain nodes in both the $a-$ and $b$-subtrees. 

Note also that inequalities $1_+ 1_-$, $1_+ 2_-$, $1_+ 3_-$, $1_+ 4_-$, 
$2_+1_-$, $3_+ 1_-$, $3_+ 4_-$ and $4_+ 1_-$ are all redundant as they 
are positive linear combinations of other inequalities on the list. For instance, 
inequality $1_+ 1_-$ can be obtained by adding together two copies of inequality $0$
and $3_+ 3_-$. Inequality $1_+ 4_-$ can be 
obtained by adding together inequalities $3_+4_-$ and 0. 

Note that if, without loss of generality, $a$ is an internal node and $b$ is a leaf,
then the irredundant inequalities $2_+2_-$,$2_+4_-$, $3_+2_-$, $3_+4_-$, $4_+2_-$ and $4_+4_-$ would not exist.
This is because $b$ is not an internal node and $r$ is not a cluster node in this case.
Similarly, if $a$ and $b$ are both leaves,
then the irredundant inequalities $2_+2_-$, $2_+3_-$, $2_+4_-$,
$3_+2_-$,$3_+4_-$, $4_+2_-$ and $4_+4_-$ would not exist.

The remaining 
inequalities, along with the others that are unchanged because they 
did not involve $y_a'$ and $y_b'$ are exactly those that we claimed 
would result from contracting onto $r$, as needed.

\end{proof}

\begin{cor}\label{cor:finalfacets}
The facet-defining inequalities of $R_T$ are:
\begin{itemize}[label=\raisebox{0.25ex}{\tiny$\bullet$}]
\item $x_i \geq 0$, for all $1 \leq i \leq n-1$,
\item $x_i + x_j \leq 1$, for all pairs of adjacent nodes, $i$ and $j$, and
\item $2 \sum_{i \in C} x_i + \sum_{j \in N_T(C)} x_j \leq \vert C \vert +1$ for all clusters $C$ in $T$.
\end{itemize}
\end{cor}

\begin{proof}
Let $r$ be the root of $T$.
Since $r$ is not a descendent of any edge, 
we interpret the coordinate $y_r$ in $R_T$ to be zero.
The fact that the facet defining inequalities of $R_T$ are a subset of the inequalities given in Corollary \ref{cor:finalfacets}
along with the inequality $x_r \leq 1$
follows directly from Theorem \ref{thm:facets}, 
since $x_r \leq 1$ is obtained by setting $y_r = 0$ in $x_r + y_r \leq 1$.
Note further that $x_r \leq 1$ is a redundant inequality,
since for any node $a$ adjacent to $r$, this inequality can be
obtained by summing $x_r + x_a \leq 1$ and $-x_a \leq 0$.
So indeed, the facet defining inequalities of $R_T$ are a subset of the proposed inequalities.

Now we must show that none of the proposed inequalities are 
redundant. To do this, we find $n-1$ affinely independent 
vertices of $R_T$ that lie on each of the proposed facets.

For all facets of the form $\{ \bx \mid x_i = 0 \}$, the $0$ vector, 
along with each of the standard basis vectors $\be_j$ such that 
$j \neq i$ are $n-1$ affinely independent vertices that lie on the face.
So, $\{ \bx \mid x_i = 0 \}$ is a facet of $R_T$.

Consider a face of the form $F = \{ \bx \mid x_i + x_j = 1 \}$ 
where $i$ and $j$ are adjacent nodes of $T$. Without loss of generality, 
let $i$ be a descendant of $j$. First, note that $\be_i, \be_j \in F$.

Let $k \neq i,j$ be an internal node of $T$. If $k$ is 
not a node in the $i$-subtree, then $\be_i + \be_k \in F$, since either 
$k$ is in the subtree of $T$ rooted at the descendant of $j$ not equal 
to $i$, or $k$ lies above $j$. In the first case, since the $i$- and 
$k$-subtrees are disjoint, we may choose any paths with highest nodes 
$i$ and $k$, which yield the desired vertex. In the second case, picking 
a path with highest node $i$, and a path with highest node $k$ that passes 
through the descendant of $j$ not equal to $i$ yields that $\be_i + \be_k$ is 
a vertex of $R_T$. Similarly, for all $k$ in the $i$-subtree, $\be_j + \be_k \in F$. 
Since every standard basis vector is in the linear span of 
\[
\{\be_i, \be_j\} \cup \{ \be_i + \be_k \mid k \neq j, k \text{ not in the } i\text{-subtree} \} \cup \{ \be_j + \be_k \mid k \neq i, k \text{ in the } i\text{-subtree} \},
\]
these $n-1$ vectors are linearly independent.

Finally, consider a face of the form $ F = \{ \bx \mid 2 \sum_{c \in C} x_c + \sum_{i \in N(C)} x_i = \vert C \vert + 1 \}$, for some cluster $C$ in $T$. Then $\vert N(C) \vert = \vert C \vert + 2$. First note that $\bu_j = \sum_{i \in N(C)} \be_i - \be_j$ is a vertex of $R_T$ for all $j \in N(C)$. If $j$ is the highest node of $N(C)$, then the $i$-subtrees for $i \in N(C), i \neq j$ are disjoint. So any two paths with highest nodes $i, k \in N(C)$, $i \neq k \neq j$ are disjoint. If $j$ is not the highest node in $N(C)$, let $k$ be the highest node. Then we may use any paths with highest nodes $i$ for all $i \in N(C)$ with $i \neq j,k$, and then a path with highest node $k$ that passes through $j$. Since the path between $k$ and $j$ contains only $j,k$ and elements of $C$, this path does not pass through any $i$-subtree for $i \in N(C)$, $i \neq j,k$. So, these paths are disjoint, as needed. So, $\{ \bu_j \mid j \in N(C)\}$ is a set of $\vert C \vert +2$ linearly independent vertices of $R_T$ that lie on $F$.

For all $c \in C$, let $\bw_c = \be_c + \sum_{i \in A_c} \be_i$ where $A_c$ is a set of $\vert C \vert -1$ elements of $N(C)$ such that (1) if $i \in N(C)$ is adjacent to $c$, then $i \not\in A_c$, (2) there exist $i,j \not\in A_c$ that are in the left and right subtrees beneath $c$, respectively, and (3) if $i$ is the highest node in $N(C)$, then $i \not\in A_c$. Note that at least one such set exists for all $c \in C$. Then $\bw_c$ is a vertex of $R_T$ since it results from the path system containing a path with highest node $c$ that passes through $i$ and $j$, where $i,j$ are the descendants of $c$ not in $A_c$ that exist by condition (2), and a path with highest node $k$ for all $k \in A_c$. Furthermore, $\bw_c \in F$ for all $c \in C$.

Note that $\{ \bu_j \mid j \in N(C) \} \cup \{ \bw_c \mid c \in C \}$ is a linearly independent set, since $\{ \bu_j \mid j \in N(C) \} $ is a linearly independent set of vectors that have all coordinates corresponding to elements of $C$ equal to 0, and each $\bw_c$ has a unique nonzero coordinate corresponding to $c \in C$.

Let $k$ be an internal node of $T$ such that $k \not\in C \cup N(C)$. If $k$ is a descendant of $j$ for some $j \in N(C)$ that is not the highest node of $N(C)$, then $\bz_k = \bu_j + \be_k$ is a vertex of $R_T$ that lies on $F$. Otherwise, $k$ is either a descendant of only the highest node, $i$, of $N(C)$, or not a descendant of any element of $N(C)$. In either of these cases, $\bz_k = \bu_i + \be_k$ is a vertex of $R_T$ that lies on $F$.

Also, $\{ \bu_j \mid j \in N(C) \} \cup \{ \bw_c \mid c \in C \} \cup \{ \bz_k \mid k \not\in C \cup N(C)\}$ is a linearly independent set as each element of $\{ \bu_j \mid j \in N(C) \} \cup \{ \bw_c \mid c \in C \}$ has coordinates corresponding to nodes not in $C$ or $N(C)$ equal to 0, and each $\bz_k$ has a unique nonzero coordinate corresponding to $k \not\in N(C) \cup C$. This set also has cardinality $\vert C \vert + 2 + \vert C \vert + n - 2\vert C \vert - 3 = n-1$. So, since we have found $n-1$ linearly independent vertices of $R_T$ that lie on $F$, $F$ is a facet of $R_T$. 
\end{proof}

We conclude this section with the remark that the number of facets of $R_T$ varies widely for different tree topologies.
For a tree with $n$ leaves and no cluster nodes, there are $2n-3$ facets of $R_T$ corresponding to each non-negativity
condition and each of the facets arising from adjacent nodes.
In contrast, the following is an example of a construction of trees
with exponentially many facets.

\begin{example}\label{ex:exponentialfacets}
Let $m$ be a positive integer. We construct a tree $T_m$ with $4m+5$ leaves as follows. 
Begin with a path, or ``spine", of length $m$. 
To the top node of this spine, attach a single pendant leaf; this top node becomes the root of $T_m$. 
Attach a balanced 4-leaf tree descended from every node of the spine, 
with two attached to the node at the bottom of the spine.
There are $2m + 1$ cluster nodes in $T_m$: 
the nodes that are in the spine and the root of each of the balanced 4-leaf trees descended from the spine.
Figure \ref{fig:exponentialfacets} depicts this tree for $m=3$. 

Let $S$ be the set of all nodes in the spine, 
and let $A$ be any set of nodes immediately descended from a spine node.
Then $S\cup A$ is a cluster.
Clusters of this form account for $2^{m+1}$ facets of $R_T$ for 
this $(4m+5)$-leaf tree.

\begin{figure}
\begin{center}
\begin{tikzpicture}[scale=.3]
\draw(0,0)--(15,15)--(30,0);
\draw (1,1)--(2,0);
\draw(2.5,2.5)--(5,0);
\draw(4,1)--(3,0);
\draw(6,6)--(12,0);
\draw(9.5,2.5)--(7,0);
\draw(8,1)--(9,0);
\draw(11,1)--(10,0);
\draw(9.5,9.5)--(19,0);
\draw(16.5,2.5)--(14,0);
\draw(15,1)--(16,0);
\draw(18,1)--(17,0);
\draw(13,13)--(26,0);
\draw(23.5,2.5)--(21,0);
\draw(22,1)--(23,0);
\draw(25,1)--(24,0);
\draw[fill] (6,6) circle [radius = .2];
\draw[fill] (9.5,9.5) circle [radius = .2];
\draw[fill] (13,13) circle [radius = .2];
\end{tikzpicture}
\end{center}
\caption{The tree construction for $T_3$ described in Example \ref{ex:exponentialfacets}. ``Spine" nodes are marked with a circle.} \label{fig:exponentialfacets}
\end{figure}
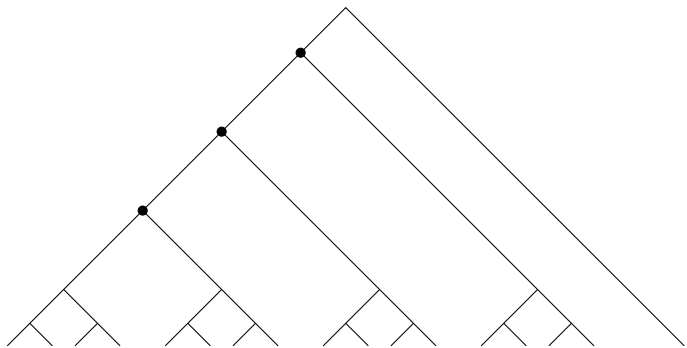

\end{example}


\section{Generators of the CFN-MC Ideal}

The aim of this section is to prove the following theorem.

\begin{thm}\label{thm:generatorsmain}
For any rooted binary phylogenetic tree $T$, the CFN-MC ideal $I_T$ has a Gr\"obner basis consisting of homogeneous quadratic binomials with squarefree initial terms.
\end{thm}

To accomplish this, we show that for most trees $T$, the CFN-MC ideal is the toric fiber product of the ideals of two smaller trees. In these cases, we can use results from \cite{sullivant2007} to describe the generators of $I_T$ in terms of the generators of the ideals of  these smaller trees. We then handle the case of trees for which $I_T$ is not a toric fiber product; such trees are called \textit{cluster trees}.

For simplicity of notation, we switch to denoting the top-vector associated to a path system $\P$ by $[\P]$. As before, note that it is possible to have two different path systems $\P$ and $\Q$ for which $[\P] = [\Q]$.  We often make use of the following notion of restriction of a path system to a subtree.

\begin{definition}
Let $T$ be a tree and let $T'$ be a subtree of $T$. Let $\P$ be a path system in $T$. Then the \emph{restriction} of $\P$ to $T'$ is the path system $\P'$ in $T'$ obtained by the following procedure for each path $P \in \P$. If the top-most node of $P$ is not in $\Int(T')$, delete $P$. Otherwise, intersect the edges of $P$ with the edges of $T'$ to obtain a path $P'$, and add $P'$ to $\P'$. 
\end{definition}

Note that if $\P'$ is the restriction of $\P$ to $T'$, then $[\P']$ is equal to $[\P]$ on each coordinate in $\Int(T')$.
\subsection{Toric Fiber Products} \label{sec:tfp}

Let $T$ be a tree that has an internal node $v$ that is adjacent to exactly two other internal nodes. There are two cases for the position of $v$ within $T$, both of which provide a natural way to divide $T$ into two smaller trees, $T'$ and $T''$. 

If $v$ is the root, then let $T'$ be the tree with $v$ as a root in which the right subtree of $v$ is equal to the right subtree of $T$ and the left subtree of $v$ is a single edge. Let $T''$ be the tree with $v$ as a root in which the left subtree of $v$ is equal to the left subtree of $T$ and the right subtree of $v$ is a single edge. This decomposition is pictured in Figure \ref{fig:tfp1}.

If $v$ is not the root, then $v$ is adjacent to two internal nodes and a leaf in $T$. Let $T'$ be the tree consisting of all non-descendants of $v$ (including $v$ itself) with a cherry added below $v$. Let $T''$ be the tree consisting of $v$ and all of its descendants. This decomposition is pictured in Figure \ref{fig:tfp2}.

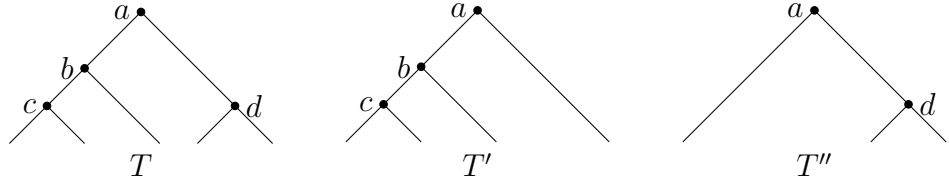
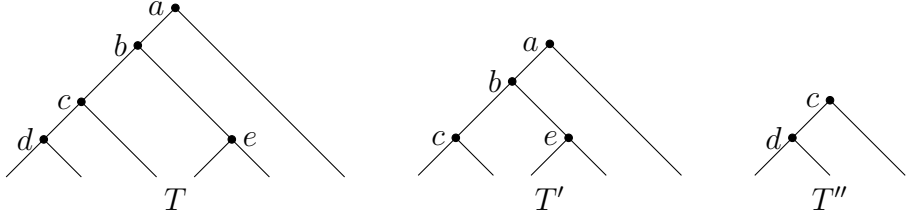
\begin{figure}

\begin{subfigure}{\textwidth}

\centering
\begin{tikzpicture}[scale=.5]
\draw (0,0)--(3.5,3.5)--(7,0);
\draw(1,1)--(2,0);
\draw(2,2)--(4,0);
\draw(6,1)--(5,0);
\draw[fill] (1,1) circle [radius=.1];
\node[left] at (1,1) {$c$};
\draw[fill] (2,2) circle [radius=.1];
\node[left] at (2,2) {$b$};
\draw[fill] (3.5,3.5) circle [radius=.1];
\node[left] at (3.5,3.5) {$a$};
\draw[fill] (6,1) circle [radius=.1];
\node[right] at (6,1) {$d$};
\node[below] at (3.5,0){$T$};
\end{tikzpicture} \qquad \begin{tikzpicture}[scale=.5]
\draw(0,0)--(3.5,3.5)--(7,0);
\draw(1,1)--(2,0);
\draw(2,2)--(4,0);
\draw[fill] (1,1) circle [radius=.1];
\node[left] at (1,1) {$c$};
\draw[fill] (2,2) circle [radius=.1];
\node[left] at (2,2) {$b$};
\draw[fill] (3.5,3.5) circle [radius=.1];
\node[left] at (3.5,3.5) {$a$};
\node[below] at (3.5,0){$T'$};
\end{tikzpicture} \qquad \begin{tikzpicture}[scale=.5]
\draw (0,0)--(3.5,3.5)--(7,0);
\draw(6,1)--(5,0);
\draw[fill] (3.5,3.5) circle [radius=.1];
\node[left] at (3.5,3.5) {$a$};
\draw[fill] (6,1) circle [radius=.1];
\node[right] at (6,1) {$d$};
\node[below] at (3.5,0){$T''$};
\end{tikzpicture}
\caption{Splitting $T$ into $T'$ and $T''$ where distinguished node ($a$) is the root.} \label{fig:tfp1}
\end{subfigure}
\begin{subfigure}{\textwidth}

\centering
\begin{tikzpicture}[scale=.5]
\draw (0,0)--(4.5,4.5)--(9,0);
\draw(3.5,3.5)--(7,0);
\draw(1,1)--(2,0);
\draw(2,2)--(4,0);
\draw(6,1)--(5,0);
\draw[fill] (4.5,4.5) circle [radius=.1];
\node[left] at (4.5,4.5) {$a$};
\draw[fill] (1,1) circle [radius=.1];
\node[left] at (1,1) {$d$};
\draw[fill] (2,2) circle [radius=.1];
\node[left] at (2,2) {$c$};
\draw[fill] (3.5,3.5) circle [radius=.1];
\node[left] at (3.5,3.5) {$b$};
\draw[fill] (6,1) circle [radius=.1];
\node[right] at (6,1) {$e$};
\node[below] at (4.5,0){$T$};
\end{tikzpicture} \qquad \begin{tikzpicture}[scale=.5]
\draw(0,0)--(3.5,3.5)--(7,0);
\draw(1,1)--(2,0);
\draw(2.5,2.5)--(5,0);
\draw(4,1)--(3,0);
\draw[fill] (1,1) circle [radius=.1];
\node[left] at (1,1) {$c$};
\draw[fill] (2.5,2.5) circle [radius=.1];
\node[left] at (2.5,2.5) {$b$};
\draw[fill] (3.5,3.5) circle [radius=.1];
\node[left] at (3.5,3.5) {$a$};
\draw[fill] (4,1) circle [radius=.1];
\node[left] at (4,1) {$e$};
\node[below] at (3.5,0){$T'$};
\end{tikzpicture} \qquad \begin{tikzpicture}[scale=.5]
\draw(0,0)--(2,2)--(4,0);
\draw(1,1)--(2,0);
\draw[fill] (1,1) circle [radius=.1];
\draw[fill](2,2) circle [radius=.1];
\node[left] at (1,1){$d$};
\node[left] at (2,2){$c$};
\node[below] at (2,0) {$T''$};
\end{tikzpicture}
\caption{Splitting $T$ into $T'$ and $T''$ where distinguished node ($c$) is not the root.} \label{fig:tfp2}
\end{subfigure}
\caption{Decomposition of $T$ into $T'$ and $T''$ via a node adjacent to exactly two internal nodes}
\end{figure}

In either case, notice that $[\P]$ is the top-vector of a path system in $T$ if and only if the restrictions of $[\P]$, $[\P']$ and $[\P'']$ to $T'$ and $T''$ respectively are top-vectors of path systems in $T'$ and $T''$ that agree on $v$. The following lemma makes this observation precise.

\begin{lemma}\label{lem:pathsintfp}
Suppose that $T$ has an internal node $v$ that is adjacent to exactly two other internal nodes. 
Let $T'$ and $T''$ be the induced trees defined above depending upon the position of $v$ within $T$.
Let $\P$ be a path system in $T'$ and let $\Rcal$ be a path system in $T''$. 
Then there exists a path system $\P \vee \Rcal$  in $T$ such that $[\P \vee \Rcal]_i = [\P]_i$ for each $i \in \Int(T')$ 
and $[\P \vee \Rcal]_j = [\Rcal]_j$ for each $j \in \Int(T'')$.
\end{lemma}

\begin{proof}
First, consider the case where $v$ is the root. 
Then $T'$ is the tree with root $v$ whose left subtree is equal to that of $T$ and whose right subtree is a single leaf.
Similarly, $T''$ is the tree with root $v$ whose right subtree is equal to that of $T$ and whose left subtree is a single leaf. Let $\P$ be a path system in $T'$ and let $\Rcal$ be a path system in $T''$.

 If $[\P]_v = [\Rcal]_v = 0$, then no path in $\P$ or $\Rcal$ passes through $v$. 
 So each path in $\P$ and $\Rcal$ is also a path in $T$.
So we let $\P \vee \Rcal = \P \cup \Rcal$,
where the edge set of each path is a subset of the edges of $T$.

If $[\P]_v = [\Rcal]_v = 1$, then let $\bar{P}$ be the path of $\P$ whose top-most node is $v$
and let $\bar{R}$ be the path of $\Rcal$ whose top-most node in $v$.
Let $\bar{PR}$ be the path in $T$ with edge set equal to that of $P$ 
on the left subtree of $T$ and that of $R$ on the right subtree of $T$.
This is a path in $T$ with top-most node $v$.
In this case, let $\P \vee \Rcal = (\P \cup \Rcal \cup \{\bar{PR}\}) \setminus \{ P, R\}$,
where the edge set of each path is a subset of the edges of $T$.

Now consider the case where $v$ is not the root. Let $\ell$ be the leaf of $T$ that is adjacent to $v$.
Then $T'$ consists of all non-descendants of $v$ and two leaves below $v$. One of these leaves is $\ell$;
let $m$ be the other leaf below $v$ in $T'$.
The tree $T''$ consists of $v$ and all of its descendants. Let $\P$ be a path system in $T'$ and let $\Rcal$ be a path system in $T''$.

First, suppose $[\P]_v = [\Rcal]_v = 0$. Since $[\P]_v = 0$, we may assume that any path
in $\P$ that passes through $v$ passes through $\ell$ and not $m$. Since $v$ is the root
of $T''$, and $[\Rcal]_v = 0$, no path in $\Rcal$ passes through $v$ or $\ell$.
So we let $\P \vee \Rcal = \P \cup \Rcal$,
where the edge set of each path is a subset of the edges of $T$.

Now suppose $[\P]_v = [\Rcal]_v = 1$. Then $\P$ contains the path
in $T'$ between leaves $m$ and $\ell$. So we let
$\P \vee \Rcal = (\P \cup \Rcal) \setminus \{ P \}$, where the 
edge set of each path is a subset of the edges of $T$.
\end{proof}

One implication of Lemma \ref{lem:pathsintfp} is that the matrix $A_T$ of $I_T$ can be obtained by pairing together all columns in $A_{T'}$ and $A_{T''}$ that agree on $v$, and consolidating the rows corresponding to $v$ and the homogenizing rows of ones from each. This translates exactly to the operation on toric ideals known as the \textit{toric fiber product}, which was introduced in \cite{sullivant2007}.

Let $I_T \subset \K[\ur]$, $I_{T'} \subset \K[\ux]$, $I_{T''} \subset \K[\uy]$. Consider the map $\xi_{I_{T'},I_{T''}}$ from $\K[\ur]$ to $\K[\ux] \otimes_{\K} \K[\uy]$ defined as follows. For any path system $\P$ in $T$, let $\P'$ and $\P''$ be the restrictions of $\P$ to $T'$ and $T''$ respectively. Then
\[
\xi_{I_{T'},I_{T''}} (r_{[\P]}) = x_{[\P']} \otimes y_{[\P'']}.
\]
Following the notation of \cite{sullivant2007}, the kernel of $\xi_{I_{T'},I_{T''}}$ is the \textit{toric fiber product} $I_{T'} \times_{\A} I_{T''}$. Here, $\A$ is the matrix
\[
\A = \begin{bmatrix}
1 & 1 \\
0 & 1
\end{bmatrix},
\]
where the first row corresponds to the homogenizing row of ones, and the second row corresponds to the shared node $v$ of $T'$ and $T''$. By Lemma \ref{lem:pathsintfp}, we can join any path systems in $T'$ and $T''$ that whose top-vectors agree on $v$ to create a path system in $T$. 

\begin{prop}\label{prop:tfp}
Suppose that $T$ has an internal node $v$ that is adjacent to exactly two other internal nodes. Let $T'$ and $T''$ be the induced trees defined above depending upon the position of $v$ within $T$. Then $I_T  \cong I_{T'} \times_{\A} I_{T''}$.
\end{prop}

\begin{proof}
The monomial map of which $I_T$ is the kernel is given by
\begin{align*}
\psi_T : \ & k[\ur]  \rightarrow  k[t_0, \dots, t_{n-1}]\\
&r_{[\P]}  \mapsto  t_0 \prod_{[\P]_i =1} t_i
\end{align*}

Let $S$ be the two-leaf tree rooted at $v$, and let $\K[\uz] = \K[z_0, z_1]$ be its associated polynomial ring. Denote by $\bar{\P}$ the restriction of $\P$ to $S$, by $\P'$ the restriction of $\P$ to $T'$ and by $\P''$ the restriction of $\P$ to $T''$. Then we have the identity
\begin{equation}\label{eqn:tfpidentity}
\psi_{S}(z_{[\bar{\P}]}) \psi_T (r_{[\P]}) = \psi_{T'}(x_{[\P']}) \psi_{T''}(y_{[\P'']}).
\end{equation}

The map defining the toric fiber product $I_{T'} \times_{\A} I_{T''}$ can be written
\begin{align*}
\xi_{I_{T'}, I_{T''}}: \ &k[\ur] \rightarrow k[t_0, \dots, t_{n-1}] \\
& r_{[\P]} \mapsto \big(t_0 \prod_{[\P']_i = 1 } t_i \big) \big(t_0 \prod_{[\P'']_i = 1} t_i\big)
\end{align*}

Notice that $t_0$ and $t_v$ can only appear in the image of $\xi_{I_{T'}, I_{T''}}$ with exponent 2. Therefore we may replace these variables by their square roots in the formula for the image of each $r_{[\P]}$ in the map $\xi_{I_{T'}, I_{T''}}$ without changing the kernel. This yields the same map as $\psi_T$. Therefore,
\[
I_T = \ker \psi_T \cong \ker \xi_{I_{T'}, I_{T''}} = I_{T'} \times_{\A} I_{T''}. \qedhere
\]\end{proof}

Let $\G_1$ be a Gr\"obner basis for $I_{T'}$ with weight 
vector $\omega_1$, and let $\G_2$ be a Gr\"obner basis for 
$I_{T''}$ with weight vector $\omega_2$. From these, we 
define several sets of polynomials in $\K[\ur]$ that together 
form a Gr\"obner basis for $I_T$ with respect to some weighted monomial order. 
Let $\P_1,\dots,\P_d$ and $\Q_1,\dots, \Q_d$ be path systems in $T'$
such that $f = \prod_{i=1}^d x_{[\P_i]} - \prod_{i=1}^d x_{[\Q_i]} \in \G_1$.
We arrange these  so that $[\P_i]_v = [\Q_i]_v$ for all $i$.
 Note that $f$ can always be written in this form
since the parameter $t_v$ must appear with the same power
in the image of each monomial under $\psi_{T'}$ in order for $f$ to be in its kernel. 
Let $R(f)$ denote the set of all $d$-tuples $(\Rcal_1', \ldots, \Rcal_d')$
of path systems in $T''$ such that $[\Rcal_i]_v = [\P_i]_v$.

By Lemma \ref{lem:pathsintfp}, for any path systems $\P$ in $T'$ and $\Rcal$ in $T''$
with $[\P]_v = [\Rcal]_v$, we can find a path system $\P \vee \Rcal$ in $T$ such that
$[\P \vee \Rcal]_i = [\P]_i$ for each $i \in \Int(T')$ and $[\P \vee \Rcal]_j = [\Rcal]_j$ for
each $j \in \Int(T'')$.

 Define the set
\[
\Lift(f) = \left \{ \prod_{i=1}^d r_{[\P_i \vee \Rcal_i]} - \prod_{i=1}^d r_{[\Q_i \vee \Rcal_i]}  :
(\Rcal_1, \ldots, \Rcal_d) \in R(f)
\right \}.
\]
 Then let
\[
\Lift(\G_1) = \cup_{f \in \G_1} \Lift(f),
\]
and similarly define $\Lift(\G_2)$. 

We now define another family of polynomials that is contained in the Gr\"obner basis for $I_T$.
Let $[\P_1], \dots, [\P_r]$ be the top-vectors of paths in $T'$ with $v$-coordinate 0 and let $[\Q_1], \dots, [\Q_s]$ be the top-vectors of paths in $T''$ with $v$-coordinate 0. (Note that these $\P_i$ and $\Q_i$ are unrelated to those in the previous paragraph). Define the set $\Quad_0(T)$ to be the set of all $2 \times 2$ minors of the matrix $M_0(T)$ with $(i,j)$th entry equal to $r_{[\P_i \vee \Q_j]}$. Define $\Quad_1(T)$ and $M_1(T)$ analogously over all top-vectors in $T'$ and $T''$ with $v$-coordinate equal to 1. Elements of $\Quad_k$ are of the form
\[
r_{[\P_i \vee \Q_j]}r_{[\P_{i'} \vee \Q_{j'}]} - r_{[\P_i \vee \Q_{j'}]} r_{[\P_{i'} \vee \Q_j]},
\]
where $[\P_i], [\Q_j], [\P_{i'}]$ and $[\Q_{j'}]$ all take value $k$ on their $v$-coordinate. Define
\[
\Quad(T) = \Quad_0(T) \cup \Quad_1(T).
\]

Let $\omega$ be a weight vector on $\K[\ur]$ so that $\Quad(T)$ is a Gr\"obner basis for the ideal generated by all elements of $\Quad(T)$. Such a weight vector exists by the proof of Proposition 2.6 in \cite{sullivant2007}. Since the $\A$-matrix of the toric fiber product is invertible, Theorem 2.9 of \cite{sullivant2007} implies the following proposition. Denote by $\xi_{I_{T'},I_{T''}}^*$ the pullback of $\xi_{I_{T'},I_{T''}}$. In other words, $\xi_{I_{T'},I_{T''}}^*$ is a map from the Cartesian products of the affine spaces associated to $\K[\ux]$ and $\K[\uy]$ to the affine space associated to $\K[\ur]$. If $\P$ and $\Q$ are path systems in $T'$ and $T''$ respectively whose top-vectors agree on $v$, then the $[\P \vee \Q]$ coordinate of $\xi_{I_{T'},I_{T''}}^*(\alpha, \beta)$ is $\alpha_{[\P]} + \beta_{[\Q]}$.

\begin{prop}\label{prop:tfpgrobner}
Suppose that $T$ has an internal node $v$ that is adjacent to exactly two other nodes. Let $T'$ and $T''$ be the induced trees defined above depending upon the position of $v$ within $T$. Then $\Lift(\G_1) \cup \Lift(\G_2) \cup \Quad(T)$ is a Gr\"obner basis for $I_T$ with respect to weight vector $\xi_{I_{T'},I_{T''}}^*(\omega_1, \omega_2) + \epsilon \omega$ for sufficiently small $\epsilon>0$.
\end{prop}

In particular, note that since the Lift operation preserves degree, and since the elements of $\Quad(T)$ are quadratic, if $\G_1$ and $\G_2$ consist of quadratic binomials, then $I_T$ has a Gr\"obner basis consisting of quadratic binomials.


\subsection{Cluster Trees}

Trees that do not have an internal node that is adjacent to 
exactly two other internal nodes do not have the toric fiber 
product structure described in the previous section. These are the 
trees whose internal nodes are comprised of one large cluster and its neighbor nodes. 
In this case, we exploit the toric fiber product structure of a subtree, 
and describe a method for lifting the Gr\"obner basis for the subtree 
to a Gr\"obner basis for the entire tree that maintains the degree of the Gr\"obner basis elements.

\begin{definition}
A rooted binary tree $T$ is called a \textit{cluster tree} if there exists a cluster $C$ in $T$ such that every internal node of $T$ is either in $C$ or in $N(C)$. Note that if $T$ is a cluster tree, then $C$ is uniquely determined by $T$.
\end{definition}

Equivalently, if $T$ has $n$ leaves, then $T$ is a cluster tree if and only if $T$ has a cluster of size $(n-3)/2$. Note that this implies that if $T$ is a cluster tree, then $T$ has an odd number of leaves. It also follows from the definition of a cluster tree that the root of $T$ must be adjacent on one side to a single leaf.

\begin{example}
The following tree is an example of a cluster tree with $\{\rho'\}$ as its distinguished cluster.
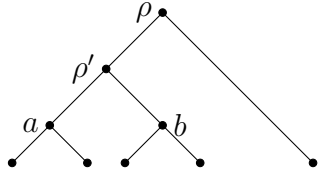
\begin{center}
\begin{tikzpicture}[scale=.5]
\draw(0,0)--(4,4);
\draw(1,1)--(2,0);
\draw(2.5,2.5)--(5,0);
\draw(4,1)--(3,0);
\draw(4,4)--(8,0);
\draw[fill](0,0) circle [radius=.1];
\draw[fill](2,0) circle [radius=.1];
\draw[fill](3,0) circle [radius=.1];
\draw[fill](5,0) circle [radius=.1];
\draw[fill](8,0) circle [radius=.1];
\draw[fill](2.5,2.5) circle [radius=.1];
\node[left] at (2.5,2.5) {$\rho'$};
\draw[fill](1,1) circle [radius=.1];
\node[left] at (1,1) {$a$};
\draw[fill](4,1) circle [radius=.1];
\node[right] at (4,1) {$b$};
\draw[fill](4,4) circle [radius=.1];
\node[left] at (4,4) {$\rho$};

\end{tikzpicture}
\end{center}
\end{example}

Let $T$ be a cluster tree with root $\rho$. Consider the tree $T'$ obtained from $T$ by deleting $\rho$ and its adjacent edges. Let $\rho'$ be the root of $T'$. Then by Proposition \ref{prop:tfp}, the CFN-MC ideal of $T'$, $I_{T'}$ is the toric fiber product $I_{U_1} \times_{\Acal} I_{U_2}$ where $U_1$ and $U_2$ are the cluster trees with root $\rho'$ and maximal clusters given by the left and right subtrees of $\rho'$ respectively. So we call $T'$ a \emph{bicluster tree}.

\begin{example}
From the previous example, $T'$, $U_1$ and $U_2$ are as follows.
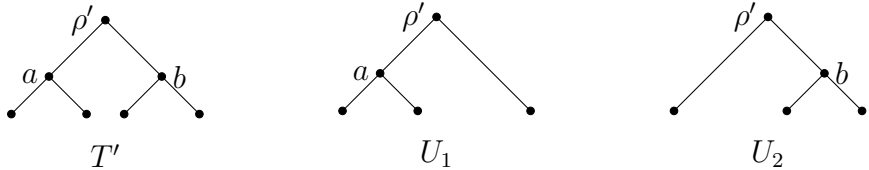
\begin{center}
\begin{tikzpicture}[scale=.5]
\draw(0,0)--(2.5,2.5)--(5,0);
\draw(1,1)--(2,0);
\draw(4,1)--(3,0);
\draw[fill](0,0) circle [radius=.1];
\draw[fill](2,0) circle [radius=.1];
\draw[fill](3,0) circle [radius=.1];
\draw[fill](5,0) circle [radius=.1];
\draw[fill](2.5,2.5) circle [radius=.1];
\node[left] at (2.5,2.5) {$\rho'$};
\draw[fill](1,1) circle [radius=.1];
\node[left] at (1,1) {$a$};
\draw[fill](4,1) circle [radius=.1];
\node[right] at (4,1) {$b$};
\node[below] at (2.5,-.5) {$T'$};
\end{tikzpicture}$\qquad \qquad$ \begin{tikzpicture}[scale=.5]
\draw(0,0)--(2.5,2.5)--(5,0);
\draw(1,1)--(2,0);
\draw[fill](0,0) circle [radius=.1];
\draw[fill](2,0) circle [radius=.1];
\draw[fill](5,0) circle [radius=.1];
\draw[fill](2.5,2.5) circle [radius=.1];
\node[left] at (2.5,2.5) {$\rho'$};
\draw[fill](1,1) circle [radius=.1];
\node[left] at (1,1) {$a$};
\node[below] at (2.5,-.5) {$U_1$};
\end{tikzpicture}$\qquad \qquad$ \begin{tikzpicture}[scale=.5]
\draw(0,0)--(2.5,2.5)--(5,0);
\draw(4,1)--(3,0);
\draw[fill](0,0) circle [radius=.1];
\draw[fill](3,0) circle [radius=.1];
\draw[fill](5,0) circle [radius=.1];
\draw[fill](2.5,2.5) circle [radius=.1];
\node[left] at (2.5,2.5) {$\rho'$};
\draw[fill](4,1) circle [radius=.1];
\node[right] at (4,1) {$b$};
\node[below] at (2.5,-.5) {$U_2$};
\end{tikzpicture}
\end{center}
\end{example}

We are interested in defining when we can add a path with highest node at $\rho$ to a path system in the larger cluster tree, $T$. This motivates the following definition of \emph{root-augmentability}.

\begin{definition}
A path system $\P$ is \textit{root-augmentable} if there exists a path $P'$ between the leaves of $T$ that has the root as its top-most node and is disjoint from all paths in $\P$. In other words, $[\P]$ has root-coordinate equal to 0, but setting it equal to 1 would still yield a valid top-vector.
\end{definition}

Let $T$ be a cluster tree and $\P$ be a path system $T$. Then $\P$ is root-augmentable if and only if this path system does not already have a path with highest node $\rho$ and the restriction of the path system to $T'$ has the following property.

\begin{definition}
A path system $\P$ in a bicluster tree $T'$ is \textit{root-leaf traversable} if there exists a path from the root to some leaf that does not include any internal node that is the top-most node of some path in $\P$. 
\end{definition}

Since $T'$ is a bicluster tree, in order for a path system in $T'$ to be root-leaf traversable, one must be able to add a path from $\rho'$ through the clusters of either $U_1$ or $U_2$ to a leaf. Note that a path system is root-leaf traversable if and only if removing all of the maximal nodes of paths in the path system leaves $\rho'$ in the same connected component as some leaf of $T'$. Therefore, root-leaf traversability is well-defined over classes of path systems with the same top-set. We often say that $[\P]$ is root-leaf traversable if $\P$ is root-leaf traversable.

 Note that we cannot use the same definition for root-augmentability and root-leaf traversability. Indeed, in a cluster tree, any path system all of whose paths do not contain the root must be root-leaf traversable. Root-augmentability should be thought of as the non-trivial notion of root-leaf traversability for cluster trees.

We can now define a special type of term order on the polynomial ring of the CFN-MC ideal of a cluster tree, and its analogue for that of a bicluster tree.

\begin{definition}
Let $S$ be a cluster tree with CFN-MC ideal $I_S \subset k[\ux]$. A term order $<$ on $k[\ux]$ is \emph{liftable} if
\begin{enumerate}
\item $I_S$ has a $<$-Gr\"obner basis consisting of degree 2 binomials, and
\item $<$ is a block order on $I_S$ with blocks
\[
\{x_{[\P]} \mid \P \text{ not root augmentable}\} > \{x_{[\P]} \mid \P \text{ root augmentable}\} 
\]
where the order induced on each block is graded. 
\end{enumerate}
\end{definition}

\begin{definition}
Let $S$ be a bicluster tree with CFN-MC ideal  $I_S \subset k[\ux]$. A term order $<$ on $k[\ux]$ is \emph{liftable} if
\begin{enumerate}
\item $I_S$ has a $<$-Gr\"obner basis consisting of degree 2 binomials, and
\item $<$ is a block order on $I_S$ with blocks
\[
\{x_{[\P]} \mid \P \text{ not root-leaf traversable}\} > \{x_{[\P]} \mid \P \text{ root-leaf traversable}\} 
\]
where the order induced on each block is graded. 
\end{enumerate}
\end{definition}

If $\G$ is the $<$-Gr\"obner basis for the CFN-MC ideal of a cluster or bicluster tree for a liftable turn order $<$, then $\G$ is \emph{liftable}.

\begin{definition}
Let $S$ be a cluster tree with CFN-MC ideal $I_S \subset k[\ux]$. Let $<$ be a term order on $k[\ux]$. Let $f = \prod_{i=1}^d x_{[\P_i]} - \prod_{i=1}^d x_{[\Q_i]}$ be a binomial in $I_S$ whose leading term is $\prod_{i=1}^d x_{[\P_i]}$. We say that $f$ satisfies the \emph{liftability property} with respect to $<$ if
\[
| \{ \P \mid \P \text{ not root-augmentable} \} | \geq | \{ \Q \mid \Q \text{ root augmentable} \} |.
\]
We define the liftability property when $S$ is a bicluster tree analogously with respect to root-leaf traversability. When the monomial order has been previously specified, we just say that the polynomial satisfies the liftability property.
\end{definition}

Note that a term order $<$ is liftable if and only if it induces a quadratic Gr\"obner basis all of whose elements satisfy the liftability property with respect to $<$.

Let $I_{U_1} \subset k[\ux]$, $I_{U_2} \subset k[\uy]$, $I_{T'} \subset k[\uz]$ and $I_T \subset k[\ur]$. Note that if $T$ is the smallest cluster tree with five leaves, then $U_1$ and $U_2$ are both trees with three leaves, so $I_{U_1}$ and $I_{U_2}$ are the zero ideal. Therefore, they vacuously have liftable Gr\"obner bases. By induction, let $\omega_1, \omega_2$ be weight vectors that induce liftable orders on $I_{U_1}$ and $I_{U_2}$, respectively. Let $\G_1$ be the liftable Gr\"obner basis for $I_{U_1}$ and $\G_2$ the liftable Gr\"obner basis for $I_{U_2}$.   Let $\a$ be the weight vector on $k[\uz]$ defined by $\a (z_{[\P]}) = 1$ if $\P$ is not root-leaf traversable and $\a(z_{\P}) = 0$ if $\P$ is root-leaf traversable. 

\begin{prop}\label{prop:Tprime}
Let $T'$ be a bicluster tree.
Let $U_1$ and $U_2$ be the unique cluster trees obtained from the left and right subtrees of $T'$
such that $I_{T'} = I_{U_1} \times_{\Acal} I_{U_2}$.
There exist a weight vector $\omega$ on $k[\uz]$ and $\epsilon, k > 0$ such that  $\xi_{U_1,U_2}^* (\omega_1, \omega_2) + \epsilon \omega + k \a$ induces a liftable order on $I_{T'}$
\end{prop}

\begin{proof}
By Lemma \ref{lem:pathsintfp}, we can write any path system in the bicluster tree $T'$ as $\P \vee \Q$ where $\P$ is a path system in $U_1$, $\Q$ is a path system in $U_2$, and the top-vectors of $\P$ and $\Q$ agree on the root of $T'$. Let $f$ be an element of the Gr\"obner basis for $I_{T'}$ described in Proposition \ref{prop:tfpgrobner}. If $f = z_{[\P_1 \vee \Q_1]} z_{[\P_2 \vee \Q_2]} - z_{[\P_1 \vee \Q_2]} z_{[\P_2 \vee \Q_1]} \in \Quad_i(T')$, then 
\begin{align*}
\xi_{U_1,U_2}^*(\omega_1, \omega_2)(z_{[\P_1 \vee \Q_1]} z_{[\P_2 \vee \Q_2]}) &= \omega_1 (x_{[\P_1]}) + \omega_2 (y_{[\Q_1]}) + \omega_1 (x_{[\P_2]}) + \omega_2 (y_{[\Q_2]})\\
&= \xi_{U_1,U_2}^*(\omega_1, \omega_2)(z_{[\P_1 \vee \Q_2]} z_{[\P_2 \vee \Q_1]}).
\end{align*}

We find $\epsilon > 0$ and weight vector $\omega$ to ``break ties" for leading monomials of each element of $\Quad_i(T')$. If $f$ is a $2 \times 2$ minor of $M_1(T')$, then every variable in $f$ is not root-leaf traversable. So the choice of leading monomial does not affect the liftability property. If $f$ is a $2 \times 2$ minor of $M_0(T')$, then there is only one case in which the number of root-leaf traversable variables in the two monomials of $f$ varies. Without loss of generality, let $\P_1, \Q_1$ be root augmentable and $\P_2, \Q_2$ not. Then $\P_2 \vee \Q_2$ is not root-leaf traversable, while $\P_1 \vee \Q_1, \P_1 \vee \Q_2$ and $\P_2 \vee \Q_1$ are. So, we must select an $\omega$ so that $z_{[\P_1 \vee \Q_1]} z_{[\P_2 \vee \Q_2]}$ is chosen as the leading monomial of $f$.

We define this weight vector $\omega$ by assigning its values on the entries of $M_0(T')$. Arrange path systems $\A_1, \dots, \A_r$ in $U_1$ so that if $\A_i$ is root augmentable and $\A_j$ is not, then $i<j$. Arrange path systems $\B_1, \dots, \B_s$ in $U_2$ so that if $\B_i$ is root augmentable and $\B_j$ is not, then $i<j$.

Define $\omega(z_{[\A_i \vee \B_j]}) = 2^{i+j}$ for all $i$ and $j$. Let $i_1 < j_1$ and $i_2< j_2$. Then 
\begin{align*}
\omega(z_{[\A_{i_1} \vee \B_{j_2}]} z_{[\A_{j_1} \vee \B_{i_2}] }) & = 2^{i_1 + j_2} + 2^{i_2 + j_1}\\
&\leq 2^{j_1 + j_2 - 1} + 2^{j_1 + j_2 -1} \\
&= 2(2^{j_1 + j_2 -1} )\\
&= 2^{j_1 + j_2} \\
&< 2^{i_1 + i_2} + 2^{j_1 + j_2} \\
& = \omega(z_{[\A_{i_1} \vee \B_{i_2}]} z_{[\A_{j_1} \vee \B_{j_2}]}) 
\end{align*}

So $\omega$ chooses the correct leading term of $f \in \Quad_1(T)$. We can allow $\omega$ to be any weight vector on the entries of $M_1(T')$ that chooses leading terms as in Proposition 2.6 of \cite{sullivant2007}. Pick $\epsilon$ to be small enough so that for all $g \in \Lift(\G_1) \cup \Lift(\G_2)$,
\[
LT_{\xi_{U_1,U_2}^*(\omega_1, \omega_2)}(g) = LT_{\xi_{U_1,U_2}^*(\omega_1, \omega_2) + \epsilon \omega}(g).
\]

Now we must add $k \a$ for some $k \geq 0$ to ensure that the correct leading term is chosen for each $f \in \Lift(\G_1) \cup \Lift(\G_2)$. Without loss of generality, let 

\[
f = z_{[\P_1 \vee \Rcal_1]} z_{[\P_2 \vee \Rcal_2]} - z_{[\Q_1 \vee \Rcal_1]}z_{[\Q_2 \vee \Rcal_2]} \in \Lift(\G_1).
\]

An analysis of all possible cases shows that the only instance in which the terms of $f$ have a varying number of root-leaf traversable indices but $\xi_{U_1,U_2}^*(\omega_1, \omega_2)$ may not select the correct leading term occurs when, without loss of generality, $\P_1, \Rcal_1$ and $\Q_2$ are not root augmentable and $\P_2, \Rcal_2$ and $\Q_1$ are root augmentable. In this case, $\P_1 \vee \Rcal_1$ is not root-leaf traversable and $\P_2 \vee \Rcal_2, \Q_1 \vee \Rcal_1$ and $\Q_2 \vee \Rcal_2$ are, but $x_{[\P_1]} x_{[\P_2]} - x_{[\Q_1]} x_{[\Q_2]}$ may not have $x_{[\P_1]} x_{[\P_2]}$ as its leading term. Suppose that under the weight vector $\xi_{U_1,U_2}^*(\omega_1, \omega_2)$, $z_{[\Q_1 \vee \Rcal_1]}z_{[\Q_2 \vee \Rcal_2]}$ is the leading term of $f$. Adding sufficiently many copies of $\a$ will change this, but we must show that adding $\a$ does not change the Gr\"obner basis $\G = \Lift(\G_1) \cup \Lift(\G_2) \cup \Quad(T')$.

Define the binomial $\overline{f} = z_{[\Q_1 \vee \Rcal_1]}z_{[\Q_2 \vee \Rcal_2]} -  z_{[\P_1 \vee \Rcal_2]} z_{[\P_2 \vee \Rcal_1]}$.  Note that $\bar{f} \in \Lift(\G_1)$ with $ z_{[\Q_1 \vee \Rcal_1]}z_{[\Q_2 \vee \Rcal_2]}$ as the leading term under $\xi_{U_1,U_2}^*(\omega_1, \omega_2) + \epsilon \omega$, and both terms of $\overline{f}$ have all root-leaf traversable indices. Since $f$ and $\overline{f}$ have the same leading term, $\G - \{f\}$ is still a Gr\"obner basis under $\xi_{U_1,U_2}^*(\omega_1, \omega_2) + \epsilon \omega$. Let $\G'$ be $\G$ with all such $f \in \Lift(\G_1) \cup \Lift(\G_2)$ that violate the liftability property removed. Then every binomial in $\G'$ satisfies the liftability property, and $\G'$ is still a Gr\"obner basis.

Let $g = m_1 - m_2 \in I_{T'}$ be a binomial. Then there exists a sequence $g_1, \dots, g_r \in \G'$ so that  $g$ reduces to 0 upon division by the elements of this sequence in order. Suppose that $m_1$ is the leading term of $g$ in the order induced by $\xi_{U_1,U_2}^*(\omega_1, \omega_2) + \epsilon \omega$, but $m_2$ is the leading term in the order induced by $\xi_{U_1,U_2}^*(\omega_1, \omega_2) + \epsilon \omega + \a$. Then $m_2$ has more variables whose indices are not root-leaf traversable than $m_1$. We claim that division by the same $g_1, \dots , g_r$, possibly in a different order, still reduces $g$ to 0. To divide $g$ by one of $g_1, \dots, g_r$, pick a $g_i$ whose leading term divides $m_2$. One must exist because all of the $g_i$ satisfy the liftability property, so it is impossible to divide $m_1$ by any $g_i$ and decrease the number of root-leaf traversable variables in it. So, we may choose an element of $\G'$ to proceed with the reduction of $g$, and $\G'$ is still a Gr\"obner basis for the weight order induced by $\xi_{U_1,U_2}^*(\omega_1, \omega_2) + \epsilon \omega + \a$.
\end{proof}

For any path system $\P$ in $T$, let $\P'$ be the path system in $T'$ obtained from $\P$ by deleting any path in $\P$ that contains the root $\rho$ of $T$. Define two maps $\psi, \psi': k[\ur] \rightarrow k[\uz]$ by
\[
\psi(r_{[\P]}) = z_{[\P']}
\]
and
\[
\psi'(r_{[\P]}) = \begin{cases}
z_{[\P']} & \text{ if $[\P]_\rho = 1$, and}\\
1 & \text{ if $[\P]_\rho = 0$.}
\end{cases}
\]

Let $<$ be a monomial order on $k[\uz]$ whose existence is established by Proposition \ref{prop:Tprime} and let $\G_<$ be the liftable Gr\"obner basis that it induces on $I_{T'}$. Then define a monomial order $\prec$ on $k[\ur]$ by $\ur^{\bb} \prec \ur^{\bc}$ if and only if
\begin{itemize}[label=\raisebox{0.25ex}{\tiny$\bullet$}]
\item $\psi(\ur^{\bb}) < \psi(\ur^{\bc})$, or
\item $\psi(\ur^{\bb}) = \psi(\ur^{\bc})$ and $\psi'(\ur^{\bb}) < \psi'(\ur^{\bc})$.
\end{itemize}

In  words, to determine which of two monomials is bigger, we delete the root and see which is bigger in the order on $T'$. If those are equal, then we only look at the indices with the root-coordinate equal to 1, and then delete the root from those and see which is bigger in the order on $T'$. 

Denote by $\Fcal$ the Gr\"obner basis for $I_{T'}$ induced by the term order $<$ on $k[\ux]$. Let $f = z_{[\P_1]} z_{[\P_2]} - z_{[\Q_1]} z_{[\Q_2]} \in \Fcal$. Define the set $\Root(f)$ to be the set of all possible binomials in $I_T$ that result from treating $\P_1, \P_2, \Q_1$ and $\Q_2$ as path systems in $T$, with or without an added path with top-most node $\rho$. Specifically,
\[
\text{Root}(f) = \{ r_{i_1 [ \P_1]} r_{i_2 [\P_2]} - r_{j_1 [\Q_1]} r_{j_2 [\Q_2]}  \}
\]
where $i_1, i_2, j_1, j_2 \in \{0,1\}$ are such that
\begin{itemize}[label=\raisebox{0.25ex}{\tiny$\bullet$}]
\item $i_1 + i_2 = j_1 + j_2$, and
\item $i_1 = 1$ only if $\P_1$ is root-leaf traversable, and similarly for $i_2, j_1, j_2$.
\end{itemize}

Denote by $\Root(T) = \bigcup_{f \in G_{<}} \Root(f)$. Define the set
\[
\text{Swap}(\rho) = \{ r_{1 [\P]} r_{0[\Q]} - r_{0[\P]}r_{1[\Q]} \}
\]
where $\P$ and $\Q$ range over all root-leaf traversable path systems in $T'$. 
Define the set $\G_{\prec} = \Root(T) \cup \Swap(\rho)$.
For the sake of brevity, we use an underline
to indicate the leading term of a polynomial.

\begin{prop}\label{prop:liftable}
The term order $\prec$ described above is liftable. In particular, $G_{\prec}$ is a Gr\"obner basis for $I_T$ with respect to $\prec$.
\end{prop}

\begin{proof}
First, we show that $G_{\prec}$ constitutes a Gr\"obner basis. Since, as explained in Section 3, $I_T$ is toric, it suffices to show that every binomial in $I_T$ can be reduced via the elements of $G_{\prec}$. Let $\underline{\prod_{i=1}^d r_{[\P_i]}} - \prod_{i=1}^d r_{[\Q_i]} \in I_T$. Then if we arrange the terms in each monomial as a table with the vector representing each $[\P_i]$ (resp. $[\Q_i]$) as a row, the column sums of each of these tables are equal. By the definition of $\prec$, we have
\[
\prod_{i=1}^d \psi(r_{[\P_i]}) \geq \prod_{i=1}^d \psi(r_{[\Q_i]}).
\]
For all $\P_i, \Q_i$, we have $\psi(r_{[\P_i]}) = z_{[\P'_i]}$ and $\psi(r_{[\Q_i]}) = z_{[\Q'_i]}$.

We can use the elements of $\Fcal$ to reduce $\prod_{i=1}^d z_{[\P'_i]}  - \prod_{i=1}^d z_{[\Q'_i]}$ in $I_{T'}$. The properties of the order on $I_{T'}$ induced by $<$ guarantee that (without loss of generality) if we divide by $z_{[\P'_1]}z_{[\P'_2]} - z_{[\Rcal'_1]}z_{[\Rcal'_2]}$ in this reduction, then the number of $\Rcal'_1$ and $\Rcal'_2$ that are root-leaf traversable is at least the number of $\P'_1$ and $\P'_2$ that are root-leaf traversable. Therefore, there is a corresponding element $\underline{r_{i_1 [\P'_1]}r_{i_2 [\P'_2]}} - r_{j_1 [\Rcal'_1]}r_{j_2 [\Rcal'_2]} \in \text{Root}(z_{[\P'_1]}z_{[\P'_2]} - z_{[\Rcal'_1]}z_{[\Rcal'_2]})$ with $r_{i_1 [\P'_1]}r_{i_2 [\P'_2]} = r_{[\P_1]}r_{[\P_2]}$. Note that by definition of $\prec$, $r_{i_1 [\P'_1]}r_{i_2 [\P'_2]}$ is indeed the leading term of this binomial.

This Gr\"obner basis reduction using elements $\text{Root}(T)$ ends in a binomial of the form
\[ \prod_{k=1}^d r_{i_k [\Rcal_k]} - \prod_{k=1}^d r_{j_k [\Rcal_k]}
\]
where $\sum_{k=1}^d i_k = \sum_{k=1}^d j_k$. At this point, we can use elements of $\text{Swap}(\rho)$ to match the columns of each monomial that correspond to the root. Note that it follows from the multiplicative property of monomial orders that we can always reduce the leading term this way by dividing by some element of $\text{Swap}(\rho)$. 

Now we can check that $\prec$ is a liftable term order on the elements of $G_{\prec}$. Any binomial in $\Swap(\rho)$ has one term that is root-augmentable and one that is not. So the choice of leading term of elements of $\Swap(\rho)$ does not affect the liftability property.

Let $f = r_{[\P_1]}r_{[\P_2]} - r_{[\Q_1]}r_{[\Q_2]} \in \Root(T)$. Then in particular, $\psi(r_{[\P_1]}r_{[\P_2]}) \neq \psi(r_{[\Q_1]}r_{[\Q_2]})$. There are several cases.

If $[\P_1]_{\rho} = [\P_2]_{\rho} = [\Q_1]_{\rho} = [\Q_2]_{\rho} = 1$, then neither monomial in $f$ has a root-augmentable term. So the leading term of $f$ does not affect the liftability property.

If $[\P_1]_{\rho} = [\Q_1]_{\rho}  = 1$ and $[\P_2]_{\rho} = [\Q_2]_{\rho}  = 0$, without loss of generality, then $[\P_1]$ and $[\Q_1]$ are both not root-augmentable, and $[\P_1']$ and $[\Q_1']$ both \textit{are} root-leaf traversable. If $[\P_2]$ and $[\Q_2]$ are both root-augmentable or are both not root-augmentable, then the choice of leading term of $f$ does not affect the liftability property. Suppose that $[\P_2]$ is not root-augmentable and $[\Q_2]$ is. Then under $\psi$, $z_{[\P_1']}z_{[\P_2']} > z_{[\Q_1']}z_{[\Q_2']}$ since $z_{[\P_1']}z_{[\P_2']} $ has one root-augmentable term and $z_{[\Q_1']}z_{[\Q_2']}$ has two root-augmentable terms.

If $[\P_1]_{\rho} = [\P_2]_{\rho} = [ \Q_1]_{\rho} = [\Q_2]_{\rho} = 0$, then the number of root-augmentable terms in either monomial in $f$ is the same as the number of root-leaf traversable terms in each under $\psi$. So, since $<$ is liftable, the monomial with the fewest root-augmentable terms is chosen as the leading term of $f$, as needed.
\end{proof}

\begin{proof}[Proof of Theorem \ref{thm:generatorsmain}]
If $T$ is the tree with three leaves, then $I_T = \langle 0 \rangle$ and the result holds vacuously. Let $T$ have $n > 3$ leaves. If $T$ is a cluster tree, then by induction on $n$, we may apply Proposition \ref{prop:liftable}, and $I_T$ has a liftable Gr\"obner basis. By definition of a liftable term order, this Gr\"obner basis consists of quadratic binomials. Otherwise, $I_T$ splits as a toric fiber product. So Proposition \ref{prop:tfpgrobner} and induction on $n$ imply that $I_T$ has a quadratic Gr\"obner basis with squarefree initial terms. \end{proof}

\begin{cor}\label{cor:triangulation}
The CFN-MC polytope has a regular unimodular triangulation and is normal.
\end{cor}

\begin{proof}
By Theorem \ref{thm:generatorsmain}, the CFN-MC ideal has a quadratic Gr\"obner basis. 
Elements of this Gr\"obner basis correspond to elements of the kernel of a  $0/1$ matrix. 
The only quadratic binomials that  could be generators of a toric ideal have the
form $a^2 - bc$ or $ab - cd$ for some indeterminates $a,b,c,d$ in the polynomial ring.
However, the type $a^2 - bc$ is not possible in a toric ideal whose associated matrix is
a $0/1$ matrix.
Since Theorem \ref{thm:generatorsmain} shows that $I_T$ has a quadratic
Gr\"obner basis, and  the leading term of each element of the Gr\"obner basis is squarefree, 
so the leading term ideal of $I_T$ with respect to the liftable term order 
$\prec$ is generated by squarefree monomials. Therefore, it is the 
Stanley-Reisner ideal of a regular unimodular triangulation of $R_T$ \cite[Theorem~8.3]{sturmfels1996}.
\end{proof}


\section{Ehrhart Function of the CFN-MC Polytope}\label{sec:ehrhart}
In this section, we show that for a rooted binary tree $T$, 
the Hilbert series of the CFN-MC ideal $I_T$ depends only on the 
number of leaves of $T$ and not on the topology of $T$. To accomplish this, 
we use Ehrhart theory of the polytopes $R_T$. Our approach is inspired 
by the work of Buczynska and Wisniewski, who proved a similar result for 
ideals arising from the CFN model without the molecular clock \cite{buczynska2007}, 
and of Kubjas, who gave a combinatorial proof of the same result  \cite{kubjas2013}. 
We provide a brief review of some definitions in Ehrhart theory below, 
and refer the reader to \cite{beck2007} for a more complete treatment of this material.

Let $P \subset \R^n$ be any $n$-dimensional polytope with integer vertices. If $P$ has volume $V$, then 
the \emph{normalized volume} of $P$ is $n! V$.
Recall that the \emph{Ehrhart function}, $i_P(m)$, counts the integer points in dilates of $P$; that is,
\[
i_P(m) = \# (\Z^n \cap mP),
\]
where $mP$ denotes the $m$th dilate of $P$. The Ehrhart function is, in fact, a polynomial in $m$. We further define the \emph{Ehrhart series} of $P$ to be the generating function
\[
\text{Ehr}_P(t) = \sum_{m \geq 0} i_P(m) t^m.
\] 
When $P$ is full-dimensional in the ambient space $\R^n$, the Ehrhart series is of the form
\[
\text{Ehr}_P(t) = \frac{h^*(t)}{(1-t)^{n+1}},
\]
where $h^*(t)$ is a polynomial in $t$ of degree at most $n$. Furthermore, recall that since $R_T$ has a regular unimodular triangulation, the Ehrhart series of $R_T$ is equal to the Hilbert series of $I_T$ \cite[Chapter~8]{sturmfels1996}. We now introduce \emph{alternating permuations} and the \emph{Euler zig-zag numbers}, which enumerate these combinatorial objects. We show that these numbers give the normalized volume of the CFN-MC polytopes.

\begin{definition}
A permutation on $n$ letters $a_1, \dots, a_n$ is \emph{alternating} if $a_1 < a_2 > a_3 < a_4 > \dots$. The \emph{Euler zig-zag number} $E_n$ is the number of alternating permutations on $n$ letters.
\end{definition}

In other words, a permutation is alternating if its descent set is exactly the set of even numbers less than $n$. For example, in one-line notation, the permutation $13254$ is alternating, while $13245$ is not since its fourth position is not a descent. The reader should note that some texts refer to these as reverse-alternating permutations and to permutations with descent set equal to the \emph{odd} numbers less than $n$ as alternating. However, the relevant results are not affected by which terminology we choose. 

Alternating permutations are fascinating combinatorial objects in their own right. For instance, the exponential generating function for the Euler zig-zag numbers satisfies
\[
\sum_{n \geq 0} E_n \frac{x^n}{n!} = \tan x + \sec x.
\]
Furthermore, the Euler zig-zag numbers satisfy the recurrence
\[
2 E_{n+1} = \sum_{k=0}^n \binom{n}{k} E_k E_{n-k}
\]
for $n \geq 1$ with initial values $E_0 = E_1 = 1$. The sequence of Euler zig-zag numbers beginning with $E_0$ begins $1, 1, 1, 2, 5, 16, 61, 272, ...$ and can be found in the Online Encyclopedia of Integer Sequences with identification number A000111 \cite{oeis}. For a thorough treatment of topics related to alternating permutations and the Euler zig-zag numbers, we refer the reader to Stanley's ``Survey of Alternating Permutations" \cite{stanley2010}. The goal of this section is to prove the following theorem.

\begin{thm}\label{thm:volumemain}
For any rooted binary tree $T$ with $n$ leaves, the normalized volume of $R_T$ is $E_{n-1}$, the $(n-1)$st Euler zig-zag number.
\end{thm}

The proof of  Theorem \ref{thm:volumemain} has two parts. 
First, we give a unimodular affine isomorphism between the CFN-MC polytope associated to the \emph{caterpillar tree} and the order polytope of the so-called ``zig-zag poset", which is known to have the desired normalized volume \cite{stanley2010}. The second, and more difficult, part is to 
show that the volume and Ehrhart polynomial of the CFN-MC polytope are the 
same for any $n$-leaf tree by giving a bijection between the lattice points 
in $(mR_T \cap \Z^{n-1})$ and $(mR_{T'} \cap \Z^{n-1})$ where $T$ and $T'$ 
are related by a single tree rotation.  Since any two binary trees on
$n$ leaves are connected by a sequences of rotations, this 
proves the theorem.

\subsection{Caterpillar Trees}
For a class of trees known as caterpillar trees, we can find a unimodular affine map between the CFN-MC polytope and the order polytope of a well-understood poset. 

\begin{definition}
A \textit{caterpillar tree} $C_n$ on $n$ leaves is the unique rooted tree topology with exactly one cherry. 
\end{definition}

\begin{definition}
The \textit{zig-zag poset} $P_n$ is the poset on underlying set $\{p_1, \dots, p_n\}$ with  the cover relations $p_i < p_{i+1}$ for $i$ odd and $p_{i} > p_{i+1}$ for $i$ even. Note that these are exactly the inequalities that appear in the definition of an alternating permutation. The \textit{order polytope} of the zig-zag poset $\mathcal{O}(P_n)$ is the set of all $\bv \in \R^n$ that satisfy $0 \leq v_i \leq 1$ for all $i$ and $v_i \leq v_j$ if $p_i < p_j$ in $P_n$.
\end{definition}

Order polytopes for arbitrary posets have been the object of considerable study, and are discussed in detail in \cite{stanley1986}. For instance, the order polytope of $P_n$ is also the convex hull of all $(v_1, \dots, v_n) \in \{0,1\}^n$ that correspond to labelings of $P_n$ that are weakly consistent with the partial order on $\{p_1, \dots, p_n\}$. 

In the case of $\mathcal{O}(P_n)$, the facet defining inequalities are those of the form
\begin{align} \begin{split}
-v_i  \leq 0 & \text{ for $i \leq n$ odd} \\
v_i  \leq 1 & \text{ for $i \leq n$ even}\\
v_i - v_{i+1}  \leq 0 & \text{ for $i \leq n-1$ odd, and} \\
-v_i + v_{i+1}  \leq 0 & \text{ for $ i  \leq  n-1$ even.} \end{split} \label{eqn:orderpolytopeineqs}
\end{align}
Note that the inequalities of the form $-v_i \leq 0$ for $i$ even and $v_i \leq 1$ for $i$ odd are redundant. 

Every order polytope has a unimodular triangulation whose simplices are in bijection with linear extensions of the underlying poset \cite{stanley1986}. In the case of the zig-zag poset $P_n$, the linear extensions of $P_n$ are in bijection with alternating permutations, since we can simply take each alternating permutation to be a labeling of the poset \cite{stanley2010}. These facts together imply the following proposition.

\begin{prop}
The order polytope $\mathcal{O}(P_n)$ has normalized volume $E_n$, the $n$th Euler zig-zag number.
\end{prop}

\begin{example}
Consider the zig-zag poset $P_4$ pictured in Figure \ref{fig:zigzag}. The matrix whose columns are the vertices of $\mathcal{O}(P_4)$ and the facet-defining hyperplanes of $\mathcal{O}(P_4)$ are given below. The volume of $\mathcal{O}(P_4)$ is $E_4 = 5.$

\begin{minipage}{.4\textwidth}
\[
\begin{bmatrix}
0 & 1 & 0 & 0 & 0 & 1 & 1 & 0 \\
1 & 1 & 0 & 1 & 1 & 1 & 1 & 0 \\
0 & 0 & 0 & 1 & 0 & 1 & 0 & 0\\
1 & 1 & 1 & 1 & 0 & 1 & 0 & 0
\end{bmatrix}
\]
\end{minipage}
\begin{minipage}{.3\textwidth}
\begin{align*}
-v_1 & \leq 0 \\
v_2 & \leq 1\\
-v_3 & \leq 0 \\
v_4 & \leq 1
\end{align*}
\end{minipage}\begin{minipage}{.3\textwidth}
\begin{align*}
v_1 - v_2 & \leq 0 \\
-v_2 + v_3 & \leq 0 \\
v_3 - v_4 & \leq 0
\end{align*}
\end{minipage}

\begin{figure}
\begin{center}
\begin{tikzpicture}
\node at (0,0) {};
\node at (0,2) {};
\draw(1,0.5)--(2,1.5)--(3,0.5)--(4,1.5);
\draw[fill] (1,.5) circle [radius=.05];
\node[left] at (1,.5) {$p_1$};
\draw[fill] (2,1.5) circle [radius=.05];
\node[left] at (2,1.5) {$p_2$};
\draw[fill] (3,.5) circle [radius=.05];
\node[left] at (3,.5) {$p_3$};
\draw[fill] (4,1.5) circle [radius=.05];
\node[left] at (4,1.5) {$p_4$};
\end{tikzpicture}
\end{center}
\caption{The zig-zag poset $P_4$}
\label{fig:zigzag}
\end{figure}
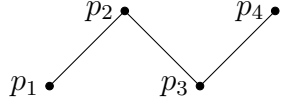

\end{example}

\begin{prop}\label{prop:cattoorderpolytope}
Let $D$ be the $n \times n$ diagonal matrix with $D_{ii} = 1$ if $i$ is odd and $D_{ii} = -1$ if $i$ is even. Let $\a$ be the vector in $\R^n$ with $a_i = 0$ if $i$ is odd and $a_i = 1$ if $i$ is even. The rigid motion of $\R^n$ defined by
\[
\phi(\bx) = D \bx + \a
\]
is a unimodular affine isomorphism from $R_{C_{n+1}}$ to $\mathcal{O}(P_n)$.
\end{prop}

\begin{proof}
First note that $\det D = \pm 1$, so $\phi$ is a unimodular affine isomorphism. The image of $\bx$ under $\phi$ is coordinate-wise by
\[
\phi(\bx)_i = \begin{cases}
x_i & \text{ if $i$ is odd, and} \\
1- x_i & \text{if $i$ is even.}
\end{cases}
\]
By Corollary, \ref{cor:finalfacets}, the facet-defining inequalities of $R_{C_{n+1}}$ are of the form $-x_i \leq 0$ for $i \leq n$ and $x_i + x_{i+1} \leq 1$ for $i \leq n-1$. Substitution $\phi(\bx)$ into each of these equations yields exactly the inequalities in Equation (\ref{eqn:orderpolytopeineqs}), as needed.
\end{proof}

\begin{cor}\label{cor:cattreevolume}
The Ehrhart functions of the CFN-MC polytope $R_{C_{n+1}}$ and the order polytope $\mathcal{O}(P_n)$ are equal for all $n$. This further implies that the normalized volume of $R_{C_{n+1}}$ is the $n$th Euler zig-zag number, $E_n$.
\end{cor}

\begin{proof}
The Ehrhart functions $i_{R_{C_{n+1}}}(m)$ and $i_{\mathcal{O}(P)}(m)$ are equal because $\phi$ is a lattice-point preserving transformation from $mR_{C_{n+1}}$ to $m\mathcal{O}(P_n)$. The leading coefficient of the Ehrhart polynomial of a polytope is the volume of that polytope. This volume is $\frac{E_n}{n!}$ for $\mathcal{O}(P_n)$ and so, for $R_{C_{n+1}}$ as well \cite{stanley2010}. So the normalized volume of $R_{C_{n+1}}$ is $E_n$.
\end{proof}

\subsection{The Ehrhart Function and Rotations}

We give an explicit bijection between the lattice points in the $m$-th 
dilate of $R_T$ and of $R_{T'}$ where $m \in \Z_+$ and $T$ and $T'$ differ by 
one rotation. 
This shows that the Ehrhart polynomials of $R_T$ and $R_{T'}$ are the same. 
Since any tree can be obtained from any other tree by a finite sequence of 
rotations, this along with Corollary \ref{cor:cattreevolume} proves Theorem \ref{thm:volumemain}.

Let $b$, $c$, $e$ be three consecutive nodes of a tree $T$, where $c$ is a descendant of $b$, and $e$ is a descendant of $c$. Note that node $e$ need not be an internal node of $T$. There is a unique 
\emph{rotation}
 associated to the triple $(b,c,e)$; namely, this move prunes $c$, the edge $ce$ and the 
 $e$-subtree from below $b$, and reattaches these on the other edge 
immediately below $b$ to yield a new tree, $T'$. 
This rotation, and its effect on the internal structure of $T$ is 
depicted in Figure \ref{fig:NNI}. Note that it is also possible for $b$ to be the root, in which case node $a$ in this figure
does not exist. A rotation splits the vertices of 
$R_T$ into two natural categories: the ones that are also vertices of $R_{T'}$ and the ones that are not.

\begin{figure}
\begin{center}
\begin{tikzpicture}
\draw (2,2)--(5,5);
\draw (2,2)--(1.5,1);
\draw (2,2)--(2.5,1);
\draw (3,3)--(4,2);
\draw (4,2)--(4.5,1);
\draw (4,2)--(3.5,1);
\draw (4,4)--(6,2);
\draw (6,2)--(5.5,1);
\draw (6,2)--(6.5,1);
\draw[dotted] (1.5,1)--(1.25,.5);
\draw[dotted] (2.5,1)--(2.75,.5);
\draw[dotted] (3.5,1)--(3.25,.5);
\draw[dotted] (4.5,1)--(4.75,.5);
\draw[dotted] (5.5,1)--(5.25,.5);
\draw[dotted] (6.5,1)--(6.75,.5);
\draw[dotted] (5,5)--(5.5,5.5);
\draw[fill] (5,5) circle  [radius =.05];
\node[left] at (5,5) {$a$};
\draw[fill] (4,4) circle  [radius =.05];
\node[left] at (4,4) {$b$};
\draw[fill] (3,3) circle  [radius =.05];
\node[left] at (3,3) {$c$};
\draw[fill] (2,2) circle  [radius =.05];
\node[left] at (2,2) {$d$};
\draw[fill] (4,2) circle  [radius =.05];
\node[left] at (4,2) {$e$};
\draw[fill] (6,2) circle  [radius =.05];
\node[right] at (6,2) {$f$};
\node at (4,0) {$T$};
\end{tikzpicture} $\qquad$ \begin{tikzpicture}
\draw (2,2)--(5,5);
\draw (2,2)--(1.5,1);
\draw (2,2)--(2.5,1);
\draw (5,3)--(4,2);
\draw (4,2)--(4.5,1);
\draw (4,2)--(3.5,1);
\draw (4,4)--(6,2);
\draw (6,2)--(5.5,1);
\draw (6,2)--(6.5,1);
\draw[dotted] (1.5,1)--(1.25,.5);
\draw[dotted] (2.5,1)--(2.75,.5);
\draw[dotted] (3.5,1)--(3.25,.5);
\draw[dotted] (4.5,1)--(4.75,.5);
\draw[dotted] (5.5,1)--(5.25,.5);
\draw[dotted] (6.5,1)--(6.75,.5);
\draw[dotted] (5,5)--(5.5,5.5);
\draw[fill] (5,5) circle  [radius =.05];
\node[left] at (5,5) {$a$};
\draw[fill] (4,4) circle  [radius =.05];
\node[left] at (4,4) {$b$};
\draw[fill] (5,3) circle  [radius =.05];
\node[left] at (5,3) {$c$};
\draw[fill] (2,2) circle  [radius =.05];
\node[left] at (2,2) {$d$};
\draw[fill] (4,2) circle  [radius =.05];
\node[left] at (4,2) {$e$};
\draw[fill] (6,2) circle  [radius =.05];
\node[right] at (6,2) {$f$};
\node at (4,0) {$T'$};
\end{tikzpicture} 
\end{center}
\caption{A rotation performed by pruning $c$ and its right subtree and reattaching in the right subtree of $b$.}
\label{fig:NNI}
\end{figure}
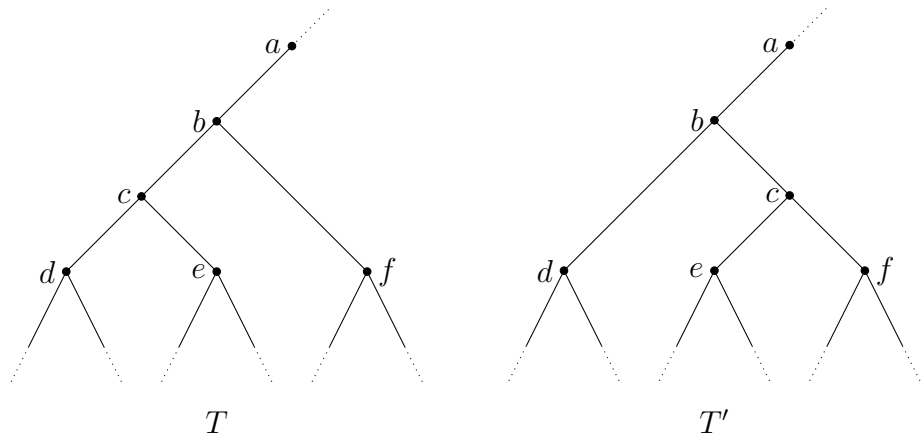

\begin{definition}
Let the tree $T'$ be obtained from $T$ by the rotation associated to $(b,c,e)$ as pictured in Figure \ref{fig:NNI}. A vertex of $R_T$ is \textit{maintaining} if it is also a vertex of $R_{T'}$. A vertex of $R_T$ is \textit{nonmaintaining} if it is not a vertex of $R_{T'}$.
\end{definition}

We use the following definition to give a characterization of the maintaining and nonmaintaining vertices of $R_T$.

\begin{definition}
Let $S$ be the top-set of a path system in $T$. Let $x$ be an internal node of $T$. Then $x$ is \textit{blocked} in $S$ if for every path from $x$ to a leaf descended from $x$, there exists a $y \in S$ that lies on this path.
\end{definition}

Note that if $x$ is blocked in $S$, then for every path system $\P$ that realizes $S$, we cannot add another path to $\P$ with top-most node above $x$ that passes through $x$.

\begin{example} Let $T$ be the tree pictured in Figure \ref{fig:blocked} with a path system $\P$ drawn in bold. Note that up to a swap of the leaves below node $h$, $\P$ is the only path system in $T$ that realizes top-set $\{a,d,e,g\}$.

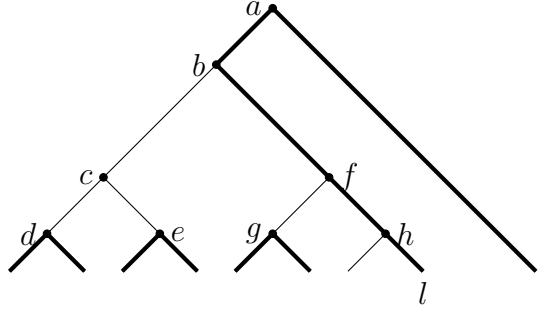
\begin{figure}
\begin{center}
\begin{tikzpicture}[scale=.5]
\draw (0,0)--(7,7)--(14,0);
\draw (1,1)--(2,0);
\draw (2.5,2.5)--(5,0);
\draw (4,1)--(3,0);
\draw (5.5,5.5)--(11,0);
\draw (10,1)--(9,0);
\draw (8.5,2.5)--(6,0);
\draw (7,1)--(8,0);
\draw[fill] (7,7) circle [radius = .1];
\node[left] at (7,7) {$a$};
\draw[fill] (5.5,5.5) circle [radius = .1];
\node[left] at (5.5,5.5) {$b$};
\draw[fill] (2.5,2.5) circle [radius = .1];
\node[left] at (2.5,2.5) {$c$};
\draw[fill] (1,1) circle [radius = .1];
\node[left] at (1,1) {$d$};
\draw[fill] (4,1) circle [radius = .1];
\node[right] at (4,1) {$e$};
\draw[fill] (8.5,2.5) circle [radius = .1];
\node[right] at (8.5,2.5) {$f$};
\draw[fill] (7,1) circle [radius = .1];
\node[left] at (7,1) {$g$};
\draw[fill] (10,1) circle [radius = .1];
\node[right] at (10,1) {$h$};
\node[below] at (11,0) {$l$};
\draw[ultra thick] (0,0)--(1,1)--(2,0);
\draw[ultra thick] (3,0)--(4,1)--(5,0);
\draw[ultra thick] (6,0)--(7,1)--(8,0);
\draw[ultra thick] (11,0)--(5.5,5.5)--(7,7)--(14,0);
\end{tikzpicture}
\end{center}
\caption{Nodes $a,c,d,e$ and $g$ are the nodes that are blocked in top-set $\{a,d,e,g\}$.}
\label{fig:blocked}
\end{figure}

By definition, $a,d,e$ and $g$ are all blocked in $\{a,d,e,g\}$. Furthermore, node $c$ is blocked in $\{a,d,e,g\}$ since any path from $c$ to a leaf descended from $c$ passes through either $d$ or $e$, but $d,e \in \{a,d,e,g\}$. On the other hand, $f$ is not blocked in $\{a,d,e,g\}$, since there is a path from $f$ to leaf $l$ that only passes through $h$, and $h \not\in \{a,d,e,g\}$. Similarly, $b$ is not blocked in $\{a,d,e,g\}$.
\end{example}

\begin{prop} \label{prop:blocked} Let $[\P]$ be a vertex of $R_T$ with associated top-set $V$. Let $a,b,c,d,e$ and $f$ be as in tree $T$ in Figure \ref{fig:NNI}.
\begin{enumerate}[label=(\roman*)]
\item If $b,c \not\in V$, then $[\P]$ is maintaining.
\item If $b \in V$, then $[\P]$ is maintaining if and only if $d$ is not blocked in $V$.
\item If $c \in V$, then $[\P]$ is maintaining if and only if $f$ is not blocked in $V$.
\end{enumerate}
\end{prop}

\begin{proof}
To prove (i), let $b,c \not\in V$. Then since all paths in $\P$ are disjoint, there can be at most one $P \in \P$ that is not contained entirely in the $d$-, $e$- or $f$-subtrees, or in $T$ without the $b$-subtree. If no such $P$ exists, then $\P$ is still a path system that realizes $V$ in $T'$, as needed. In particular, if $b$ is the root of $T$, then no such path can exist since in this case, every path is contained in the $b$-subtree. Now suppose that such a $P \in \P$ does exist. Then it must be the case that $b$ is the direct descendent of some node $a$. We can modify P to be a path $P'$ in $T'$ as follows.

If $ab, bc, cd \in P$, then let $P'$ be the path in $T'$ obtained from $P$ by replacing edges $bc$ and $cd$ with edge $bd$, and leaving all others the same. 
 If $ab, bc, ce \in P$, then $P' = P$ is also a path in $T'$, and we do not need to modify it. 
  If $ab, bf \in P$, then let $P'$ be the path in $T'$ obtained by replacing edge $bf$ in $P$ with edges $bc$ and $cf$, and leaving all others the same. 
Since $b,c \not\in V$, these are the only cases. Then $(\P - \{P\})\cup\{P'\}$ is a path system in $T'$ and $[(\P - \{P\})\cup\{P'\}] = [\P]$. So $[\P]$ is maintaining.

To prove (ii), let $b \in V$. Suppose that $[\P]$ is maintaining. Then there exists a path system $\P'$ in $T'$ such that $[\P'] = [\P]$. Let $P \in \P'$ have top-most node $b$. Then $bd \in P$ and $P$ includes a path $\bar{P}$ from $d$ to a leaf descended from $d$. Since all paths in $\P'$ are disjoint, no path in $\P'$ has its top-most node along $\bar{P}$. So $d$ is not blocked in $V$.

Suppose that $d$ is not blocked in $V$. Then there exists a path $\hat{P}$ from $d$ to a leaf descended from $d$ with none of its nodes in $V$. Let $\P$ be the path system in $T$ that realizes $V$, and let $P \in \P$ have top-most node $b$. We may assume that $\hat{P}$ is contained in $P$. $P$ also includes a path $\bar{P}$ from $f$ to a leaf descended from $f$. Let $P' = \hat{P} \cup \bar{P} \cup \{bd, bc, cf\}$. Then $\P' = (\P - \{P\})\cup\{P'\}$ is a path system in $T'$ with $[\P'] = [\P]$. So $[\P]$ is maintaining. 

To prove (iii), let $c \in V$. Suppose that $[\P]$ is maintaining. Then there exists a path system $\P'$ in $T'$ such that $[\P'] = [\P]$. Let $P \in \P'$ have top-most node $c$. Then $cf \in P$ and $P$ includes a path $\bar{P}$ from $f$ to a leaf descended from $f$. Since all paths in $\P'$ are disjoint, no path in $\P'$ has top-most node along $\bar{P}$. So $f$ is not blocked in $V$.

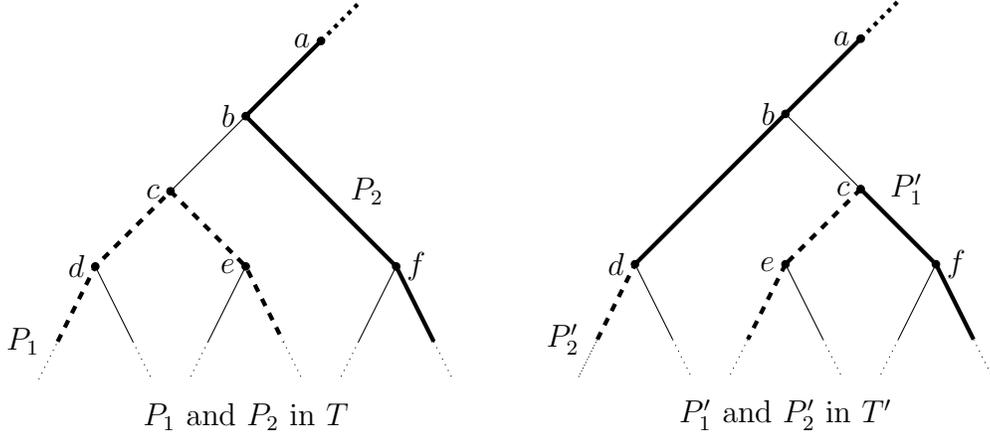
\begin{figure}
\begin{center}
\begin{tikzpicture}
\draw[dotted](1.25,.5)--(1.5,1);
\draw[dotted](4.5,1)--(4.75,.5);
\draw[ultra thick, dashed] (1.5,1)--(2,2)--(3,3)--(4,2)--(4.5,1);
\draw[ultra thick] (5,5)--(4,4)--(6,2)--(6.5,1);
\node[left] at (1.4,1) {$P_1$};
\node[right] at (5.25,3) {$P_2$};
\draw (3,3)--(5,5);
\draw (2,2)--(2.5,1);
\draw (4,2)--(3.5,1);
\draw (4,4)--(6,2);
\draw (6,2)--(5.5,1);
\draw (6,2)--(6.5,1);
\draw[dotted] (2.5,1)--(2.75,.5);
\draw[dotted] (3.5,1)--(3.25,.5);
\draw[dotted] (5.5,1)--(5.25,.5);
\draw[ultra thick, dotted](5,5)--(5.5,5.5);
\draw[dotted] (6.5,1)--(6.75,.5);
\draw[fill] (5,5) circle  [radius =.05];
\node[left] at (5,5) {$a$};
\draw[fill] (4,4) circle  [radius =.05];
\node[left] at (4,4) {$b$};
\draw[fill] (3,3) circle  [radius =.05];
\node[left] at (3,3) {$c$};
\draw[fill] (2,2) circle  [radius =.05];
\node[left] at (2,2) {$d$};
\draw[fill] (4,2) circle  [radius =.05];
\node[left] at (4,2) {$e$};
\draw[fill] (6,2) circle  [radius =.05];
\node[right] at (6,2) {$f$};
\node at (4,0) {$P_1$ and $P_2$ in $T$};
\end{tikzpicture} $\qquad$ \begin{tikzpicture}
\draw[dotted](1.25,.5)--(1.5,1);
\draw[ultra thick, dashed] (1.5,1)--(1.5,1)--(2,2);
\draw[ultra thick] (2,2)--(5,5);
\draw (3,3)--(5,5);
\draw (2,2)--(2.5,1);
\node[left] at (1.4,1) {$P_2'$};
\node[right] at (5.25,3) {$P_1'$};
\draw[ultra thick, dashed] (5,3)--(4,2);
\draw (4,2)--(4.5,1);
\draw [ultra thick, dashed] (4,2)--(3.5,1);
\draw (4,4)--(5,3);
\draw [ultra thick] (5,3)--(6,2);
\draw (6,2)--(5.5,1);
\draw [ultra thick] (6,2)--(6.5,1);
\draw[dotted] (1.5,1)--(1.25,.5);
\draw[dotted] (2.5,1)--(2.75,.5);
\draw[dotted] (3.5,1)--(3.25,.5);
\draw[dotted] (4.5,1)--(4.75,.5);
\draw[dotted] (5.5,1)--(5.25,.5);
\draw[dotted] (6.5,1)--(6.75,.5);
\draw[ultra thick, dotted] (5,5)--(5.5,5.5);
\draw[fill] (5,5) circle  [radius =.05];
\node[left] at (5,5) {$a$};
\draw[fill] (4,4) circle  [radius =.05];
\node[left] at (4,4) {$b$};
\draw[fill] (5,3) circle  [radius =.05];
\node[left] at (5,3) {$c$};
\draw[fill] (2,2) circle  [radius =.05];
\node[left] at (2,2) {$d$};
\draw[fill] (4,2) circle  [radius =.05];
\node[left] at (4,2) {$e$};
\draw[fill] (6,2) circle  [radius =.05];
\node[right] at (6,2) {$f$};
\node at (4,0) {$P_1'$ and $P_2'$ in $T'$};
\end{tikzpicture} 
\end{center}
\caption{Proposition \ref{prop:blocked} (iii): In this case, $\bar{P_1}$ contains all dashed edges descended from $d$. $\hat{P_1}$ contains all dotted edges descended from $e$. $\bar{P}$ contains all thick solid edges descended from $f$. $\hat{P_2}$ contains all thick solid edges above $b$.}
\label{fig:blockediii}
\end{figure}

Suppose that $f$ is not blocked in $V$. Let $P_1 \in \P$ have top-most node $c$. 
Since $f$ is not blocked in $V$, there exists a path $\bar{P}$ from $f$ to a leaf
descended from $f$ such that no node on $\bar{P}$ is in $V$.
Note that this path may be contained in some path $P_2 \in \P$. 
If such $P_2$ exists, it must have top-most node above $b$. So $bf \in P_2$ in this case. Let $\hat{P_2} = P_2 - (\bar{P} \cup \{bf\})$. These paths are illustrated in Figure \ref{fig:blockediii}.

Furthermore, $P_1$ contains a path $\bar{P_1}$ from $d$ to a leaf descended from $d$, and $\hat{P_1}$ from $e$ to a leaf descended from $e$.

Let $P_1'$ be the path in $T'$ with top-most node $c$,
\[
P_1' = \hat{P_1} \cup \bar{P} \cup \{ce, cf\}.
\]

If there exists a $P_2 \in \P$ that contains $f$, let $P_2'$ be the path in $T'$,
\[
P_2' = \bar{P_1} \cup \hat{P_2} \cup \{bd\}.
\]

Then $\P' = (\P - \{P_1,P_2\}) \cup \{P_1',P_2'\}$, or $(\P - \{P_1\}) \cup \{ P_1' \}$ if no such $P_2$ exists, is a path system in $T'$ $[\P'] = [\P]$.  So $\bv$ is maintaining.
\end{proof}

For simplicity, if the $b$-coordinate of $[\P]$ is equal to 1 (ie. $[\P]_b = 1$) and $[\P]$ is nonmaintaining, we say that $[\P]$ is $b$-nonmaintaining, and similarly for node $c$. 

\begin{prop} \label{prop:nonmaintainingbij}
The $b$-nonmaintaining vertices of $R_T$ are in bijection with the $c$-nonmaintaining vertices of $R_{T'}$. Similarly, the $c$-nonmaintaining vertices of $R_T$ are in bijection with the $b$-nonmaintaining vertices of $R_{T'}$
\end{prop}

\begin{proof}
Let $[\P]$ be a $b$-nonmaintaining vertex of $R_T$. Then by Proposition \ref{prop:blocked}, $d$ is blocked in the top-set $V$ of $[\P]$. So the path $P \in \P$ with top-most node $b$ passes through the $e$-subtree of $T$. Let $P'$ be the path in $T'$ given by
\[
P' = (P - \{bc,bf\}) \cup \{ cf \}.
\]
Then $\P'=(\P - \{P\}) \cup P')$ is a path system in $T'$ that matches $[\P]$ on all coordinates other than the $b$- and $c$-coordinates, and that has $b$-coordinate equal to 0 and $c$-coordinate equal to 1. Since $d$ is blocked in $\P'$, $[\P']$ is $c$-nonmaintaining in $T'$. (Note that the node labels do not match those of Proposition \ref{prop:blocked} since we are applying the result to the tree obtained after the rotation has been performed.) Performing the reverse operation on a $c$-nonmaintaining vertex $[\P']$ of $T'$ shows that this is a bijection.
\end{proof}

\begin{definition} If $[\P]$ is nonmaintaining, let $\P'$ be the path system described in the proof of Proposition \ref{prop:nonmaintainingbij}, so that $[\P']_b = [\P]_c$, $[\P']_c = [\P]_b$, and $[\P']$ matches $[\P]$ for all other nodes of $T$. Proposition \ref{prop:nonmaintainingbij} allows us to define the following involution between the vertices of $R_T$ and $R_{T'}$:
\begin{align*}
\phi^{T,T'}: \ &\text{vert}(R_T) \rightarrow  \text{vert}(R_{T'}) \\
& \quad [\P] \quad \mapsto \quad \begin{cases}
[\P] , & \text{ if $[\P]$ is maintaining}\\
[\P'], & \text{ if $[\P]$ is nonmaintaining}.
\end{cases}
\end{align*}
\end{definition}

We now turn our attention to the integer lattice points in the $m$\textsuperscript{th} dilates of $R_T$ and $R_{T'}$ for $m \in \Z_{+}$. Let $\bv \in \Z^{n-1} \cap mR_T$. Recall that by Corollary \ref{cor:triangulation}, $R_T$ is normal. So, we may write $\bv = [\P_1] + \dots + [\P_m]$ for some $[\P_1], \dots, [\P_m] \in \text{vert}(R_T)$. We call $[\P_1] + \dots + [\P_m]$ a \textit{representation} of $\bv$. Such a representation is \textit{minimal} if it uses the smallest number of nonmaintaining vertices over all representations of $\bv$.

For each vertex $[\P_i]$ of $R_T$, let $V_i$ denote the top-set associated to $[\P_i]$.

\begin{definition}
A representation $[\P_1] + \dots + [\P_m] = \bv \in mR_T \cap \Z^{n-1}$ is \emph{$d$-compressed} if
\begin{itemize}[label=\raisebox{0.25ex}{\tiny$\bullet$}]
\item all of $[\P_1], \dots, [\P_m]$ with $b$-coordinate equal to 1 are maintaining, or
\item for all $[\P_i]$ with $b,c \notin V_i$, $d$ is blocked in $V_i$.
\end{itemize}
Similarly, this representation is \emph{$f$-compressed} if 
\begin{itemize}[label=\raisebox{0.25ex}{\tiny$\bullet$}]
\item all of $[\P_1] , \dots ,[\P_m]$ with $c$-coordinate equal to 1 are maintaining, or
\item for all $[\P_i]$ with $b,c \notin V_i$, $f$ is blocked in $V_i$.
\end{itemize}
If $[\P_1] + \dots + [\P_m]$ is both $d$-compressed and $f$-compressed, then we say that the representation is \emph{$df$-compressed}.
\end{definition}

Consider the map $\phi^{T,T'}_m: (mR_T \cap \Z^{n-1}) \rightarrow (mR_{T'} \cap \Z^{n-1})$ defined by
\[
\phi^{T,T'}_m (\bv) = \sum_{i=1}^m \phi^{T,T'}([\P_i])
\]
where $\sum_{i=1}^m [\P_i]$ is a minimal representation of $\bv$. Both the well-definedness of this map, as well as the fact that it is a bijection follow Lemmas \ref{lem:compressed} and \ref{lem:minreps} below.

\begin{lemma}\label{lem:compressed}
If $[\P_1] + \dots + [\P_m]$ is a minimal representation of $\bv \in mR_T \cap \Z^{n-1}$, then $[\P_1] + \dots + [\P_m]$ is $df$-compressed.
\end{lemma}

In the proofs of Lemmas \ref{lem:compressed} and \ref{lem:minreps}, we use the following notation. For all top-sets $S$ and all nodes $x$ of $T$, denote by $S^x$ the intersection of $S$ with the set of internal nodes of the $x$-subtree. For all $\bw \in \R^{n-1}$, denote by $\bw^x$ the restriction of $\bw$ to the coordinates corresponding to nodes in the $x$-subtree. For all path systems $\P$ in $T$, denote by $\P^x$ the set of all paths in $\P$ that are contained in the $x$-subtree. 

\begin{proof}[Proof of Lemma \ref{lem:compressed}]
 We prove the contrapositive. Suppose that $[\P_1] + \dots + [\P_m]$ is not $df$-compressed. Then this representation is either not $d$-compressed or not $f$-compressed. We show that $[\P_1] + \dots + [\P_m]$ is not minimal in both cases.

If $[\P_1] + \dots + [\P_m]$ is not $d$-compressed, then it must be the case that both conditions in the definition of a $d$-compressed representation fail.  Since the first condition fails, there exists a $[\P_i]$ with $b$-coordinate equal to 1 which is nonmaintaining. Without loss of generality, we may assume that $[\P_1]$ is this $b$-nonmaintaining vertex. Since the second condition fails, there exists a $[\P_i]$ with $b,c \not\in V_i$ and where $d$ is not blocked in $V_i$. Without loss of generality, we may also assume that $[\P_2]$ is this vertex with $b, c \not\in V_2$ but where $d$ is not blocked in $[\P_2]$. We claim that $\bar{V_1} = (V_1 - (V_1^d \cup V_1^e)) \cup V_2^d \cup V_2^e$ and $\bar{V_2} = (V_2 - (V_2^d \cup V_2^e)) \cup V_1^d \cup V_1^e$ are valid top-sets in $T$ with associated path systems $\bar{\P_1}$ and $\bar{\P_2}$ respectively. We further claim that $[\bar{\P_1}]$ and $[\bar{\P_2}]$ are maintaining. The operation described in this proof is illustrated for an example tree $T$ and top-sets $V_1$ and $V_2$ in Figure \ref{fig:dcompression}.

\begin{figure}
\begin{tikzpicture}[scale=.45]
\draw(0,0)--(8,8)--(16,0);
\draw(1,1)--(2,0);
\draw (2.5,2.5)--(5,0);
\draw (4,1)--(3,0);
\draw (4,4)--(8,0);
\draw (7,1)--(6,0);
\draw (7,7)--(14,0);
\draw (11.5,2.5) -- (9,0);
\draw (10,1)--(11,0);
\draw (13,1)--(12,0);
\draw[fill] (8,8) circle [radius = .1];
\node[left] at (8,8) {$a$};
\draw[fill] (7,7) circle [radius = .1];
\node[left] at (7,7) {$b$};
\draw[fill] (4,4) circle [radius = .1];
\node[left] at (4,4) {$c$};
\draw[fill] (2.5,2.5) circle [radius = .1];
\node[left] at (2.5,2.5) {$d$};
\draw[fill] (7,1) circle [radius = .1];
\node[left] at (7,1) {$e$};
\draw[fill] (11.5,2.5) circle [radius = .1];
\node[right] at (11.7,2.5) {$f$};
\draw[fill] (1,1) circle [radius = .1];
\node[left] at (1,1) {$g$};
\draw[fill] (4,1) circle [radius = .1];
\node[right] at (4,1) {$h$};
\draw[fill] (10,1) circle [radius = .1];
\node[left] at (10,1) {$i$};
\draw[fill] (13,1) circle [radius = .1];
\node[right] at (13.2,1) {$j$};
\draw[ultra thick] (0,0)--(1,1)--(2,0);
\draw[ultra thick] (3,0)--(4,1)--(5,0);
\draw[ultra thick, dashed] (8,0)--(4,4)--(7,7)--(14,0);
\draw[ultra thick] (9,0)--(10,1)--(11,0);
\node[below, align=center] at (8,0) {A $b$-nonmaintaining path system $\P_1$  \\ realizing $V_1 = \{b,g,h,i\}$ };
\end{tikzpicture}$\qquad$ \begin{tikzpicture}[scale=.45]
\draw(0,0)--(8,8)--(16,0);
\draw(1,1)--(2,0);
\draw (2.5,2.5)--(5,0);
\draw (4,1)--(3,0);
\draw (4,4)--(8,0);
\draw (7,1)--(6,0);
\draw (7,7)--(14,0);
\draw (11.5,2.5) -- (9,0);
\draw (10,1)--(11,0);
\draw (13,1)--(12,0);
\draw[fill] (8,8) circle [radius = .1];
\node[left] at (8,8) {$a$};
\draw[fill] (7,7) circle [radius = .1];
\node[left] at (7,7) {$b$};
\draw[fill] (4,4) circle [radius = .1];
\node[left] at (4,4) {$c$};
\draw[fill] (2.5,2.5) circle [radius = .1];
\node[left] at (2.5,2.5) {$d$};
\draw[fill] (7,1) circle [radius = .1];
\node[left] at (7,1) {$e$};
\draw[fill] (11.5,2.5) circle [radius = .1];
\node[right] at (11.7,2.5) {$f$};
\draw[fill] (1,1) circle [radius = .1];
\node[left] at (1,1) {$g$};
\draw[fill] (4,1) circle [radius = .1];
\node[right] at (4,1) {$h$};
\draw[fill] (10,1) circle [radius = .1];
\node[left] at (10,1) {$i$};
\draw[fill] (13,1) circle [radius = .1];
\node[right] at (13.2,1) {$j$};
\draw[ultra thick] (0,0)--(2.5,2.5);
\draw[ultra thick] (2.5,2.5) -- (8,8)--(16,0);
\draw [ultra thick] (6,0)--(7,1)--(8,0);
\draw[ultra thick] (12,0)--(13,1)--(14,0);
\node[below, align=center] at (8,0) {A path system $\P_2$ realizing \\ $V_2 = \{a,e,j\}$  with $d$ not blocked in $V_2$};
\end{tikzpicture}

\begin{tikzpicture}[scale=.45]
\draw(0,0)--(8,8)--(16,0);
\draw(1,1)--(2,0);
\draw (2.5,2.5)--(5,0);
\draw (4,1)--(3,0);
\draw (4,4)--(8,0);
\draw (7,1)--(6,0);
\draw (7,7)--(14,0);
\draw (11.5,2.5) -- (9,0);
\draw (10,1)--(11,0);
\draw (13,1)--(12,0);
\draw[fill] (8,8) circle [radius = .1];
\node[left] at (8,8) {$a$};
\draw[fill] (7,7) circle [radius = .1];
\node[left] at (7,7) {$b$};
\draw[fill] (4,4) circle [radius = .1];
\node[left] at (4,4) {$c$};
\draw[fill] (2.5,2.5) circle [radius = .1];
\node[left] at (2.5,2.5) {$d$};
\draw[fill] (7,1) circle [radius = .1];
\node[left] at (7,1) {$e$};
\draw[fill] (11.5,2.5) circle [radius = .1];
\node[right] at (11.7,2.5) {$f$};
\draw[fill] (1,1) circle [radius = .1];
\node[left] at (1,1) {$g$};
\draw[fill] (4,1) circle [radius = .1];
\node[right] at (4,1) {$h$};
\draw[fill] (10,1) circle [radius = .1];
\node[left] at (10,1) {$i$};
\draw[fill] (13,1) circle [radius = .1];
\node[right] at (13.2,1) {$j$};
\draw[ultra thick] (0,0)--(2.5,2.5);
\draw[ultra thick, dashed] (2.5,2.5)--(7,7)--(14,0);
\draw [ultra thick] (6,0)--(7,1)--(8,0);
\draw[ultra thick] (9,0)--(10,1)--(11,0);
\node[below, align=center] at (8,0) {The $b$-maintaining path system $\bar{\P_1}$  \\ realizing $\bar{V_1} = \{b,e,i\}$  };
\end{tikzpicture}$\qquad$\begin{tikzpicture}[scale=.45]
\draw(0,0)--(8,8)--(16,0);
\draw(1,1)--(2,0);
\draw (2.5,2.5)--(5,0);
\draw (4,1)--(3,0);
\draw (4,4)--(8,0);
\draw (7,1)--(6,0);
\draw (7,7)--(14,0);
\draw (11.5,2.5) -- (9,0);
\draw (10,1)--(11,0);
\draw (13,1)--(12,0);
\draw[fill] (8,8) circle [radius = .1];
\node[left] at (8,8) {$a$};
\draw[fill] (7,7) circle [radius = .1];
\node[left] at (7,7) {$b$};
\draw[fill] (4,4) circle [radius = .1];
\node[left] at (4,4) {$c$};
\draw[fill] (2.5,2.5) circle [radius = .1];
\node[left] at (2.5,2.5) {$d$};
\draw[fill] (7,1) circle [radius = .1];
\node[left] at (7,1) {$e$};
\draw[fill] (11.5,2.5) circle [radius = .1];
\node[right] at (11.7,2.5) {$f$};
\draw[fill] (1,1) circle [radius = .1];
\node[left] at (1,1) {$g$};
\draw[fill] (4,1) circle [radius = .1];
\node[right] at (4,1) {$h$};
\draw[fill] (10,1) circle [radius = .1];
\node[left] at (10,1) {$i$};
\draw[fill] (13,1) circle [radius = .1];
\node[right] at (13.2,1) {$j$};
\draw[ultra thick, dashed] (4,4)--(8,0);
\draw[ultra thick] (4,4)--(8,8)--(16,0);
\draw[ultra thick] (0,0)--(1,1)--(2,0);
\draw[ultra thick] (3,0)--(4,1)--(5,0);
\draw[ultra thick] (12,0)--(13,1)--(14,0);
\node[below, align=center] at (8,0) {The path system $\bar{\P_2}$ realizing \\ $\bar{V_2} = \{a,g,h,j\}$ };
\end{tikzpicture}
\caption{Proof of Lemma \ref{lem:compressed}, The first row of trees contain path systems -- one whose top-set is $b$-nonmaintaining, and one without $b$ or $c$ in its top-set, but with $d$ not blocked in its top-set. The second row of trees are the path systems obtained by performing the operation in the proof of Lemma \ref{lem:compressed}. Note that the representation given by the path systems in the second row is $df$-compressed.}
\label{fig:dcompression}
\end{figure}

Let $P \in \P_1$ with top-most node $b$. Let $\hat{P}$ be the path from $b$ to a leaf below $f$ contained in $P$. Let $\bar{P}$ be a path from $d$ to a leaf descended from $d$ that does not contain any nodes in $V_2$; this exists since $d$ is not blocked in $V_2$. Let $P' = \hat{P} \cup \{bc, cd \} \cup \bar{P}$. Then 
\[
\bar{\P_1} = (\P_1 - (\{P\} \cup \P_1^d \cup \P_1^e))\cup\{P'\} \cup \P_2^d \cup \P_2^e
\]
 realizes $\bar{V_1}$. Also, $d$ is not blocked in $\bar{V_1}$, so $[\bar{\P_1}]$ is maintaining.
 
 If there is no path in $\P_2$ that contains edges $cd$ or $ce$, then it is clear that 
 \[
\bar{\P_2} =  (\P_2 - (\P_2^d \cup \P_2^e)) \cup \P_1^d \cup \P_1^e
\]
realizes $\bar{V_2}$. Otherwise, suppose that $Q \in \P_2$ is a path that contains $cd$ or $ce$. Let $\hat{Q}$ be the path contained in $Q$ without the edges in the $c$-subtree. Let $P$ be the path in $\P_1$ with top-most node $b$, and let $\tilde{P}$ be the path contained in $P$ from $e$ to a leaf descended from $e$; this exists since $d$ is blocked in $V_1$. Let $Q' = \hat{Q} \cup \{ce\} \cup \tilde{P}$. Then
 \[
 \bar{\P_2} = (\P_2 - (\{Q\} \cup \P_2^d \cup \P_2^e))\cup\{Q'\} \cup \P_1^d \cup \P_1^e
\]
realizes $\bar{V_2}$, as needed. Since $b$ and $c \not\in \bar{V_2}$, $[\bar{\P_2}]$ is maintaining. 

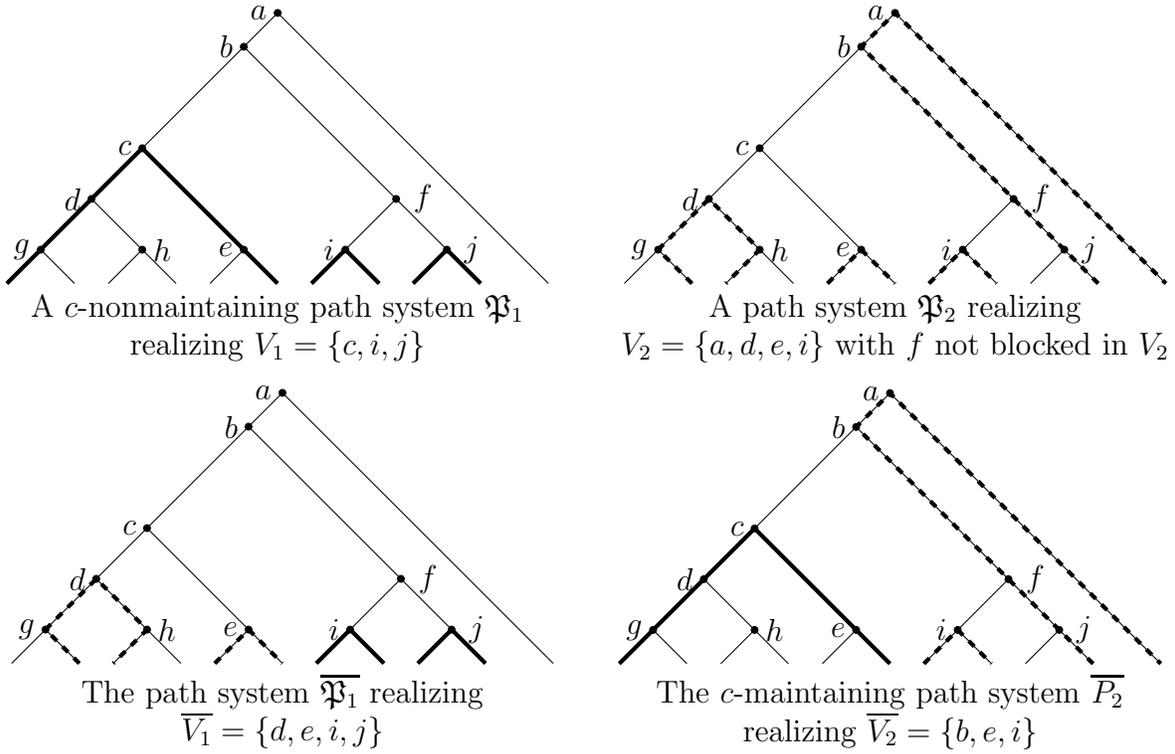
\begin{figure}

\begin{center}
\begin{tikzpicture}[scale=.45]
\draw(0,0)--(8,8)--(16,0);
\draw(1,1)--(2,0);
\draw (2.5,2.5)--(5,0);
\draw (4,1)--(3,0);
\draw (4,4)--(8,0);
\draw (7,1)--(6,0);
\draw (7,7)--(14,0);
\draw (11.5,2.5) -- (9,0);
\draw (10,1)--(11,0);
\draw (13,1)--(12,0);
\draw[fill] (8,8) circle [radius = .1];
\node[left] at (8,8) {$a$};
\draw[fill] (7,7) circle [radius = .1];
\node[left] at (7,7) {$b$};
\draw[fill] (4,4) circle [radius = .1];
\node[left] at (4,4) {$c$};
\draw[fill] (2.5,2.5) circle [radius = .1];
\node[left] at (2.5,2.5) {$d$};
\draw[fill] (7,1) circle [radius = .1];
\node[left] at (7,1) {$e$};
\draw[fill] (11.5,2.5) circle [radius = .1];
\node[right] at (11.7,2.5) {$f$};
\draw[fill] (1,1) circle [radius = .1];
\node[left] at (1,1) {$g$};
\draw[fill] (4,1) circle [radius = .1];
\node[right] at (4,1) {$h$};
\draw[fill] (10,1) circle [radius = .1];
\node[left] at (10,1) {$i$};
\draw[fill] (13,1) circle [radius = .1];
\node[right] at (13.2,1) {$j$};
\draw [ultra thick] (0,0)--(4,4)--(8,0);
\draw [ultra thick] (9,0)--(10,1)--(11,0);
\draw [ultra thick] (12,0)--(13,1)--(14,0);
\node[below, align=center] at (8,0) {A $c$-nonmaintaining path system $\P_1$  
\\ realizing $V_1 = \{c,i,j\}$ };
\end{tikzpicture}$\qquad$\begin{tikzpicture}[scale=.45]
\draw(0,0)--(8,8)--(16,0);
\draw(1,1)--(2,0);
\draw (2.5,2.5)--(5,0);
\draw (4,1)--(3,0);
\draw (4,4)--(8,0);
\draw (7,1)--(6,0);
\draw (7,7)--(14,0);
\draw (11.5,2.5) -- (9,0);
\draw (10,1)--(11,0);
\draw (13,1)--(12,0);
\draw[fill] (8,8) circle [radius = .1];
\node[left] at (8,8) {$a$};
\draw[fill] (7,7) circle [radius = .1];
\node[left] at (7,7) {$b$};
\draw[fill] (4,4) circle [radius = .1];
\node[left] at (4,4) {$c$};
\draw[fill] (2.5,2.5) circle [radius = .1];
\node[left] at (2.5,2.5) {$d$};
\draw[fill] (7,1) circle [radius = .1];
\node[left] at (7,1) {$e$};
\draw[fill] (11.5,2.5) circle [radius = .1];
\node[right] at (11.7,2.5) {$f$};
\draw[fill] (1,1) circle [radius = .1];
\node[left] at (1,1) {$g$};
\draw[fill] (4,1) circle [radius = .1];
\node[right] at (4,1) {$h$};
\draw[fill] (10,1) circle [radius = .1];
\node[left] at (10,1) {$i$};
\draw[fill] (13,1) circle [radius = .1];
\node[right] at (13.2,1) {$j$};
\draw[ultra thick, dashed] (14,0)--(7,7)--(8,8)--(16,0);
\draw[ultra thick, dashed] (2,0)--(1,1)--(2.5,2.5)--(4,1)--(3,0);
\draw[ultra thick, dashed] (6,0)--(7,1)--(8,0);
\draw[ultra thick, dashed] (9,0)--(10,1)--(11,0);
\node[below, align=center] at (8,0) {A path system $\P_2$ realizing \\
$V_2 = \{a,d,e,i\}$  with $f$ not blocked in $V_2$};
\end{tikzpicture}

\begin{tikzpicture}[scale=.45]
\draw(0,0)--(8,8)--(16,0);
\draw(1,1)--(2,0);
\draw (2.5,2.5)--(5,0);
\draw (4,1)--(3,0);
\draw (4,4)--(8,0);
\draw (7,1)--(6,0);
\draw (7,7)--(14,0);
\draw (11.5,2.5) -- (9,0);
\draw (10,1)--(11,0);
\draw (13,1)--(12,0);
\draw[fill] (8,8) circle [radius = .1];
\node[left] at (8,8) {$a$};
\draw[fill] (7,7) circle [radius = .1];
\node[left] at (7,7) {$b$};
\draw[fill] (4,4) circle [radius = .1];
\node[left] at (4,4) {$c$};
\draw[fill] (2.5,2.5) circle [radius = .1];
\node[left] at (2.5,2.5) {$d$};
\draw[fill] (7,1) circle [radius = .1];
\node[left] at (7,1) {$e$};
\draw[fill] (11.5,2.5) circle [radius = .1];
\node[right] at (11.7,2.5) {$f$};
\draw[fill] (1,1) circle [radius = .1];
\node[left] at (1,1) {$g$};
\draw[fill] (4,1) circle [radius = .1];
\node[right] at (4,1) {$h$};
\draw[fill] (10,1) circle [radius = .1];
\node[left] at (10,1) {$i$};
\draw[fill] (13,1) circle [radius = .1];
\node[right] at (13.2,1) {$j$};
\draw[ultra thick, dashed] (2,0)--(1,1)--(2.5,2.5)--(4,1)--(3,0);
\draw[ultra thick, dashed] (6,0)--(7,1)--(8,0);
\draw [ultra thick] (9,0)--(10,1)--(11,0);
\draw [ultra thick] (12,0)--(13,1)--(14,0);
\node[below, align=center] at (8,0) {The path system $\bar{\P_1}$ realizing \\$\bar{V_1} = \{d,e,i,j\}$};
\end{tikzpicture}$\qquad$\begin{tikzpicture}[scale=.45]
\draw(0,0)--(8,8)--(16,0);
\draw(1,1)--(2,0);
\draw (2.5,2.5)--(5,0);
\draw (4,1)--(3,0);
\draw (4,4)--(8,0);
\draw (7,1)--(6,0);
\draw (7,7)--(14,0);
\draw (11.5,2.5) -- (9,0);
\draw (10,1)--(11,0);
\draw (13,1)--(12,0);
\draw[fill] (8,8) circle [radius = .1];
\node[left] at (8,8) {$a$};
\draw[fill] (7,7) circle [radius = .1];
\node[left] at (7,7) {$b$};
\draw[fill] (4,4) circle [radius = .1];
\node[left] at (4,4) {$c$};
\draw[fill] (2.5,2.5) circle [radius = .1];
\node[left] at (2.5,2.5) {$d$};
\draw[fill] (7,1) circle [radius = .1];
\node[left] at (7,1) {$e$};
\draw[fill] (11.5,2.5) circle [radius = .1];
\node[right] at (11.7,2.5) {$f$};
\draw[fill] (1,1) circle [radius = .1];
\node[left] at (1,1) {$g$};
\draw[fill] (4,1) circle [radius = .1];
\node[right] at (4,1) {$h$};
\draw[fill] (10,1) circle [radius = .1];
\node[left] at (10,1) {$i$};
\draw[fill] (13,1) circle [radius = .1];
\node[right] at (13.2,1) {$j$};
\draw[ultra thick, dashed] (14,0)--(7,7)--(8,8)--(16,0);
\draw[ultra thick, dashed] (9,0)--(10,1)--(11,0);
\draw [ultra thick] (0,0)--(4,4)--(8,0);
\node[below, align=center] at (8,0) {The $c$-maintaining path system $\bar{P_2}$ 
 \\realizing $\bar{V_2} = \{b,e,i\}$  };
\end{tikzpicture}
\end{center}
\caption{Proof of Lemma \ref{lem:compressed}, The first row of trees contain path systems -- one whose top-set is $c$-nonmaintaining, and one without $b$ or $c$ in its top-set, but with $f$ not blocked in its top-set. The second row of trees are the path systems obtained by performing the operation in the proof of Lemma \ref{lem:compressed}. Note that the representation given by the path systems in the second row is $df$-compressed.}
\label{fig:fcompression}
\end{figure}

This operation preserves the number of times each internal node is a top-most node. So, $\bv = [\bar{\P_1}] + [\bar{\P_2}] + [\P_3] + \dots + [\P_m]$ is a representation of $\bv$ using fewer nonmaintaining vertices, and $[\P_1] + \dots + [\P_m]$ is not minimal.

If $[\P_1] + \dots + [\P_m]$ is not $f$-compressed, we proceed by a similar argument. Without loss of generality, we may assume that $[\P_1]$ is $c$-nonmaintaining and that $f$ is not blocked in $[\P_2]$. Then we claim that $\bar{V_1} = (V_1 - (V_1^c)) \cup V_2^c $ and $\bar{V_2} = (V_2 - (V_2^c)) \cup V_1^c$ are valid top-sets in $T$ with associated path systems $\bar{\P_1}$ and $\bar{\P_2}$ respectively. Furthermore, we claim that $[\bar{\P_1}]$ and $[\bar{\P_2}]$ are maintaining. Figure \ref{fig:fcompression} shows an example of a path system that is not $f$-compressed and the path system obtained from it by performing this operation.

Since $f$ is not blocked in $V_2$, we may assume that for all $P \in \P_2$, $bc \not\in P$. This is because $b \not\in V_2$, and any $P \in \P_2$ with top-most node above $b$ may pass through the $f$-subtree instead of the $c$-subtree since $f$ is not blocked in $V_2$. So, in both $\P_1$ and $\P_2$, all paths that intersect the $c$-subtree are contained entirely within the $c$-subtree. So $\bar{\P_1} = (\P_1 - \P_1^c) \cup \P_2^c$ and $\bar{\P_2} = (\P_2 - \P_2^c) \cup \P_1^c$ are path systems that realize $\bar{V_1}$ and $\bar{V_2}$, respectively. Since $b,c  \not\in \bar{V_1}$, $[\bar{\P_1}]$ is maintaining. Furthermore, since $f$ is not blocked in $V_2$, and since $\bar{V_2^f} = V_2^f$, $f$ is not blocked in $\bar{V_2}$ and $[\bar{\P_2}]$ is maintaining.

This operation preserves the number of times each internal node is used as a top-most node. So $\bv = [\bar{\P_1}] + [\bar{\P_2}] + [\P_3] + \dots + [\P_m]$ is a representation of $[\P]$ using fewer $c$-nonmaintaining vertices. 
\end{proof}

\begin{lemma}\label{lem:minreps}
Let $\bv, \bu \in mR_T \cap \Z^{n-1}$ such that $v_b + v_c = u_b + u_c$ and $v_x = u_x$ for all $x \neq b,c$. Let $\bv = [\P_1] + \dots + [\P_m]$ be a $df$-compressed representation of $\bv$ and let $\bu = [\Q_1] + \dots + [\Q_m]$ be any representation of $\bu$. If the multiset $\{ [\Q_1] , \dots , [\Q_m]\}$ contains fewer $b$-nonmaintaining or $c$-nonmaintaining vertices than the multiset $\{[\P_1] , \dots , [\P_m]\}$, then $[\P_1] + \dots + [\P_m]$ is not a minimal representation of $\bv$.
\end{lemma}

\begin{proof}
First suppose that $\{ [\Q_1] , \dots , [\Q_m]\}$ contains fewer $b$-nonmaintaining vertices than $\{[\P_1] , \dots , [\P_m]\}$. Then without loss of generality, let $[\P]_1$ be $b$-nonmaintaining. For all $i$, let $V_i$ denote the top-set corresponding to $\P_i$ and let $U_i$ denote the top-set corresponding to $\Q_i$. Figure \ref{fig:minreps1} depicts an example of this case and of the procedure that we describe in the following proof.

Since $[\P_1] + \dots + [\P_m]$ is $d$-compressed, for all $V_i$ with $b,c \notin V_i$, $d$ is blocked in $V_i$. Without loss of generality, let $[\P_1] , \dots , [\P_r]$ and $[\Q_1] , \dots , [\Q_{r'}]$ be the $b$-nonmaintaining vertices where $r' < r$. Let $[\P_{r+1}], \dots, [\P_s]$ and $[\Q_{r'+1}], \dots, [\Q_s]$ be the rest of the vertices with $b$ or $c$ coordinate equal to 1. Note that by assumption, there are the same number of vertices summed in the representations of $\bu$ and $\bv$.

Let $\bar{V_i} = (V_i - V_i^d) \cup U_i^d$ for all $i$. We claim that each of the $\bar{V_i}$ are valid top-sets, and that the collection of all corresponding $[\bar{\P}_i]$ has the same number of $b$-nonmaintaining vertices as the $[\Q_i]$, and the same number of $c$-nonmaintaining vertices as the $[\P_i]$.

First, let $i \leq r'$. Then since $[\P]_i$ and $[\Q]_i$ are both $b$-nonmaintaining, $d$ is blocked in both $V_i$ and $U_i$. So $cd \not\in \P_i, \Q_i$. Therefore, all paths in $\P_i$ and $\Q_i$ that intersect the $d$-subtree are contained entirely within the $d$-subtree. So $\bar{\P_i} = (\P_i - \P_i^d) \cup \Q_i^d$ is a path system that realizes $\bar{V_i}$, as needed.

Next, let $r' < i \leq s$. Then $[\Q]_i$ either is $b$-maintaining or has $c$-coordinate equal to 1. In either case, $d$ is not blocked in $U_i$. So there exists a path $Q$ from $d$ to a leaf descended from $d$ with no node along $Q$ in $U_i$. Let $P \in \P_i$ be the path with either $b$ or $c$ as its top-most node.

If $P$ has $b$ as its top-most node, then let $\hat{P}$ be the path from $b$ to a node below $f$ that is contained in $P$. Let $P' = \hat{P} \cup \{bc, cd\} \cup Q$. Then
\[
\bar{P_i} = (\P_i - (\{ P \} \cup \P_i^d)) \cup \{ P' \} \cup \Q_i^d
\]
realizes $\bar{V_i}$.

If $P$ has $c$ as its top-most node, then let $\hat{P}$ be the path from $c$ to a leaf below $e$ that is contained in $P$. Let $P' - \hat{P} \cup \{cd\} \cup Q$. Then
\[
\bar{P_i} = (\P_i - (\{ P \} \cup \P_i^d)) \cup \{ P' \} \cup \Q_i^d
\]
realizes $ \bar{V_i}$.

Note that in all cases when $r' < i \leq s$, $d$ is not blocked in $\bar{V_i}$. Since $r'<r$, this means that there are fewer $b$-nonmaintaining vertices in $\{[\bar{\P_1}], \dots , [\bar{\P_s}] \}$ than in $\{[\P_1], \dots , [\P_s]\}$. Furthermore, since the paths in the $f$-subtrees remain unchanged, this operation cannot create new $c$-nonmaintaining vertices.

Finally, let $i>s$. Then $b,c \not\in V_i, U_i$. Since $d$ is blocked in every $V_i$, all paths in $\P_i$ that intersect the $d$-subtree are contained entirely in the $d$-subtree. So $\bar{\P} = (\P_i - \P_i^d) \cup \Q_i^d$ is a path system that realizes $\bar{V_i}$.

Since
\[
\sum_{i=1}^m [\P_i]^d = \sum_{i=1}^m [\Q_i]^d,
\]
and since $\bar{[\P_i]}^d = [\Q_i]^d$ for all $i$, $[\bar{\P_1}] + \dots + [\bar{\P_m}]$ is a representation of $\bv$ using fewer nonmaintaining vertices than $[\P_1] + \dots + [\P_m]$. So $[\P_1] + \dots + [\P_m]$ is not minimal. An example of the operation used to obtain $[\bar{\P_1}],\dots ,[\bar{\P_m}]$ is illustrated in Figure \ref{fig:minreps1}.

\begin{sidewaysfigure}

\begin{center}
\begin{tikzpicture}[scale=.45]
\draw(0,0)--(8,8)--(16,0);
\draw(1,1)--(2,0);
\draw (2.5,2.5)--(5,0);
\draw (4,1)--(3,0);
\draw (4,4)--(8,0);
\draw (7,1)--(6,0);
\draw (7,7)--(14,0);
\draw (11.5,2.5) -- (9,0);
\draw (10,1)--(11,0);
\draw (13,1)--(12,0);
\draw[fill] (8,8) circle [radius = .1];
\node[left] at (8,8) {$a$};
\draw[fill] (7,7) circle [radius = .1];
\node[left] at (7,7) {$b$};
\draw[fill] (4,4) circle [radius = .1];
\node[left] at (4,4) {$c$};
\draw[fill] (2.5,2.5) circle [radius = .1];
\node[left] at (2.5,2.5) {$d$};
\draw[fill] (7,1) circle [radius = .1];
\node[left] at (7,1) {$e$};
\draw[fill] (11.5,2.5) circle [radius = .1];
\node[right] at (11.7,2.5) {$f$};
\draw[fill] (1,1) circle [radius = .1];
\node[left] at (1,1) {$g$};
\draw[fill] (4,1) circle [radius = .1];
\node[right] at (4,1) {$h$};
\draw[fill] (10,1) circle [radius = .1];
\node[left] at (10,1) {$i$};
\draw[fill] (13,1) circle [radius = .1];
\node[right] at (13.2,1) {$j$};
\draw[ultra thick] (2,0)--(1,1)--(2.5,2.5)--(4,1)--(3,0);
\draw[ultra thick] (8,0)--(4,4)--(7,7)--(14,0);
\node[below, align=center] at (8,0) {A path system realizing 
$V_1 = \{b,d\}$ \\
where $[\P_1]$ is $b$-nonmaintaining};
\end{tikzpicture}$\qquad$\begin{tikzpicture}[scale=.45]
\draw(0,0)--(8,8)--(16,0);
\draw(1,1)--(2,0);
\draw (2.5,2.5)--(5,0);
\draw (4,1)--(3,0);
\draw (4,4)--(8,0);
\draw (7,1)--(6,0);
\draw (7,7)--(14,0);
\draw (11.5,2.5) -- (9,0);
\draw (10,1)--(11,0);
\draw (13,1)--(12,0);
\draw[fill] (8,8) circle [radius = .1];
\node[left] at (8,8) {$a$};
\draw[fill] (7,7) circle [radius = .1];
\node[left] at (7,7) {$b$};
\draw[fill] (4,4) circle [radius = .1];
\node[left] at (4,4) {$c$};
\draw[fill] (2.5,2.5) circle [radius = .1];
\node[left] at (2.5,2.5) {$d$};
\draw[fill] (7,1) circle [radius = .1];
\node[left] at (7,1) {$e$};
\draw[fill] (11.5,2.5) circle [radius = .1];
\node[right] at (11.7,2.5) {$f$};
\draw[fill] (1,1) circle [radius = .1];
\node[left] at (1,1) {$g$};
\draw[fill] (4,1) circle [radius = .1];
\node[right] at (4,1) {$h$};
\draw[fill] (10,1) circle [radius = .1];
\node[left] at (10,1) {$i$};
\draw[fill] (13,1) circle [radius = .1];
\node[right] at (13.2,1) {$j$};
\draw[ultra thick] (5,0)--(2.5,2.5)--(4,4)--(7,1)--(6,0);
\draw[ultra thick] (9,0)--(10,1)--(11,0);
\draw[ultra thick] (12,0)--(13,1)--(14,0);
\node[below, align=center] at (8,0) {A path system realizing\\
$V_2 = \{c,i,j\}$};
\end{tikzpicture}$\qquad$\begin{tikzpicture}[scale=.45]
\draw(0,0)--(8,8)--(16,0);
\draw(1,1)--(2,0);
\draw (2.5,2.5)--(5,0);
\draw (4,1)--(3,0);
\draw (4,4)--(8,0);
\draw (7,1)--(6,0);
\draw (7,7)--(14,0);
\draw (11.5,2.5) -- (9,0);
\draw (10,1)--(11,0);
\draw (13,1)--(12,0);
\draw[fill] (8,8) circle [radius = .1];
\node[left] at (8,8) {$a$};
\draw[fill] (7,7) circle [radius = .1];
\node[left] at (7,7) {$b$};
\draw[fill] (4,4) circle [radius = .1];
\node[left] at (4,4) {$c$};
\draw[fill] (2.5,2.5) circle [radius = .1];
\node[left] at (2.5,2.5) {$d$};
\draw[fill] (7,1) circle [radius = .1];
\node[left] at (7,1) {$e$};
\draw[fill] (11.5,2.5) circle [radius = .1];
\node[right] at (11.7,2.5) {$f$};
\draw[fill] (1,1) circle [radius = .1];
\node[left] at (1,1) {$g$};
\draw[fill] (4,1) circle [radius = .1];
\node[right] at (4,1) {$h$};
\draw[fill] (10,1) circle [radius = .1];
\node[left] at (10,1) {$i$};
\draw[fill] (13,1) circle [radius = .1];
\node[right] at (13.2,1) {$j$};
\draw[ultra thick] (0,0)--(1,1)--(2,0);
\draw[ultra thick] (3,0)--(4,1)--(5,0);
\draw[ultra thick] (6,0)--(7,1)--(8,0);
\draw[ultra thick] (14,0)--(7,7)--(8,8)--(16,0);
\node[below, align=center] at (8,0) {A path system realizing
$V_3 = \{a,e,g,h\}$ \\
with $d$ blocked in $V_3$};
\end{tikzpicture}

\begin{tikzpicture}[scale=.45]
\draw(0,0)--(8,8)--(16,0);
\draw(1,1)--(2,0);
\draw (2.5,2.5)--(5,0);
\draw (4,1)--(3,0);
\draw (4,4)--(8,0);
\draw (7,1)--(6,0);
\draw (7,7)--(14,0);
\draw (11.5,2.5) -- (9,0);
\draw (10,1)--(11,0);
\draw (13,1)--(12,0);
\draw[fill] (8,8) circle [radius = .1];
\node[left] at (8,8) {$a$};
\draw[fill] (7,7) circle [radius = .1];
\node[left] at (7,7) {$b$};
\draw[fill] (4,4) circle [radius = .1];
\node[left] at (4,4) {$c$};
\draw[fill] (2.5,2.5) circle [radius = .1];
\node[left] at (2.5,2.5) {$d$};
\draw[fill] (7,1) circle [radius = .1];
\node[left] at (7,1) {$e$};
\draw[fill] (11.5,2.5) circle [radius = .1];
\node[right] at (11.7,2.5) {$f$};
\draw[fill] (1,1) circle [radius = .1];
\node[left] at (1,1) {$g$};
\draw[fill] (4,1) circle [radius = .1];
\node[right] at (4,1) {$h$};
\draw[fill] (10,1) circle [radius = .1];
\node[left] at (10,1) {$i$};
\draw[fill] (13,1) circle [radius = .1];
\node[right] at (13.2,1) {$j$};
\draw[ultra thick, dashed] (0,0)--(7,7)--(11.5,2.5)--(9,0);
\draw[ultra thick, dashed] (3,0)--(4,1)--(5,0);
\draw[ultra thick, dashed] (6,0)--(7,1)--(8,0);
\draw[ultra thick, dashed] (12,0)--(13,1)--(14,0);
\node[below, align=center] at (8,0) {A path system realizing 
$U_1 = \{b,h,e,j\}$  \\
where $[\Q_1]$ is $b$-maintaining};
\end{tikzpicture}$\qquad$\begin{tikzpicture}[scale=.45]
\draw(0,0)--(8,8)--(16,0);
\draw(1,1)--(2,0);
\draw (2.5,2.5)--(5,0);
\draw (4,1)--(3,0);
\draw (4,4)--(8,0);
\draw (7,1)--(6,0);
\draw (7,7)--(14,0);
\draw (11.5,2.5) -- (9,0);
\draw (10,1)--(11,0);
\draw (13,1)--(12,0);
\draw[fill] (8,8) circle [radius = .1];
\node[left] at (8,8) {$a$};
\draw[fill] (7,7) circle [radius = .1];
\node[left] at (7,7) {$b$};
\draw[fill] (4,4) circle [radius = .1];
\node[left] at (4,4) {$c$};
\draw[fill] (2.5,2.5) circle [radius = .1];
\node[left] at (2.5,2.5) {$d$};
\draw[fill] (7,1) circle [radius = .1];
\node[left] at (7,1) {$e$};
\draw[fill] (11.5,2.5) circle [radius = .1];
\node[right] at (11.7,2.5) {$f$};
\draw[fill] (1,1) circle [radius = .1];
\node[left] at (1,1) {$g$};
\draw[fill] (4,1) circle [radius = .1];
\node[right] at (4,1) {$h$};
\draw[fill] (10,1) circle [radius = .1];
\node[left] at (10,1) {$i$};
\draw[fill] (13,1) circle [radius = .1];
\node[right] at (13.2,1) {$j$};
\draw[ultra thick, dashed] (5,0)--(2.5,2.5)--(7,7)--(14,0);
\draw[ultra thick, dashed] (0,0)--(1,1)--(2,0);
\draw[ultra thick, dashed] (9,0)--(10,1)--(11,0);
\node[below, align=center] at (8,0) {A path system realizing
$U_2 = \{b,g,i\}$ \\
where $[\Q_2]$ is $b$-maintaining};
\end{tikzpicture}$\qquad$\begin{tikzpicture}[scale=.45]
\draw(0,0)--(8,8)--(16,0);
\draw(1,1)--(2,0);
\draw (2.5,2.5)--(5,0);
\draw (4,1)--(3,0);
\draw (4,4)--(8,0);
\draw (7,1)--(6,0);
\draw (7,7)--(14,0);
\draw (11.5,2.5) -- (9,0);
\draw (10,1)--(11,0);
\draw (13,1)--(12,0);
\draw[fill] (8,8) circle [radius = .1];
\node[left] at (8,8) {$a$};
\draw[fill] (7,7) circle [radius = .1];
\node[left] at (7,7) {$b$};
\draw[fill] (4,4) circle [radius = .1];
\node[left] at (4,4) {$c$};
\draw[fill] (2.5,2.5) circle [radius = .1];
\node[left] at (2.5,2.5) {$d$};
\draw[fill] (7,1) circle [radius = .1];
\node[left] at (7,1) {$e$};
\draw[fill] (11.5,2.5) circle [radius = .1];
\node[right] at (11.7,2.5) {$f$};
\draw[fill] (1,1) circle [radius = .1];
\node[left] at (1,1) {$g$};
\draw[fill] (4,1) circle [radius = .1];
\node[right] at (4,1) {$h$};
\draw[fill] (10,1) circle [radius = .1];
\node[left] at (10,1) {$i$};
\draw[fill] (13,1) circle [radius = .1];
\node[right] at (13.2,1) {$j$};
\draw[ultra thick, dashed] (8,0)--(4,4)--(8,8)--(16,0);
\draw[ultra thick, dashed] (2,0)--(1,1)--(2.5,2.5)--(4,1)--(3,0);
\node[below, align=center] at (8,0) {A path system realizing\\
$U_3 = \{a,d\}$};
\end{tikzpicture}

\begin{tikzpicture}[scale=.45]
\draw(0,0)--(8,8)--(16,0);
\draw(1,1)--(2,0);
\draw (2.5,2.5)--(5,0);
\draw (4,1)--(3,0);
\draw (4,4)--(8,0);
\draw (7,1)--(6,0);
\draw (7,7)--(14,0);
\draw (11.5,2.5) -- (9,0);
\draw (10,1)--(11,0);
\draw (13,1)--(12,0);
\draw[fill] (8,8) circle [radius = .1];
\node[left] at (8,8) {$a$};
\draw[fill] (7,7) circle [radius = .1];
\node[left] at (7,7) {$b$};
\draw[fill] (4,4) circle [radius = .1];
\node[left] at (4,4) {$c$};
\draw[fill] (2.5,2.5) circle [radius = .1];
\node[left] at (2.5,2.5) {$d$};
\draw[fill] (7,1) circle [radius = .1];
\node[left] at (7,1) {$e$};
\draw[fill] (11.5,2.5) circle [radius = .1];
\node[right] at (11.7,2.5) {$f$};
\draw[fill] (1,1) circle [radius = .1];
\node[left] at (1,1) {$g$};
\draw[fill] (4,1) circle [radius = .1];
\node[right] at (4,1) {$h$};
\draw[fill] (10,1) circle [radius = .1];
\node[left] at (10,1) {$i$};
\draw[fill] (13,1) circle [radius = .1];
\node[right] at (13.2,1) {$j$};
\draw[ultra thick] (4,4)--(7,7)--(11.5,2.5)--(9,0);
\draw[ultra thick, dashed] (0,0)--(4,4);
\draw[ultra thick, dashed] (3,0)--(4,1)--(5,0);
\node[below, align=center] at (8,0) {A path system realizing\\
$\bar{V_1} = \{b,h\}$ where $[\bar{\P_1}]$ is $b$-maintaining};
\end{tikzpicture}$ \qquad$\begin{tikzpicture}[scale=.45]
\draw(0,0)--(8,8)--(16,0);
\draw(1,1)--(2,0);
\draw (2.5,2.5)--(5,0);
\draw (4,1)--(3,0);
\draw (4,4)--(8,0);
\draw (7,1)--(6,0);
\draw (7,7)--(14,0);
\draw (11.5,2.5) -- (9,0);
\draw (10,1)--(11,0);
\draw (13,1)--(12,0);
\draw[fill] (8,8) circle [radius = .1];
\node[left] at (8,8) {$a$};
\draw[fill] (7,7) circle [radius = .1];
\node[left] at (7,7) {$b$};
\draw[fill] (4,4) circle [radius = .1];
\node[left] at (4,4) {$c$};
\draw[fill] (2.5,2.5) circle [radius = .1];
\node[left] at (2.5,2.5) {$d$};
\draw[fill] (7,1) circle [radius = .1];
\node[left] at (7,1) {$e$};
\draw[fill] (11.5,2.5) circle [radius = .1];
\node[right] at (11.7,2.5) {$f$};
\draw[fill] (1,1) circle [radius = .1];
\node[left] at (1,1) {$g$};
\draw[fill] (4,1) circle [radius = .1];
\node[right] at (4,1) {$h$};
\draw[fill] (10,1) circle [radius = .1];
\node[left] at (10,1) {$i$};
\draw[fill] (13,1) circle [radius = .1];
\node[right] at (13.2,1) {$j$};
\draw[ultra thick] (2.5,2.5)--(4,4)--(7,1)--(6,0);
\draw[ultra thick] (9,0)--(10,1)--(11,0);
\draw[ultra thick] (12,0)--(13,1)--(14,0);
\draw[ultra thick, dashed] (0,0)--(1,1)--(2,0);
\draw[ultra thick, dashed] (5,0)--(2.5,2.5);
\node[below, align=center] at (8,0) {A path system realizing\\
$\bar{V_2} = \{c,g,i,j\}$ };
\end{tikzpicture}$\qquad$\begin{tikzpicture}[scale=.45]
\draw(0,0)--(8,8)--(16,0);
\draw(1,1)--(2,0);
\draw (2.5,2.5)--(5,0);
\draw (4,1)--(3,0);
\draw (4,4)--(8,0);
\draw (7,1)--(6,0);
\draw (7,7)--(14,0);
\draw (11.5,2.5) -- (9,0);
\draw (10,1)--(11,0);
\draw (13,1)--(12,0);
\draw[fill] (8,8) circle [radius = .1];
\node[left] at (8,8) {$a$};
\draw[fill] (7,7) circle [radius = .1];
\node[left] at (7,7) {$b$};
\draw[fill] (4,4) circle [radius = .1];
\node[left] at (4,4) {$c$};
\draw[fill] (2.5,2.5) circle [radius = .1];
\node[left] at (2.5,2.5) {$d$};
\draw[fill] (7,1) circle [radius = .1];
\node[left] at (7,1) {$e$};
\draw[fill] (11.5,2.5) circle [radius = .1];
\node[right] at (11.7,2.5) {$f$};
\draw[fill] (1,1) circle [radius = .1];
\node[left] at (1,1) {$g$};
\draw[fill] (4,1) circle [radius = .1];
\node[right] at (4,1) {$h$};
\draw[fill] (10,1) circle [radius = .1];
\node[left] at (10,1) {$i$};
\draw[fill] (13,1) circle [radius = .1];
\node[right] at (13.2,1) {$j$};
\draw[ultra thick] (6,0)--(7,1)--(8,0);
\draw[ultra thick] (14,0)--(7,7)--(8,8)--(16,0);
\draw[ultra thick, dashed] (2,0)--(1,1)--(2.5,2.5)--(4,1)--(3,0);
\node[below, align=center] at (8,0) {A path system realizing\\
$\bar{V_3} = \{a,d,e\}$};
\end{tikzpicture}
\end{center}
\caption{Proof of Lemma \ref{lem:minreps}. This figure illustrates the case where $\{ [\Q_1] , \dots , [\Q_m]\}$ contains fewer $b$-nonmaintaining vertices than $\{[\P_1] , \dots , [\P_m]\}$. The first row of trees are path systems that realize $\bv = [\P_1] + [\P_2] + [\P_3] \in 3R_T$. The second row of trees are path systems that realize $\bu = [\Q_1] + [\Q_2] + [\Q_3] \in 3R_T$ that satisfies the assumptions of the lemma. The third row of trees are a new set of path systems that realize $\bv$ using fewer $b$-nonmaintaining vertices, which we obtained by applying the procedure discussed in the proof of the lemma.}
\label{fig:minreps1}
\end{sidewaysfigure}

Now suppose that $\{ [\Q_1], \dots , [\Q_m]\}$ contains fewer $c$-nonmaintaining vertices than $\{[\P_1], \dots, [\P_m]\}$.   Figure \ref{fig:minreps2} depicts an example of this case and of the procedure that we describe in the following proof. Without loss of generality, let $[\P_1]$ be $c$-nonmaintaining.  Since $[\P_1] + \dots + [\P_m]$ is $f$-compressed, for all $V_i$ with $b,c \notin V_i$, $f$ is blocked in $V_i$. Without loss of generality, let $[\P_1], \dots , [\P_r]$ and $[\Q_1], \dots, [\Q_{r'}]$ be the $c$-nonmaintaining vertices where $r' < r$. Let $[\P_{r+1}], \dots, [\P_s]$ and $[\Q_{r'+1}], \dots [\Q_s]$ be the rest of the vertices with $b$ or $c$ coordinate equal to 1. Note that by assumption, there are the same number of vertices summed in the representations of $\bu$ and $\bv$.

\begin{sidewaysfigure}

\begin{center}
\begin{tikzpicture}[scale=.4]
\draw(0,0)--(8,8)--(16,0);
\draw(1,1)--(2,0);
\draw (2.5,2.5)--(5,0);
\draw (4,1)--(3,0);
\draw (4,4)--(8,0);
\draw (7,1)--(6,0);
\draw (7,7)--(14,0);
\draw (11.5,2.5) -- (9,0);
\draw (10,1)--(11,0);
\draw (13,1)--(12,0);
\draw[fill] (8,8) circle [radius = .1];
\node[left] at (8,8) {$a$};
\draw[fill] (7,7) circle [radius = .1];
\node[left] at (7,7) {$b$};
\draw[fill] (4,4) circle [radius = .1];
\node[left] at (4,4) {$c$};
\draw[fill] (2.5,2.5) circle [radius = .1];
\node[left] at (2.5,2.5) {$d$};
\draw[fill] (7,1) circle [radius = .1];
\node[left] at (7,1) {$e$};
\draw[fill] (11.5,2.5) circle [radius = .1];
\node[right] at (11.7,2.5) {$f$};
\draw[fill] (1,1) circle [radius = .1];
\node[left] at (1,1) {$g$};
\draw[fill] (4,1) circle [radius = .1];
\node[right] at (4,1) {$h$};
\draw[fill] (10,1) circle [radius = .1];
\node[left] at (10,1) {$i$};
\draw[fill] (13,1) circle [radius = .1];
\node[right] at (13.2,1) {$j$};
\draw[ultra thick] (2,0)--(1,1)--(4,4)--(7,1)--(6,0);
\draw[ultra thick] (9,0)--(11.5,2.5)--(14,0);
\node[below, align=center] at (8,0) {A path system realizing $V_1 = \{c,f\}$\\
 in which $[\P_1]$ is $c$-nonmaintaining};
\end{tikzpicture}$\quad$\begin{tikzpicture}[scale=.4]
\draw(0,0)--(8,8)--(16,0);
\draw(1,1)--(2,0);
\draw (2.5,2.5)--(5,0);
\draw (4,1)--(3,0);
\draw (4,4)--(8,0);
\draw (7,1)--(6,0);
\draw (7,7)--(14,0);
\draw (11.5,2.5) -- (9,0);
\draw (10,1)--(11,0);
\draw (13,1)--(12,0);
\draw[fill] (8,8) circle [radius = .1];
\node[left] at (8,8) {$a$};
\draw[fill] (7,7) circle [radius = .1];
\node[left] at (7,7) {$b$};
\draw[fill] (4,4) circle [radius = .1];
\node[left] at (4,4) {$c$};
\draw[fill] (2.5,2.5) circle [radius = .1];
\node[left] at (2.5,2.5) {$d$};
\draw[fill] (7,1) circle [radius = .1];
\node[left] at (7,1) {$e$};
\draw[fill] (11.5,2.5) circle [radius = .1];
\node[right] at (11.7,2.5) {$f$};
\draw[fill] (1,1) circle [radius = .1];
\node[left] at (1,1) {$g$};
\draw[fill] (4,1) circle [radius = .1];
\node[right] at (4,1) {$h$};
\draw[fill] (10,1) circle [radius = .1];
\node[left] at (10,1) {$i$};
\draw[fill] (13,1) circle [radius = .1];
\node[right] at (13.2,1) {$j$};
\draw[ultra thick] (0,0)--(1,1)--(2,0);
\draw[ultra thick] (5,0)--(2.5,2.5)--(4,4)--(8,0);
\node[below, align=center] at (8,0) {A path system realizing $V_2 = \{c,g\}$ \\
in which $[\P_2]$ is $c$-maintaining};
\end{tikzpicture}$\quad$\begin{tikzpicture}[scale=.4]
\draw(0,0)--(8,8)--(16,0);
\draw(1,1)--(2,0);
\draw (2.5,2.5)--(5,0);
\draw (4,1)--(3,0);
\draw (4,4)--(8,0);
\draw (7,1)--(6,0);
\draw (7,7)--(14,0);
\draw (11.5,2.5) -- (9,0);
\draw (10,1)--(11,0);
\draw (13,1)--(12,0);
\draw[fill] (8,8) circle [radius = .1];
\node[left] at (8,8) {$a$};
\draw[fill] (7,7) circle [radius = .1];
\node[left] at (7,7) {$b$};
\draw[fill] (4,4) circle [radius = .1];
\node[left] at (4,4) {$c$};
\draw[fill] (2.5,2.5) circle [radius = .1];
\node[left] at (2.5,2.5) {$d$};
\draw[fill] (7,1) circle [radius = .1];
\node[left] at (7,1) {$e$};
\draw[fill] (11.5,2.5) circle [radius = .1];
\node[right] at (11.7,2.5) {$f$};
\draw[fill] (1,1) circle [radius = .1];
\node[left] at (1,1) {$g$};
\draw[fill] (4,1) circle [radius = .1];
\node[right] at (4,1) {$h$};
\draw[fill] (10,1) circle [radius = .1];
\node[left] at (10,1) {$i$};
\draw[fill] (13,1) circle [radius = .1];
\node[right] at (13.2,1) {$j$};
\draw[ultra thick] (2,0)--(1,1)--(2.5,2.5)--(4,1)--(3,0);
\draw[ultra thick] (8,0)--(4,4)--(8,8)--(16,0);
\draw[ultra thick] (9,0)--(10,1)--(11,0);
\draw[ultra thick] (12,0)--(13,1)--(14,0);
\node[below, align=center] at (8,0) {A path system realizing\\
$V_3 = \{a,d,i,j\}$ in which $f$ is blocked};
\end{tikzpicture}

\begin{tikzpicture}[scale=.4]
\draw(0,0)--(8,8)--(16,0);
\draw(1,1)--(2,0);
\draw (2.5,2.5)--(5,0);
\draw (4,1)--(3,0);
\draw (4,4)--(8,0);
\draw (7,1)--(6,0);
\draw (7,7)--(14,0);
\draw (11.5,2.5) -- (9,0);
\draw (10,1)--(11,0);
\draw (13,1)--(12,0);
\draw[fill] (8,8) circle [radius = .1];
\node[left] at (8,8) {$a$};
\draw[fill] (7,7) circle [radius = .1];
\node[left] at (7,7) {$b$};
\draw[fill] (4,4) circle [radius = .1];
\node[left] at (4,4) {$c$};
\draw[fill] (2.5,2.5) circle [radius = .1];
\node[left] at (2.5,2.5) {$d$};
\draw[fill] (7,1) circle [radius = .1];
\node[left] at (7,1) {$e$};
\draw[fill] (11.5,2.5) circle [radius = .1];
\node[right] at (11.7,2.5) {$f$};
\draw[fill] (1,1) circle [radius = .1];
\node[left] at (1,1) {$g$};
\draw[fill] (4,1) circle [radius = .1];
\node[right] at (4,1) {$h$};
\draw[fill] (10,1) circle [radius = .1];
\node[left] at (10,1) {$i$};
\draw[fill] (13,1) circle [radius = .1];
\node[right] at (13.2,1) {$j$};
\draw[ultra thick, dashed] (0,0)--(1,1)--(2,0);
\draw[ultra thick, dashed] (5,0)--(2.5,2.5)--(4,4)--(8,0);
\draw[ultra thick, dashed] (9,0)--(10,1)--(11,0);
\node[below, align=center] at (8,0) {A path system realizing $U_1 = \{c,g,i\}$\\
 where $[\Q_1]$ is $c$-maintaining};
\end{tikzpicture}$\quad$\begin{tikzpicture}[scale=.4]
\draw(0,0)--(8,8)--(16,0);
\draw(1,1)--(2,0);
\draw (2.5,2.5)--(5,0);
\draw (4,1)--(3,0);
\draw (4,4)--(8,0);
\draw (7,1)--(6,0);
\draw (7,7)--(14,0);
\draw (11.5,2.5) -- (9,0);
\draw (10,1)--(11,0);
\draw (13,1)--(12,0);
\draw[fill] (8,8) circle [radius = .1];
\node[left] at (8,8) {$a$};
\draw[fill] (7,7) circle [radius = .1];
\node[left] at (7,7) {$b$};
\draw[fill] (4,4) circle [radius = .1];
\node[left] at (4,4) {$c$};
\draw[fill] (2.5,2.5) circle [radius = .1];
\node[left] at (2.5,2.5) {$d$};
\draw[fill] (7,1) circle [radius = .1];
\node[left] at (7,1) {$e$};
\draw[fill] (11.5,2.5) circle [radius = .1];
\node[right] at (11.7,2.5) {$f$};
\draw[fill] (1,1) circle [radius = .1];
\node[left] at (1,1) {$g$};
\draw[fill] (4,1) circle [radius = .1];
\node[right] at (4,1) {$h$};
\draw[fill] (10,1) circle [radius = .1];
\node[left] at (10,1) {$i$};
\draw[fill] (13,1) circle [radius = .1];
\node[right] at (13.2,1) {$j$};
\draw[ultra thick, dashed] (0,0)--(7,7)--(11.5,2.5)--(9,0);
\draw[ultra thick, dashed] (12,0)--(13,1)--(14,0);
\node[below, align=center] at (8,0) {A path system realizing $U_2 = \{b,j\}$\\
 in which $f$ is blocked};
\end{tikzpicture}$\quad$\begin{tikzpicture}[scale=.4]
\draw(0,0)--(8,8)--(16,0);
\draw(1,1)--(2,0);
\draw (2.5,2.5)--(5,0);
\draw (4,1)--(3,0);
\draw (4,4)--(8,0);
\draw (7,1)--(6,0);
\draw (7,7)--(14,0);
\draw (11.5,2.5) -- (9,0);
\draw (10,1)--(11,0);
\draw (13,1)--(12,0);
\draw[fill] (8,8) circle [radius = .1];
\node[left] at (8,8) {$a$};
\draw[fill] (7,7) circle [radius = .1];
\node[left] at (7,7) {$b$};
\draw[fill] (4,4) circle [radius = .1];
\node[left] at (4,4) {$c$};
\draw[fill] (2.5,2.5) circle [radius = .1];
\node[left] at (2.5,2.5) {$d$};
\draw[fill] (7,1) circle [radius = .1];
\node[left] at (7,1) {$e$};
\draw[fill] (11.5,2.5) circle [radius = .1];
\node[right] at (11.7,2.5) {$f$};
\draw[fill] (1,1) circle [radius = .1];
\node[left] at (1,1) {$g$};
\draw[fill] (4,1) circle [radius = .1];
\node[right] at (4,1) {$h$};
\draw[fill] (10,1) circle [radius = .1];
\node[left] at (10,1) {$i$};
\draw[fill] (13,1) circle [radius = .1];
\node[right] at (13.2,1) {$j$};
\draw[ultra thick, dashed] (0,0)--(2.5,2.5)--(5,0);
\draw[ultra thick, dashed] (8,0)--(4,4)--(8,8)--(16,0);
\draw[ultra thick, dashed] (11,0)--(10,1)--(11.5,2.5)--(13,1)--(12,0);
\node[below, align=center] at (8,0) {A path system realizing\\
$U_3 = \{a,d,f\}$};
\end{tikzpicture}

\begin{tikzpicture}[scale=.4]
\draw(0,0)--(8,8)--(16,0);
\draw(1,1)--(2,0);
\draw (2.5,2.5)--(5,0);
\draw (4,1)--(3,0);
\draw (4,4)--(8,0);
\draw (7,1)--(6,0);
\draw (7,7)--(14,0);
\draw (11.5,2.5) -- (9,0);
\draw (10,1)--(11,0);
\draw (13,1)--(12,0);
\draw[fill] (8,8) circle [radius = .1];
\node[left] at (8,8) {$a$};
\draw[fill] (7,7) circle [radius = .1];
\node[left] at (7,7) {$b$};
\draw[fill] (4,4) circle [radius = .1];
\node[left] at (4,4) {$c$};
\draw[fill] (2.5,2.5) circle [radius = .1];
\node[left] at (2.5,2.5) {$d$};
\draw[fill] (7,1) circle [radius = .1];
\node[left] at (7,1) {$e$};
\draw[fill] (11.5,2.5) circle [radius = .1];
\node[right] at (11.7,2.5) {$f$};
\draw[fill] (1,1) circle [radius = .1];
\node[left] at (1,1) {$g$};
\draw[fill] (4,1) circle [radius = .1];
\node[right] at (4,1) {$h$};
\draw[fill] (10,1) circle [radius = .1];
\node[left] at (10,1) {$i$};
\draw[fill] (13,1) circle [radius = .1];
\node[right] at (13.2,1) {$j$};
\draw[ultra thick] (2,0)--(1,1)--(4,4)--(7,1)--(6,0);
\draw[ultra thick, dotted] (9,0)--(10,1)--(11,0);
\node[below, align=center] at (8,0) {A path system realizing $\bar{V_1} = \{c,i\}$\\
 where $[\bar{\P_1}]$ is $c$-maintaining};
\end{tikzpicture}$\quad$\begin{tikzpicture}[scale=.4]
\draw(0,0)--(8,8)--(16,0);
\draw(1,1)--(2,0);
\draw (2.5,2.5)--(5,0);
\draw (4,1)--(3,0);
\draw (4,4)--(8,0);
\draw (7,1)--(6,0);
\draw (7,7)--(14,0);
\draw (11.5,2.5) -- (9,0);
\draw (10,1)--(11,0);
\draw (13,1)--(12,0);
\draw[fill] (8,8) circle [radius = .1];
\node[left] at (8,8) {$a$};
\draw[fill] (7,7) circle [radius = .1];
\node[left] at (7,7) {$b$};
\draw[fill] (4,4) circle [radius = .1];
\node[left] at (4,4) {$c$};
\draw[fill] (2.5,2.5) circle [radius = .1];
\node[left] at (2.5,2.5) {$d$};
\draw[fill] (7,1) circle [radius = .1];
\node[left] at (7,1) {$e$};
\draw[fill] (11.5,2.5) circle [radius = .1];
\node[right] at (11.7,2.5) {$f$};
\draw[fill] (1,1) circle [radius = .1];
\node[left] at (1,1) {$g$};
\draw[fill] (4,1) circle [radius = .1];
\node[right] at (4,1) {$h$};
\draw[fill] (10,1) circle [radius = .1];
\node[left] at (10,1) {$i$};
\draw[fill] (13,1) circle [radius = .1];
\node[right] at (13.2,1) {$j$};
\draw[ultra thick] (0,0)--(1,1)--(2,0);
\draw[ultra thick] (5,0)--(2.5,2.5)--(4,4)--(8,0);
\draw[ultra thick, dashed] (12,0)--(13,1)--(14,0);
\node[below, align=center] at (8,0) {A path system realizing \\
$\bar{V_2} = \{c,g,j\}$ where $[\bar{\P_2}]$ is $c$-maintaining};
\end{tikzpicture}$\quad$\begin{tikzpicture}[scale=.4]
\draw(0,0)--(8,8)--(16,0);
\draw(1,1)--(2,0);
\draw (2.5,2.5)--(5,0);
\draw (4,1)--(3,0);
\draw (4,4)--(8,0);
\draw (7,1)--(6,0);
\draw (7,7)--(14,0);
\draw (11.5,2.5) -- (9,0);
\draw (10,1)--(11,0);
\draw (13,1)--(12,0);
\draw[fill] (8,8) circle [radius = .1];
\node[left] at (8,8) {$a$};
\draw[fill] (7,7) circle [radius = .1];
\node[left] at (7,7) {$b$};
\draw[fill] (4,4) circle [radius = .1];
\node[left] at (4,4) {$c$};
\draw[fill] (2.5,2.5) circle [radius = .1];
\node[left] at (2.5,2.5) {$d$};
\draw[fill] (7,1) circle [radius = .1];
\node[left] at (7,1) {$e$};
\draw[fill] (11.5,2.5) circle [radius = .1];
\node[right] at (11.7,2.5) {$f$};
\draw[fill] (1,1) circle [radius = .1];
\node[left] at (1,1) {$g$};
\draw[fill] (4,1) circle [radius = .1];
\node[right] at (4,1) {$h$};
\draw[fill] (10,1) circle [radius = .1];
\node[left] at (10,1) {$i$};
\draw[fill] (13,1) circle [radius = .1];
\node[right] at (13.2,1) {$j$};
\draw[ultra thick] (2,0)--(1,1)--(2.5,2.5)--(4,1)--(3,0);
\draw[ultra thick] (8,0)--(4,4)--(8,8)--(16,0);
\draw[ultra thick, dashed] (11,0)--(10,1)--(11.5,2.5)--(13,1)--(12,0);
\node[below, align=center] at (8,0) {A path system realizing \\
$\bar{V_3} = \{a,d,f\}$ };
\end{tikzpicture}
\end{center}
\caption{Proof of Lemma \ref{lem:minreps}. $\{ [\Q_1] , \dots , [\Q_m]\}$ contains fewer $c$-nonmaintaining vertices than $\{[\P_1] , \dots , [\P_m]\}$. The first row of trees are path systems that realize $\bv = [\P_1] + [\P_2] + [\P_3] \in 3R_T$. The second row of trees are path systems that realize $\bu = [\Q_1] + [\Q_2] + [\Q_3] \in 3R_T$, which satisfies the assumptions of the lemma. The third row of trees are a new set of path systems that realize $\bv$ using fewer $c$-nonmaintaining vertices, which we obtained by applying the procedure discussed in the proof of the lemma.} \label{fig:minreps2}
\end{sidewaysfigure}

Let $\bar{V_i} = (V_i - V_i^f) \cup U_i^f$ for all $i$. We claim that each of the $\bar{V_i}$ are valid top-sets, and that the collection of all $\bar{V_i}$ has the same number of $c$-nonmaintaining vertices as the $U_i$, and the same number of $b$-nonmaintaining vertices as the $V_i$.

First, let $i \leq r'$. Then since $[\P_i]$ and $[\Q_i]$ are both $c$-nonmaintaining, $f$ is blocked in both $V_i$ and $U_i$.  Therefore, all paths in $\P_i$ and $\Q_i$ that intersect the $f$-subtree are contained entirely within the $f$-subtree. So $\bar{\P_i} = (\P_i - \P_i^f) \cup \Q_i^f$ is a path system that realizes $\bar{V_i}$, as needed.

Next, let $r' < i \leq s$. Then $[\Q_i]$ is either $c$-maintaining or has $b$-coordinate equal to 1. In either case, $f$ is not blocked in $U_i$. So there exists a path $Q$ from from $f$ to a leaf descended from $f$ with no node along $Q$ in $U_i$.

Consider the case when $b \in V_i$. Let $P \in \P_i$ with $b$ as its top-most node, and let $\hat{P}$ be the path from $b$ to a leaf below $c$ that is contained in $P$. Let $P' = \hat{P} \cup \{bf\} \cup Q$. Then
\[
\bar{\P_i} = (\P_i - (\{ P\} \cup \P_i^f)) \cup \{P'\} \cup \Q_i^f
\]
realizes $\bar{V_i}$.

Now suppose that $c \in V_i$. If there does not exist $P \in \P_i$ with a node above $b$ as its top-most node that passes through the $f$-subtree, then all paths in $\P_i$ that intersect the $f$-subtree are contained in the $f$-subtree. So $\bar{\P_i} = (\P_i - \P_i^d) \cup \Q_i^d$ is a path system that realizes $\bar{V_i}$.

If there does exist $P \in \P_i$ with top-most node above $b$ that passes through the $f$-subtree, let $\hat{P}$ denote $P$ without the portion of $P$ that lies in the $f$-subtree. Let $P' = \hat{P} \cup Q$. Then
\[
\bar{\P_i} = (\P_i - (\{P\} \cup \P_i^f))\cup \{P'\} \cup \Q_i^f
\]
is a path system that realizes $\bar{V_i}$.

Note that in all cases when $r' < i \leq s$, $f$ is not blocked in $\bar{V_i}$. Since $r'<r$, this means that there are fewer $c$-nonmaintaining vertices in  $\{[\bar{\P_1}], \dots , [\bar{\P_s}] \}$ than in $\{[\P_1], \dots , [\P_s]\}$. Furthermore, since the paths in the $d$-subtrees remain unchanged, this operation cannot create new $b$-nonmaintaining vertices.

Finally, let $i>s$. Then $b,c \not\in V_i, U_i$. Since $f$ is blocked in every $V_i$, all paths in $\P_i$ that intersect the $f$-subtree are contained entirely in the $f$-subtree. So $\bar{\P} = (\P_i - \P_i^f) \cup \Q_i^f$ is a path system that realizes $\bar{V_i}$.

Since
\[
\sum_{i=1}^m [\P_i]^f= \sum_{i=1}^m [\Q_i]^f,
\]
and since $[\bar{\P_i}]^f = [\Q_i]^f$ for all $i$, $[\bar{\P_1}] + \dots + [\bar{\P_m}]$ is a representation of $\bv$ using fewer nonmaintaining vertices than $[\P_1] + \dots + [\P_m]$. So $[\P_1] + \dots + [\P_m]$ is not minimal.  An example of the operation used to obtain $[\bar{\P_1}],\dots ,[\bar{\P_m}]$ is illustrated in Figure \ref{fig:minreps2}.
\end{proof}

\begin{cor}
The map $\phi^{T,T'}_m$ is well-defined.
\end{cor}

\begin{proof}
It suffices to show that all minimal representations of $v \in mR_T \cap \Z^{n-1}$ have the same number of $b$- and $c$-nonmaintaining vertices. This follows from Lemma \ref{lem:minreps}.
\end{proof}

\begin{cor}\label{cor:latticepointbijection}
The map $\phi^{T,T'}_m$ is a bijection.
\end{cor}

\begin{proof}
It suffices to show that $\phi^{T',T}_m$ is the inverse map of $\phi^{T,T'}_m$. Suppose that it is not. Then there exists some $\bv \in mR_T \cap \Z^{n-1}$ such that $\bv = [\P_1] + \dots + [\P_m]$ is a minimal representation, but $\phi^{T,T'}([\P_1]) + \dots + \phi^{T,T'}([\P_m]) = \phi^{T,T'}_m(\bv)$ is not a minimal representation of $\phi^{T,T'}_m(\bv)$.

Let $[\Q_1] + \dots + [\Q_m]= \phi^{T,T'}_m(\bv)$ be minimal. Consider the image $\phi^{T',T}_m (\phi^{T,T'}_m(\bv)) = \phi^{T',T}([\Q_1]) + \dots + \phi^{T',T}([\Q_m])$. The set $\{ \phi^{T',T}([\Q_1]), \dots , \phi^{T',T}([\Q_m]) \}$ contains fewer nonmaintaining vertices than $\{ [\P_1], \dots , [\P_m] \}$ because $\phi^{T',T}$ maps maintaining vertices to maintaining vertices by the proof of Proposition \ref{prop:nonmaintainingbij}. The set $\{ \phi^{T',T}([\Q_1]), \dots , \phi^{T',T}([\Q_m]) \}$ also satisfies all of the assumptions of Lemma \ref{lem:minreps}. So, $[\P_1] + \dots + [\P_m]$ was not actually a minimal representation of $\bv$ and we have reached a contradiction.
\end{proof}

\begin{thm}\label{thm:hilbertseries}
For all rooted binary trees $T$ with $n$ leaves, the Hilbert series of $I_T$ is equal to the Hilbert series of $I_{C_n}$.
\end{thm}

\begin{proof}
Every rooted binary tree can be obtained from the caterpillar tree by a finite sequence of rotations. So, it follows from Corollary \ref{cor:latticepointbijection} that the number of lattice points in the $m$\textsuperscript{th} dilates of $R_T$ is equal to that of $R_{C_n}$ for all trees $T$ with $n$ leaves. So, the Ehrhart polynomials and hence, the Ehrhart series of $R_T$ and $R_{C_n}$ are equal. The Ehrhart series of $R_T$ is equal to the Hilbert series of $I_T$.
\end{proof}

\begin{proof}[Proof of Theorem \ref{thm:volumemain}]
The leading coefficient of the Ehrhart polynomial of a polytope is the (unnormalized) volume of the polytope. We have shown the equality of the Ehrhart polynomials of $R_{C_n}$ and $R_T$ for any $n$-leaf tree $T$. So, $R_{C_n}$ and $R_T$ have the same normalized volumes. This is the $(n-1)$st Euler zig-zag number by Corollary \ref{cor:cattreevolume}.
\end{proof}


\section*{Acknowledgments}

Jane Coons was partially supported by the US National Science Foundation (DGE-1746939)
and by the David and Lucille Packard Foundation.
Seth Sullivant was partially supported by the US National Science Foundation (DMS 1615660) and
by the David and Lucille Packard Foundation.

\bibliographystyle{acm}
\bibliography{CFNMC_bib}

\end{document}